\documentclass{article}
\usepackage[utf8]{inputenc}

\usepackage{geometry}
\usepackage{amsmath}
\usepackage{comment}
\usepackage{leftidx}
\usepackage{multicol}
\usepackage{amssymb}
\usepackage{stmaryrd}
\usepackage{accents}
\usepackage{mathrsfs}
\usepackage{xfrac}
\usepackage{amsthm}
\usepackage{xcolor}
\usepackage{sectsty}
\usepackage{relsize}
\usepackage{float}

\usepackage{indentfirst}
\usepackage{tikz}
\tikzset{shorten <>/.style={shorten >=#1,shorten <=#1}}

\usepackage{tikz-cd}
\newcounter{nodemaker}
\tikzset{Rightarrow/.style={double equal sign distance,>={Implies},->},
triple/.style={-,preaction={draw,Rightarrow}},
quadruple/.style={preaction={draw,Rightarrow,shorten >=0pt},shorten >=1pt,-,double,double
distance=0.2pt}}

\tikzset{%
    symbol/.style={%
        draw=none,
        every to/.append style={%
            edge node={node [sloped, allow upside down, auto=false]{$#1$}}}
    }
}
 \usepackage{quiver}
\makeatletter
\newcommand{\bigdoublevee}{\big@doubleop{\bigvee}}
\newcommand{\bigdoublewedge}{\big@doubleop{\bigwedge}}
\newcommand{\big@doubleop}[1]{%
  \DOTSB\mathop{\mathpalette\big@doubleop@aux{#1}}\slimits@
}

\newcommand\big@doubleop@aux[2]{%
  \sbox\z@{$\m@th#1#2$}%
  \makebox[1.35\wd\z@][s]{$\m@th#1#2\hss#2$}%
}
\makeatother

\makeatletter
\newcommand*{\doublerightarrow}[2]{\mathrel{
  \settowidth{\@tempdima}{$\scriptstyle#1$}
  \settowidth{\@tempdimb}{$\scriptstyle#2$}
  \ifdim\@tempdimb>\@tempdima \@tempdima=\@tempdimb\fi
  \mathop{\vcenter{
    \offinterlineskip\ialign{\hbox to\dimexpr\@tempdima+1em{##}\cr
    \rightarrowfill\cr\noalign{\kern.5ex}
    \rightarrowfill\cr}}}\limits^{\!#1}_{\!#2}}}
\newcommand*{\triplerightarrow}[1]{\mathrel{
  \settowidth{\@tempdima}{$\scriptstyle#1$}
  \mathop{\vcenter{
    \offinterlineskip\ialign{\hbox to\dimexpr\@tempdima+1em{##}\cr
    \rightarrowfill\cr\noalign{\kern.5ex}
    \rightarrowfill\cr\noalign{\kern.5ex}
    \rightarrowfill\cr}}}\limits^{\!#1}}}
\makeatother

\newcommand{\hirayo}{\text{\usefont{U}{min}{m}{n}\symbol{'207}}}
\DeclareFontFamily{U}{min}{}
\DeclareFontShape{U}{min}{m}{n}{<-> udmj30}{}

\DeclareUnicodeCharacter{3088}{\yo}

\usepackage{geometry}
\newcommand{\Spec}
 {{\bf Spec}}
 
 \newcommand{\Alg}
 {{\rm Alg}}
 
 \newcommand{\Lex}
 {{\bf Lex}}
 
 \newcommand{\Et}
 {{\bf Et}}
 
\newcommand{\Loc}
 {{\bf Loc}}
 
  \newcommand{\Diag}
 {{\bf Diag}}
 
 \newcommand{\GTop}
 {{\bf GTop}}
 
 \newcommand{\Geom}
 {{\bf Geom}}
 
 \newcommand{\Cat}
 {{\bf Cat}}

\newcommand{\cod}
 {{\rm cod}}

\newcommand{\comp}
 {\circ}

\newcommand{\Cont}
 {{\bf Cont}}

\newcommand{\dom}
 {{\rm dom}}
 
\newcommand{\fp}
{{\rm fp }}

\newcommand{\li}
{{\textup{lim } }}

\newcommand{\oplaxlim}
{{\textup{oplaxlim } }}

\newcommand{\oplaxcolim}
{{\textup{oplaxcolim } }}

\newcommand{\laxcolim}
{{\textup{laxcolim } }}

\newcommand{\bicolim}
{{\textup{bicolim } }}

\newcommand{\bilim}
{{\textup{bilim } }}

\newcommand{\pscolim}
{{\textup{pscolim } }}

\newcommand{\pslim}
{{\textup{pslim } }}

\newcommand{\lan}
{{\textup{lan } }}

\newcommand{\ran}
{{\textup{ran } }}

\newcommand{\colim}
{{\textup{colim}}}

\newcommand{\coeq}
{{\textup{coeq}}}

\newcommand{\eq}
{{\textup{eq}}}

\newcommand{\comma}[2]
{\mbox{$(#1\!\downarrow\!#2)$}}

\newcommand{\empstg}
 {[\,]}

\newcommand{\epi}
 {\twoheadrightarrow}

\newcommand{\hy}
 {\mbox{-}}

\newcommand{\im}
 {{\rm im}}

\newcommand{\imp}
 {\!\Rightarrow\!}

\newcommand{\Ind}
 {{\rm Ind}}
 
 \newcommand{\Pro}
 {{\rm Pro}}

\newcommand{\mono}
 {\rightarrowtail}

\newcommand{\ob}
 {{\rm ob}}
 
 \newcommand{\can}
 {{\rm can}}
 
 \newcommand{\Hom}
 {{\rm Hom}}

\newcommand{\op}
 {^{\rm op}}
 
 \newcommand{\pt}
 {{\bf pt}}

\newcommand{\Set}
 {{\bf Set }}

\newcommand{\Sh}
 {{\bf Sh}}

\newcommand{\sh}
 {{\bf sh}}

\newcommand{\Sub}
 {{\rm Sub}}

\newcommand{\Flat}
 {{\bf Flat}}
\newcommand{\biIns}
{{\textup{biIns } }}

\newcommand{\biInv}
{{\textup{biInv } }}

\newcommand{\biend}
{{\textup{bi}\int }}

\newcommand{\biLan}
{{\textup{biLan } }}

 \newcommand{\Cart}
 {{\bf Cart}}

  \newcommand{\ps}
 {{\bf ps}}

 \newcommand{\St}
 {{\bf St}}

\sectionfont{\large}
\subsectionfont{\normalsize}

\usepackage{hyperref}
\hypersetup{
    colorlinks=true,
    linkcolor=blue,
    filecolor=magenta,      
    urlcolor=cyan,
}

\usepackage{cleveref}

\newtheorem{theorem}{Theorem}[subsection]
\theoremstyle{proposition}
\newtheorem{proposition}[theorem]{Proposition}
\newtheorem{corollary}[theorem]{Corollary}
\newtheorem{corollary'}[theorem]{Corollary}
\newtheorem{lemma}[theorem]{Lemma}
\theoremstyle{definition}
\newtheorem{definition}[theorem]{Definition}

\theoremstyle{definition}
\newtheorem{remark}[theorem]{Remark}
\theoremstyle{definition}
\newtheorem{division}[theorem]{}
\theoremstyle{conjecture}

\usepackage[backend=biber,
sorting=nty
]{biblatex}
\nocite{*}

\title{Local right biadjoints, bistable pseudofunctors and 2-geometries for Grothendieck topoi}
\author{Axel Osmond}
\date{August 2021}

\begin{document}

\maketitle

\begin{abstract}
    We provide bicategorical analogs of several aspects of the notion of geometry in the sense of the theory of spectrum. We first introduce a notion of \emph{local right biadjoint}, and prove it to be equivalent to a notion of \emph{bistable} pseudofunctor, categorifying an analog 1-categorical result. We also describe further laxness conditions, giving some properties of the already known \emph{lax familial} pseudofunctors. We also describe 2-dimensional analogs of orthogonality and factorization systems, and use them to construct examples of bistable pseudofunctors through inclusion of \emph{left objects and left maps}. We apply the latter construction to several examples of factorization systems for geometric morphisms to produce geometry-like situations for Grothendieck topoi. In particular we prove a new (terminally connected, pro-etale) factorization theorem for geometric morphisms, which corresponds to a certain \emph{local geometry} involved in the general construction of spectra. 
\end{abstract}

\tableofcontents

\section*{Introduction}

Here we provide a 2-dimensional framework to formalize some observations that arose from the general theory of Spectrum. \\

The construction of the spectrum subsumes either algebraic geometry or Stone-like dualities into a generic process where one turns a categorical, algebraic-like situation into a geometric construction in order to correct some defect of universality in a universal property. Although several alternative ways exist to understand this process, one of its main explanation is the following: from a \emph{right multiadjoint} $U : \mathcal{A} \rightarrow \mathcal{B}$ into a locally finitely presentable category $ \mathcal{B}$ satisfying an additional ``approximability" condition, one can construct a functor $ \Spec $ associating to each object of $\mathcal{B}$ a space - its \emph{spectrum} - whose points correspond to local units under this object, and correcting the multiadjunction into a global adjunction; in this process, one deploys in fact hidden geometric behaviors which can be visualized by formal duality in the opposite category of $\mathcal{B}$. Though the details and the exact property of this construction are not the topic of this work, we would like to point out an interesting phenomenon which requires further precision. \\

In situations as evoked above, which we may call a \emph{geometry}, the objects of $\mathcal{A}$, said \emph{local objects}, will behave as duals of \emph{focal spaces}, that are spaces with a focal point, while a class of \emph{etale maps}, defined through some kind of strong left-orthogonality condition relatively to maps in the range of $U$, will behave as dual of inclusion of saturated compacts, amongst which one will distinguish basic compact open ones used to produce basis for spectral topologies on spectra of objects. \\

Though the spectrum constructed from those data is generally seen as a topological space, it defines also a Grothendieck topos; in fact, one can characterize the geometric properties of the spectrum of the different actors of a geometry as the following:\begin{itemize}
    \item the spectrum of a local object is always a \emph{local topos} 
    \item the spectrum of a \emph{local map} - that is, a map in the range of $U$ - is always a \emph{terminally connected} geometric morphism
    \item the spectrum of an etale map is always a \emph{etale} geometric morphism
\end{itemize}

Those observations led to the following interrogation: is there a way to define a notion of 2-dimensional geometry on the bicategory of Grothendieck topoi, involving local topoi, terminally connected and etale geometric morphisms, such that any ordinary 1-categorical geometry would be equipped, through its spectral functor, with a canonical ``morphism of geometries" above this one ? More generally, what should a general notion of 2-geometry be, and what aspects of the 1-categorical version should be categorified to formalize our observation ? If such a notion makes sense, are there other instances of 2-geometries on Grothendieck topoi, and what kind of alternative algebraic geometry would they convey ? \\

The purpose of this work is to provide a formal 2-categorical framework to answer those questions. Categorifying the components of the 1-dimensional notion of geometry as right multi-adjointness or its orthogonality and factorization aspects, we propose a particular bicategorical notion of \emph{bifactorization 2-geometry} suited for Grothendieck topoi. We should warn the reader that the notion of 2-geometry investigated in this paper is somewhat divergent in several aspects from the 1-dimensional notion of geometry, ignoring for instance the accessibility aspects - no categorification of which will ever apply to the (opposite) bicategory of Grothendieck topoi, in the same way as frames are not an accessible category - or also the approximability condition which is related in some way to the small object argument. A distinct, more faithful 2-categorification of the construction of spectra, involving 2-dimensional topos theory and accessibility, will also be proposed in an forthcoming work and applied to syntax-semantics adjunction, but while using partially the result of this paper, the strategy will be quite different. \\

We first investigate the 2-dimensional analogs of multi-adjointness and its generalization. Multiversal properties are situations where a universal property is jointly assumed by a family of solutions rather than a unique one. In particular, \emph{right multi-adjoints} are functors equipping each object of their codomain with a (small) cone of \emph{local units}, satisfying the property that any map from this object to the functor factorizes uniquely through one of those units. 
More generally, local right adjoint are functors whose each restriction at slices has a left adjoints, while \emph{stable functors} are defined through a unique factorization property involving a notion of \emph{candidates} (also called \emph{generic morphisms}), which are defined through a left orthogonality condition relatively to morphisms in the range of the functor. It can be shown, see \cite{osmond2020diersI} that local right adjoints and stable functors are actually the same thing, while right multi-adjoints are the local right adjoints satisfying the solution set condition; in particular local units and generic morphisms are two different ways to see the same kind of objects. \\

Here, we will present 2-dimensional analogs of local right adjoints and stable functors, namely \emph{local right biadjoints} and \emph{bistable pseudofunctors}. The latter are related to \cite{weber2004generic} notion of \emph{familial 2-functor}, while a lax version has been investigated in \cite{walker2020lax}. We first describe the general properties of local right biadjoints, as the 2-dimensional Beck-Chevalley condition relating the slice-wise left adjoint at \cref{BC 2dim} or the factorization data associated to them, and prove at \cref{bistable = local right biadjoint} that local right bi-adjoints and bistable pseudofunctors coincide. We will also discuss corresponding results for lax-familialness. \\ 

In a second time, we show how to construct a bistable pseudofunctor from factorization data: in a bicategory equipped with a \emph{bifactorization system} $(\mathcal{L}, \mathcal{R})$ in a bi-version of \cite{DUPONT200365} together with a biterminal object $ 1$, we prove that 
the \emph{left objects} - those whose terminal map is in $\mathcal{L}$ - form altogether with left maps between them a (non-full) sub-bicategory whose inclusion is opposite to a bistable pseudofunctor, see \cref{Bistable inclusion of left objects}. Moreover, this result can be relativized above any object, or also refined by considering objects whose terminal map is in a subclass $ \mathcal{L}' \subseteq \mathcal{L}$ satisfying some specific cancellation property relatively to $\mathcal{L}$, see \cref{bistable for specified left objects}.\\

We then discuss how local right bi-adjoints also enjoy a geometric interpretation motivating our conventions in the aforementioned construction, in particular how in this context left objects can be seen as ``focal spaces", rights maps as ``etale maps" and local units as inclusions of focal components.\\  

In the last section, we apply this construction to several instances of factorization systems for geometric morphisms. In all those examples, the left maps express some level of connectedness as opposed to the right class which express some level of truncation; the key idea is that geometric properties of Grothendieck topoi are encoded in their terminal map, so that left objects in those examples correspond to specific, well-known classes of Grothendieck topoi. A first well-behaved example is the \emph{hyperconnected-localic} factorization. We also discuss three related factorizations: first we recall elements of the \emph{connected-etale} factorization for locally connected geometric morphisms, and a recent generalization of the latter into the \emph{terminally connected-etale} factorization for essential geometric morphisms; then we prove at \cref{terminally connected-proetale factorization} a further generalization of this to arbitrary geometric morphisms, still with terminally connected geometric morphisms on the left, but now \emph{pro-etale geometric} morphisms on the right; this latter factorization, related to \emph{Grothendieck-Verdier localization}, recovers exactly the one involved in the classical construction of the spectrum when combined with the \emph{local geometric morphisms} as specified left maps. \\

We should also make mention of the fact that we chose to develop the following theory for bicategory and pseudofunctors rather than in stricter analog, both for greater generality but also because our leading examples, Grothendieck topoi, are more a bicategory than a 2-category; however we will make no mention of the bicategorical structure, only of the canonical 2-cells provided by pseudofunctoriality. 

\section{Two-dimensional stability and local biadjunctions}

We here introduce a bi-version of the notion of local right adjoint and stable functors. We first describe \emph{local right biadjoint}, those pseudofunctors whose restrictions at pseudoslices have a left biadjoint. As in the 1-dimensional case, the pseudoslice-wise left adjoints are related by an automatic Beck-Chevalley condition due to the universal property of the local units, which expresses that postcomposing an arrow with an arrow in the range of the pseudofunctor does not modify the local unit. In a parallel way we introduce \emph{bistable pseudofunctors}, which involve factorization of 1-cells toward objects in their range through a class of \emph{generic} 1-cells defined through a strong orthogonality condition (analog of the \emph{diagonally universal morphisms of Diers}). We then prove those two notions to be equivalent. Though the idea behind this is essentially the same as in the 1-dimensional setting - that is, local units are the generic 1-cells - it involves quite technical aspects and deserves careful manipulation of the 2-dimensional data, for which we believe it is worth giving a detailed account of the intricate architecture encoded in those notions. 

\subsection{Local right biadjoint}

As in the 1-dimensional case, this notion can be alternatively reformulated in term of local adjunction.

\begin{definition}
A functor $ U : \mathcal{A} \rightarrow \mathcal{B}$ is a \emph{local right biadjoint} if for each object $ A$ in $ \mathcal{A}$, $U$ induces at the pseudoslice a biadjunction 
\[ 
\begin{tikzcd}
\mathcal{A}/A \arrow[rr, "U_A"', bend right] \arrow[rr, "\perp", phantom] &  & \mathcal{B}/U(A) \arrow[ll, "L_A"', bend right]
\end{tikzcd} \]
\end{definition}

\begin{remark}
Let us unfold this condition. In each $ A$ we have a local biadjunction $L_A \dashv U_A$, whose unit and counit, for each choice of $ f: B \rightarrow U(A)$ in $ \mathcal{B}/U(A)$ and $ u : A' \rightarrow A$ in $\mathcal{A}/A$, shall be denoted respectively as
\[ 
\begin{tikzcd}
B \arrow[rr, "f"] \arrow[rd, "h^A_f"'] & {} \arrow[d, "\eta^A_f \atop \simeq", phantom] & U(A) \\
                                       & U(A_f) \arrow[ru, "U_AL_A(f)"']                &     
\end{tikzcd} \quad 
\begin{tikzcd}
A_{U_A(u)} \arrow[rd, "L_AU_A(u)"'] \arrow[rr, "e^A_u"] & {} \arrow[d, "\epsilon^A_u \atop \simeq", phantom] & A' \arrow[ld, "u"] \\
                                                        & A                                                  &                   
\end{tikzcd} \]

Moreover, the unit and counit are pseudonatural in the following sense. For an arrow $ (g,\alpha) : f_1 \rightarrow f_2$ in $\mathcal{B}/U(A)$ as on the right, with its image $ L_{A}(g,\alpha) = (L_A(g), L_A(\alpha)) : L_A(f_1) \rightarrow L_A(f_2)$ as on the right
\[\begin{tikzcd}
	{B_1} && {B_2} \\
	& {U(A)}
	\arrow[""{name=0, anchor=center, inner sep=0}, "g", from=1-1, to=1-3]
	\arrow["{f_2}", from=1-3, to=2-2]
	\arrow["{f_1}"', from=1-1, to=2-2]
	\arrow["{\alpha \atop \simeq}"{description}, Rightarrow, draw=none, from=0, to=2-2]
\end{tikzcd} \quad \quad 
\begin{tikzcd}
	{A_{f_1}} && {A_{f_2}} \\
	& A
	\arrow[""{name=0, anchor=center, inner sep=0}, "{L_A(g)}", from=1-1, to=1-3]
	\arrow["{L_A(f_2)}", from=1-3, to=2-2]
	\arrow["{L_A(f_1)}"', from=1-1, to=2-2]
	\arrow["{L_A(\alpha) \atop \simeq}"{description}, Rightarrow, draw=none, from=0, to=2-2]
\end{tikzcd}\]
then we have an invertible 2-cell $ \eta^A_{(g,\alpha)}$
\[\begin{tikzcd}
	{B_1} && {B_2} \\
	{U(A_{f_1})} && {U(A_{f_2})}
	\arrow[""{name=0, anchor=center, inner sep=0}, "{U(L_A(g))}"', from=2-1, to=2-3]
	\arrow["{h^A_{f_1}}"', from=1-1, to=2-1]
	\arrow[""{name=1, anchor=center, inner sep=0}, "g", from=1-1, to=1-3]
	\arrow["{h^A_{f_2}}", from=1-3, to=2-3]
	\arrow["{\eta^A_{(g,\alpha)} \atop \simeq}"{description}, Rightarrow, draw=none, from=1, to=0]
\end{tikzcd}\]
such that $ \alpha$ decomposes as the following 2-cell
\[\begin{tikzcd}[sep=large]
	{B_1} &&&& {B_2} \\
	& {U(A_{f_1})} && {U(A_{f_2})} \\
	&& {U(A)}
	\arrow[""{name=0, anchor=center, inner sep=0}, "{U(L_A(g))}"{description}, from=2-2, to=2-4]
	\arrow["{U(L_A(f_2))}"{description}, from=2-4, to=3-3]
	\arrow["{U(L_A(f_1))}"{description}, from=2-2, to=3-3]
	\arrow["{h^A_{f_1}}"{description}, from=1-1, to=2-2]
	\arrow[""{name=1, anchor=center, inner sep=0}, "g", from=1-1, to=1-5]
	\arrow["{h^A_{f_2}}"{description}, from=1-5, to=2-4]
	\arrow[""{name=2, anchor=center, inner sep=0}, "{f_1}"', curve={height=40pt}, from=1-1, to=3-3]
	\arrow[""{name=3, anchor=center, inner sep=0}, "{f_2}", curve={height=-40pt}, from=1-5, to=3-3]
	\arrow["{U(L_A(\alpha)) \atop \simeq}"{description, pos=0.33}, Rightarrow, draw=none, from=0, to=3-3]
	\arrow["{\eta^A_{(g,\alpha)} \atop \simeq}"{description}, Rightarrow, draw=none, from=1, to=0]
	\arrow["{\eta^A_{f_1} \atop \simeq}"{description, pos=0.4}, Rightarrow, draw=none, from=2-2, to=2]
	\arrow["{\eta^A_{f_2} \atop \simeq}"{description, pos=0.4}, Rightarrow, draw=none, from=2-4, to=3]
\end{tikzcd}\]
Moreover, for a 2-cell $ \phi : (g, \alpha) \Rightarrow (g', \alpha')$ in $ \mathcal{B}/U(A)$, the composite 2-cell above factorizes as the 2-cell 
\[\begin{tikzcd}[sep=large]
	{B_1} &&&& {B_2} \\
	& {U(A_{f_1})} && {U(A_{f_2})} \\
	&& {U(A)}
	\arrow[""{name=0, anchor=center, inner sep=0}, "{U(L_A(g'))}"{description}, from=2-2, to=2-4]
	\arrow["{U(L_A(f_2))}"{description}, from=2-4, to=3-3]
	\arrow["{U(L_A(f_1))}"{description}, from=2-2, to=3-3]
	\arrow["{h^A_{f_1}}"{description}, from=1-1, to=2-2]
	\arrow[""{name=1, anchor=center, inner sep=0}, "g", curve={height=-25pt}, from=1-1, to=1-5]
	\arrow["{h^A_{f_2}}"{description}, from=1-5, to=2-4]
	\arrow[""{name=2, anchor=center, inner sep=0}, "{f_1}"', curve={height=40pt}, from=1-1, to=3-3]
	\arrow[""{name=3, anchor=center, inner sep=0}, "{f_2}", curve={height=-40pt}, from=1-5, to=3-3]
	\arrow[""{name=4, anchor=center, inner sep=0}, "{g'}"{description}, from=1-1, to=1-5]
	\arrow["{U(L_A(\alpha')) \atop \simeq}"{description, pos=0.33}, Rightarrow, draw=none, from=0, to=3-3]
	\arrow["{\eta^A_{f_1} \atop \simeq}"{description, pos=0.4}, Rightarrow, draw=none, from=2-2, to=2]
	\arrow["{\eta^A_{f_2} \atop \simeq}"{description, pos=0.4}, Rightarrow, draw=none, from=2-4, to=3]
	\arrow["\phi", shorten <=2pt, shorten >=2pt, Rightarrow, from=1, to=4]
	\arrow["{\eta^A_{(g',\alpha')} \atop \simeq}"{description}, Rightarrow, draw=none, from=4, to=0]
\end{tikzcd}\]
(and similarly for the counit, though we wont have any use of it later).\\

Now the condition of biajdunction states that for each $ f : B \rightarrow U(A)$ in $\mathcal{B}/U(A)$ and $ u : A' \rightarrow A$ in $\mathcal{A}/A$, we get an adjunction between homcategories
\[ 
\begin{tikzcd}
{\mathcal{A}/A [L_A(f),u]} \arrow[rr, "{G^A_{f,u}}"' description, bend right=20] & \perp & {\mathcal{B}/U(A)[f,U_A(u)]} \arrow[ll, "{F^A_{f,u}}"' description, bend right=20]
\end{tikzcd} \]
where $ F^A_{f,u}$ sends $ (g,\alpha)$ on the pasting of its image along $ L_A$ with the counit of $u$, while $ G^A_{f,u}$ sends $(w, \omega) $ on the pasting of its image along $U_A$ with the unit of $f$ 
\[ 
\begin{tikzcd}[column sep=large, row sep=large]
A_f \arrow[r, "L_A(g)"] \arrow[rd, "L_A(f)"', "L_A(\alpha) \atop \simeq"{inner sep=1pt, near start}] & A_{U(u)} \arrow[r, "e^A_u"] \arrow[d, "L_A(U(u))" description] & A' \arrow[ld, "u", "\epsilon^A_u \atop \simeq"'{inner sep=3pt}] \\
                                              & A                                                              &                   
\end{tikzcd} \quad  
\begin{tikzcd}[column sep=large, row sep=large]
B \arrow[rd, "f"', "\eta^A_f \atop \simeq"{inner sep=3pt}] \arrow[r, "h^A_f"] & U(A_f) \arrow[d, "U_AL_A(f)" description] \arrow[r, "U(w)"] & U(A') \arrow[ld, "U(u)", "U(\omega) \atop \simeq"'{inner sep=3pt}] \\
                                      & U(A)                                                        &                         
\end{tikzcd} \]
Moreover, the unit and counit, denoted as $ \mathfrak{i}^A_{f,u} : 1 \Rightarrow G_{f,u}^AF_{f,u}^A$ and $ \mathfrak{e}^A_{f,u} : F_{f,u}^AG_{f,u}^A \Rightarrow 1$ in the following, consist respectively, for $ (g,\alpha) $ and $ (w,\omega) $ of 2-cells in $ \mathcal{B}/U(A)$ and $ \mathcal{A}/A$
\[ 
\begin{tikzcd}[sep=large]
B \arrow[r, "h^A_f" description] \arrow[rrd, "f"', bend right=20, " \eta^A_f \atop \simeq"{inner sep=3pt}] \arrow[rrrr, "g", bend left=20, " (\mathfrak{i}^A_{f,u})_{(g,\alpha)} \atop \simeq"'{inner sep=3pt}] & U(A_f) \arrow[rr, "U(L_A(g))"] \arrow[rd, "U_AL_A(f)" description] & {} \arrow[d, "U(L_A(\alpha)) \atop \simeq", phantom, near start] & U(A_{U(u)}) \arrow[r, "U(e^A_u)" description] \arrow[ld, "U_AL_A(U(u))" description] & U(A') \arrow[lld, "U(u)", bend left=15, "U(\epsilon^A_u) \atop \simeq"'{inner sep=3pt, near start}] \\
                                                                                            &                                                                    & U(A)                                                 &                                                                                      &                                     
\end{tikzcd} = \begin{tikzcd}[column sep=small]
B_1 \arrow[rd, "f_1"'] \arrow[rr, "g"] & {} \arrow[d, "\alpha \atop \simeq", phantom] & U(A') \arrow[ld, "f_2"] \\
                                       & U(A)                                         &                      
\end{tikzcd}   \]
(where the counit is left as an exercice for the reader), and those 2-cells must be invertible so the adjunction above is actually an equivalence between homcategories.
\end{remark}

\begin{remark}\label{expression of the BC mate}
For any arrow $ u : A_1 \rightarrow A_2$ in $ \mathcal{A}$, we have a pseudocommutative square expressing the pseudofunctoriality of $U$ trough a natural equivalence
\[ 
\begin{tikzcd}
\mathcal{A}/A_1 \arrow[rr, "U_{A_1}"] \arrow[d, "\mathcal{A}/u"'] \arrow[rrd, "\alpha^U_{u, L_{A_1}} \atop \simeq", phantom] &  & \mathcal{B}/U(A_1) \arrow[d, "\mathcal{B}/U(u)"] \\
\mathcal{A}/A_2 \arrow[rr, "U_{A_2}"']                                                           &  & \mathcal{B}/U(A_2) 
\end{tikzcd}   \]
where $ \alpha^U_u$ has as component the natural isomorphisms provided by pseudofunctoriality of $U$, defined on each $ v : A \rightarrow A_1$ as $ \alpha_u : U(u)U(v) \simeq U(uv)$. \\

Now this natural equivalence has a canonical mate defined as follows. From the biadjunction in $A$, we have in each arrow $ f: B \rightarrow U(A_1)$ an equivalence of homcategories
\[ \mathcal{B}/U(A_2)[U(u)f, U_{A_2}(uL_{A_1}(f)) ] \simeq \mathcal{A}/A_2[L_{A_2}(U(u)f), uL_{A_1}(f)] \]

In particular the arrow $ (h^{A_1}_f, U(u)*\eta^{A_1}_f \alpha^U_{u, L_{A_1}(f)}) : U(u)f \rightarrow U_{A_2}(uL_{A_1}(f))$ obtained as the pasting below 
\[\begin{tikzcd}[column sep=huge]
	B && {U(A_f)} \\
	& {U(A_1)} \\
	& {U(A_2)}
	\arrow[""{name=0, anchor=center, inner sep=0}, "{h^{A_1}_f}", from=1-1, to=1-3]
	\arrow["{UL_{A_1}(f)}"{description}, from=1-3, to=2-2]
	\arrow["f"', from=1-1, to=2-2]
	\arrow["{U(u)}"', from=2-2, to=3-2]
	\arrow[""{name=1, anchor=center, inner sep=0}, "{U(uL_{A_1}(f))}", curve={height=-20pt}, from=1-3, to=3-2]
	\arrow["{\eta^{A_1}_f \atop \simeq}"{description}, Rightarrow, draw=none, from=0, to=2-2]
	\arrow["{\alpha^U_{u, L_{A_1}(f)} \atop \simeq}"{description}, Rightarrow, draw=none, from=2-2, to=1]
\end{tikzcd} \]
is sent by the equivalence of homcategories to the following pasting
\[\begin{tikzcd}[sep=huge]
	{A_{U(u)f}} && {A_{U(uL_{A_1}(f))}} && {A_f} \\
	&& {A_2}
	\arrow["{L_{A_2}(h^{A_1}_f)}"{description}, from=1-1, to=1-3]
	\arrow["{e^{A_2}_{uL_{A_1}(f)}}"{description}, from=1-3, to=1-5]
	\arrow[""{name=0, anchor=center, inner sep=0}, "{L_{A_2}(U(u)f)}"', shift right=1, curve={height=12pt}, from=1-1, to=2-3]
	\arrow["{L_{A_2}(uL_{A_1}(f))}"{description, pos=0.7}, from=1-3, to=2-3]
	\arrow[""{name=1, anchor=center, inner sep=0}, "{uL_{A_1}(f)}", shift left=1, curve={height=-12pt}, from=1-5, to=2-3]
	\arrow[""{name=2, anchor=center, inner sep=0}, "{s^u_f}", curve={height=-24pt}, from=1-1, to=1-5]
	\arrow["{L_{A_2}(U(u)*\eta^{A_1}_f \alpha^U_{u, L_{A_1}(f)}) \atop \simeq}"{description}, shift left=2, Rightarrow, draw=none, from=0, to=1-3]
	\arrow["{\epsilon^{A_2}_{uL_{A_1}(f)} \atop \simeq}"{description, pos=0.4}, shift left=2, Rightarrow, draw=none, from=1-3, to=1]
	\arrow["{=}"{description}, Rightarrow, draw=none, from=2, to=1-3]
\end{tikzcd} \]
we shall denote as 
\[  \begin{tikzcd}
	{A_{U(u)f}} && {A_f} \\
	& {A_2}
	\arrow["{L_{A_2}(U(u)f)}"', from=1-1, to=2-2]
	\arrow["{uL_{A_1}(f)}", from=1-3, to=2-2]
	\arrow[""{name=0, anchor=center, inner sep=0}, "{s^u_f}", from=1-1, to=1-3]
	\arrow["{\sigma^u_f \atop \simeq}"{description}, Rightarrow, draw=none, from=0, to=2-2]
\end{tikzcd} \]
Moreover, the pointwisely invertible unit $ \frak{i}^{A_2}_{h^{A_2}_{U(u)f}, uL_{A_1}(f)}$ of the equivalence of homcategories above provides us at this 2-cell with an invertible 2-cell we shall denote $ \nu_f$ for concision
\[ (h^{A_1}_f, U(u)*\eta^{A_1}_f \alpha^U_{u, L_{A_1}(f)}) \stackrel{\nu^u_f \atop \simeq}{\Rightarrow} ( U(s^u_f) h^{A_2}_{U(u)f}, U(\sigma^u_f)* h^{A_2}_{U(u)f} \eta^{A_2}_{U(u)f})  \]
This 2-cell provides itself a decomposition of the 2-cell above as depicted below 
\[\begin{tikzcd}[sep=huge]
	B & {U(A_{U(u)f})} & {U(A_f)} \\
	& {U(A_2)}
	\arrow["{U(s^u_f)}"{description}, from=1-2, to=1-3]
	\arrow[""{name=0, anchor=center, inner sep=0}, "{U(uL_{A_1}(f))}", curve={height=-12pt}, from=1-3, to=2-2]
	\arrow["UL_{A_2}(U(u)f)"{description, near end}, from=1-2, to=2-2]
	\arrow["{h^{A_2}_{U(u)f}}"{description}, from=1-1, to=1-2]
	\arrow[""{name=1, anchor=center, inner sep=0}, "{h^{A_1}_f}", start anchor=45, curve={height=-30pt}, from=1-1, to=1-3]
	\arrow[""{name=2, anchor=center, inner sep=0}, "{{U(u)f}}"', curve={height=12pt}, from=1-1, to=2-2]
	\arrow["{\eta^{A_2}_{U(u)f} \atop \simeq}"{description}, Rightarrow, draw=none, from=1-2, to=2]
	\arrow["{\nu^u_f \atop \simeq}"{description}, Rightarrow, draw=none, from=1, to=1-2]
	\arrow["{U(\sigma^u_f) \atop \simeq}"{description}, Rightarrow, draw=none, from=1-2, to=0]
\end{tikzcd} \quad = \quad  
\begin{tikzcd}[sep=large]
	B && {U(A_f)} \\
	& {U(A_2)}
	\arrow[""{name=0, anchor=center, inner sep=0}, "{h^{A_1}_f}", from=1-1, to=1-3]
	\arrow["{U(uL_{A_1}(f))}", from=1-3, to=2-2]
	\arrow["{U(u)f}"', from=1-1, to=2-2]
	\arrow["{U(u)*\eta^{A_1}_f \alpha^U_{u, L_{A_1}(f)} \atop \simeq}"{description, pos=0.28}, Rightarrow, draw=none, from=0, to=2-2]
\end{tikzcd}\]
Those 2-cells can be pasted to produce a decomposition of the unit 2-cell of the composite $ U(u)f$:
\[\begin{tikzcd}[sep=large]
	B && {U(A_1)} \\
	& {U(A_f)} \\
	{U(A_{U(u)f})} && {U(A_2)}
	\arrow[""{name=0, anchor=center, inner sep=0}, "f", from=1-1, to=1-3]
	\arrow[""{name=1, anchor=center, inner sep=0}, "{U(u)}", from=1-3, to=3-3]
	\arrow["{h^{A_1}_f}"{description}, from=1-1, to=2-2]
	\arrow["{UL_{A_1}(f)}"{description}, from=2-2, to=1-3]
	\arrow["U(uL_{A_1}(f))"{description}, from=2-2, to=3-3]
	\arrow[""{name=2, anchor=center, inner sep=0}, "{h^{A_2}_{U(u)f}}"', from=1-1, to=3-1]
	\arrow[""{name=3, anchor=center, inner sep=0}, "{UL_{A_2}(U(u)f)}"', from=3-1, to=3-3]
	\arrow["{U(s^u_f)}"{description}, from=3-1, to=2-2]
	\arrow["{\eta^{A_1}_f \atop \simeq}"{description}, Rightarrow, draw=none, from=0, to=2-2]
	\arrow["{\alpha^U_{u,L_{A_1}(f)} \atop \simeq}"{description}, Rightarrow, draw=none, from=1, to=2-2]
	\arrow["{U(\sigma^u_f) \atop \simeq}"{description}, Rightarrow, draw=none, from=2-2, to=3]
	\arrow["{\nu^u_f \atop \simeq}"{description}, Rightarrow, draw=none, from=2, to=2-2]
\end{tikzcd} \quad = \quad 
\begin{tikzcd}
	B && {U(A_1)} \\
	\\
	{U(A_{U(u)f})} && {U(A_2)}
	\arrow["f", from=1-1, to=1-3]
	\arrow["{U(u)}", from=1-3, to=3-3]
	\arrow["{h^{A_2}_{U(u)f}}"', from=1-1, to=3-1]
	\arrow["{UL_{A_2}(U(u)f)}"', from=3-1, to=3-3]
	\arrow["{\eta^{A_2}_{U(u)f} \atop \simeq}"{description}, draw=none, from=1-1, to=3-3]
\end{tikzcd} \]
The maps $ \sigma^u_f$ define altogether a pseudonatural transformation as depicted below
\[\begin{tikzcd}
	{\mathcal{A}/A_1 } & {\mathcal{B}/U(A_1)} \\
	{\mathcal{A}/A_2   } & {\mathcal{B}/U(A_2)}
	\arrow["{L_{A_1}}"', from=1-2, to=1-1]
	\arrow["{\mathcal{B}/U(u)}", from=1-2, to=2-2]
	\arrow[""{name=0, anchor=center, inner sep=0}, "{\mathcal{A}/u}"', from=1-1, to=2-1]
	\arrow[""{name=1, anchor=center, inner sep=0}, "{L_{A_2}}", from=2-2, to=2-1]
	\arrow["{\sigma^u}"', curve={height=6pt}, shorten <=4pt, shorten >=4pt, Rightarrow, from=1, to=0]
\end{tikzcd}\]

Let us also describe how this construction interacts with 2-cells, for it involves some subtleties. Let be some 2-cell $ \phi : u \rightarrow u'$ in $\mathcal{A}$: then the respective natural transformations $ \sigma^u$ and $\sigma^{u'}$ are related as follows. For $ f : B \rightarrow U(A_1)$ we have a diagram in $\mathcal{A}$, which is a span in $ \mathcal{A}/A_2$
\[\begin{tikzcd}[sep=large]
	{A_{U(u)f}} \\
	{A_f)} & {A_1} && {A_2} \\
	{A_{U(u')f}}
	\arrow[""{name=0, anchor=center, inner sep=0}, "{u'}"{description, pos=0.4}, curve={height=12pt}, from=2-2, to=2-4]
	\arrow[""{name=1, anchor=center, inner sep=0}, "{u}"{description, pos=0.4}, curve={height=-12pt}, from=2-2, to=2-4]
	\arrow["{L_{A_1}(f)}"{description}, from=2-1, to=2-2]
	\arrow["{s^u_f}", from=1-1, to=2-1]
	\arrow["{s^{u'}_f}", from=3-1, to=2-1]
	\arrow[""{name=2, anchor=center, inner sep=0}, "{L_{A_2}(U(u)f)}", curve={height=-12pt}, from=1-1, to=2-4]
	\arrow[""{name=3, anchor=center, inner sep=0}, "{L_{A_2}(U(u')f)}"', curve={height=12pt}, from=3-1, to=2-4]
	\arrow["{\sigma^u_f \atop \simeq}"{description}, draw=none, from=2, to=2-1]
	\arrow["{\sigma^{u'}_f \atop \simeq}"{description}, Rightarrow, draw=none, from=3, to=2-1]
	\arrow["{\phi}", shift right=2, shorten <=3pt, shorten >=3pt, Rightarrow, from=1, to=0]
\end{tikzcd}\]
which is part of the following diagram in $\mathcal{B}$
\[\begin{tikzcd}[sep=large]
	& {U(A_{U(u)f})} \\
	B & {U(A_f)} & {U(A_1)} && {U(A_2)} \\
	& {U(A_{U(u')f})}
	\arrow[""{name=0, anchor=center, inner sep=0}, "{U(u')}"{description, pos=0.4}, curve={height=12pt}, from=2-3, to=2-5]
	\arrow[""{name=1, anchor=center, inner sep=0}, "{U(u)}"{description, pos=0.4}, curve={height=-12pt}, from=2-3, to=2-5]
	\arrow["{UL_{A_1}(f)}"{description}, from=2-2, to=2-3]
	\arrow["{h^{A_1}_f}"{description}, from=2-1, to=2-2]
	\arrow[""{name=2, anchor=center, inner sep=0}, "{h^{A_2}_{U(u)f}}", from=2-1, to=1-2]
	\arrow["{U(s^u_f)}"{description}, from=1-2, to=2-2]
	\arrow[""{name=3, anchor=center, inner sep=0}, "{h^{A_2}_{U(u')f}}"', from=2-1, to=3-2]
	\arrow["{U(s^{u'}_f)}"{description}, from=3-2, to=2-2]
	\arrow[""{name=4, anchor=center, inner sep=0}, "{UL_{A_2}(U(u)f)}", curve={height=-12pt}, from=1-2, to=2-5]
	\arrow[""{name=5, anchor=center, inner sep=0}, "{UL_{A_2}(U(u')f)}"', curve={height=12pt}, from=3-2, to=2-5]
	\arrow["{\nu^u_f \atop \simeq}"{description}, Rightarrow, draw=none, from=2, to=2-2]
	\arrow["{\nu^{u'}_f \atop \simeq}"{description, pos=0.6}, Rightarrow, draw=none, from=3, to=2-2]
	\arrow["{U(\sigma^u_f) \atop \simeq}"{description}, draw=none, from=4, to=2-2]
	\arrow["{U(\sigma^{u'}_f) \atop \simeq}"{description}, Rightarrow, draw=none, from=5, to=2-2]
	\arrow["{U(\phi)}", shift right=2, shorten <=3pt, shorten >=3pt, Rightarrow, from=1, to=0]
\end{tikzcd}\]

Beware however that this does not produce any 3-dimensional data relating explicitly $ \sigma^u$ and $ \sigma_{u'}$, for the pseudoslices are not pseudofunctorial on the underlying 2-category as stated in \cref{Pseudoslice is not pseudofunctorial}. Moreover, observe that at this step we cannot infer the existence of a canonical map $ A_{U(u)f)} \rightarrow A_{U(u')f}$, though such a map shall exist for non trivial reason as seen below. Moreover, this is related to be problem of the bistable factorization of globular 2-cells as we shall discuss later. 

\end{remark}

But actually we automatically have the following:

\begin{proposition}
If $ U : \mathcal{A} \rightarrow \mathcal{B}$ is a local right adjoint, then for any $ u : A_1 \rightarrow A_2$ in $ \mathcal{A}$, we have the 2-dimensional Beck-Chevalley condition at $u$: that is, the pseudonatural transformation $ \sigma^u$ is a pointwise equivalence. 
\end{proposition}

\begin{remark}\label{pseudoinverse of the BC map}
The last condition unfolds as follows: there is a natural transformation $ \tau^u : \mathcal{A}/u L_{A_1} \Rightarrow L_{A_2} \mathcal{B}/U(u)$ and two natural, pointwisely invertible modifications $ \mathfrak{a}^u$, $ \mathfrak{b}^u$ whose pasting produces the identity of $ \sigma^u$
\[ 
\begin{tikzcd}
L_{A_2}\mathcal{B}/U(u) \arrow[d, equal] \arrow[r, "\sigma^u"] & \mathcal{A}/u L_{A_1} \arrow[d, equal] \arrow[ld, "\tau^u" description, "\mathfrak{a}^u \atop \simeq"{inner sep= 3pt}, "\mathfrak{b}^u \atop \simeq"'{inner sep=3pt}] \\
L_{A_2}\mathcal{B}/U(u) \arrow[r, "\sigma^u"']          & \mathcal{A}/u L_{A_1}                                           
\end{tikzcd} = 
\begin{tikzcd}
L_{A_2}\mathcal{B}/U(u) \arrow[d, equal] \arrow[r, "\sigma^u"] \arrow[rd, "\simeq", phantom] & \mathcal{A}/u L_{A_1} \arrow[d, equal] \\
L_{A_2}\mathcal{B}/U(u) \arrow[r, "\sigma^u"']                                        & \mathcal{A}/u L_{A_1}          
\end{tikzcd} \]
Moreover, as equivalences in bicategories of pseudofunctors are exactly the pointwise equivalences, this is equivalent to require, for each $ f : B \rightarrow U(A_1) $, the existence of $(t^u_n, \tau^u_n) : u L_{A_1}(f) \rightarrow L_{A_2}(U(u)f)$ in $ \mathcal{A}/A_2$ pseudoinverse to $ (\sigma^u) = (s^u_n, \sigma^u_n) : L_{A_2}(U(u)f) \rightarrow u L_{A_1}(f)$, that is, such that
\[ 
\begin{tikzcd}
                                                                                                          & A_f \arrow[d, "L_{A_1}(f)" description] \arrow[rd, "t^u_n"]                &                                        \\
A_{U(u)f} \arrow[rd, "L_{A_2}(U(u)f)"'] \arrow[ru, "s^u_n"] \arrow[r, "\sigma^u_n \atop \simeq", phantom] & A_1 \arrow[d, "u" description] \arrow[r, "\tau^u_n \atop \simeq", phantom] & A_{U(u)f} \arrow[ld, "L_{A_2}(U(u)f)"] \\
                                                                                                          & A_2                                                                        &                                       
\end{tikzcd} = 
\begin{tikzcd}
                                                                       & A_f \arrow[rd, "t^u_n"] \arrow[d, "\mathfrak{b}^u_n \atop \simeq", phantom] &                                        \\
A_{U(u)f} \arrow[rd, "L_{A_2}(U(u)f)"'] \arrow[ru, "s^u_n"] \arrow[rr, equal] & {} \arrow[d, "\simeq", phantom]                                             & A_{U(u)f} \arrow[ld, "L_{A_2}(U(u)f)"] \\
                                                                       & A_2                                                                         &                                       
\end{tikzcd} \] 
\[ 
\begin{tikzcd}
A_f \arrow[d, "L_{A_1}(f)"'] \arrow[r, "t^u_n"] \arrow[rd, "\sigma^u_n \atop \simeq", phantom] & A_{U(u)f} \arrow[d, "L_{A_2}(U(u)f)" description] \arrow[r, "s^u_n"] \arrow[rd, "\tau^u_n \atop \simeq", phantom] & A_f \arrow[d, "L_{A_1}(f)"] \\
A_1 \arrow[r, "u"']                                                                            & A_2                                                                                                               & A_1 \arrow[l, "u"]         
\end{tikzcd} = 
\begin{tikzcd}
                                                                        & A_{U(u)f} \arrow[rd, "s^u_n"] \arrow[d, "\mathfrak{a}^u_n \atop \simeq", phantom] &                             \\
A_f \arrow[d, "L_{A_1}(f)"'] \arrow[ru, "t^u_n"] \arrow[rr, equal] & {} \arrow[d, "\simeq", phantom]                                                   & A_f \arrow[d, "L_{A_1}(f)"] \\
A_1 \arrow[r, "u"']                                                     & A_2                                                                               & A_1 \arrow[l, "u"]         
\end{tikzcd} \]
\end{remark}

\begin{proof}
If we paste the composite 2-cell 
\[\begin{tikzcd}[row sep=large, column sep=huge]
	B && {U(A_{U(u)f})} \\
	{U(A_f)} && {U(A_1)}
	\arrow["{h^{A_1}_f}"', from=1-1, to=2-1]
	\arrow["{UL_{A_1}(f)}"', from=2-1, to=2-3]
	\arrow[""{name=0, anchor=center, inner sep=0}, "{h^{A_2}_{U(u)f}}", from=1-1, to=1-3]
	\arrow[""{name=1, anchor=center, inner sep=0}, "{U(L_{A_1}(f)s^u_f)}", from=1-3, to=2-3]
	\arrow[""{name=2, anchor=center, inner sep=0}, "{U(s^u_f)}"{description}, from=1-3, to=2-1]
	\arrow["{\nu^u_f \atop \simeq}"{description}, shift right=1, Rightarrow, draw=none, from=0, to=2-1]
	\arrow["{\alpha^U_{(s^u_f, L_{A_1}(f))} \atop \simeq}"{description, pos=0.4}, shift left=2, Rightarrow, draw=none, from=1, to=2]
\end{tikzcd}\]
with the inverse of the unit 2-cell at $f$ as below
\[\begin{tikzcd}[sep=large]
	B && {U(A_{U(u)f})} \\
	& {U(A_f)} \\
	& {U(A_1)}
	\arrow["{h^{A_1}_f}"{description}, from=1-1, to=2-2]
	\arrow["{UL_{A_1}(f)}"{description}, from=2-2, to=3-2]
	\arrow[""{name=0, anchor=center, inner sep=0}, "f"', shift right=1, curve={height=22pt}, from=1-1, to=3-2]
	\arrow[""{name=1, anchor=center, inner sep=0}, "{h^{A_2}_{U(u)f}}", from=1-1, to=1-3]
	\arrow[""{name=2, anchor=center, inner sep=0}, "{U(L_{A_1}(f)s^u_f)}", shift left=3, curve={height=-22pt}, from=1-3, to=3-2]
	\arrow["{U(s^u_f)}"{description}, from=1-3, to=2-2]
	\arrow["{\nu^u_f \atop \simeq}"{description}, Rightarrow, draw=none, from=1, to=2-2]
	\arrow["{(\eta^{A_1}_f)^{-1} \atop \simeq}"{description, pos=0.4}, Rightarrow, draw=none, from=2-2, to=0]
	\arrow["{\alpha^U_{(s^u_f, L_{A_1}(f))} \atop \simeq}"{description, pos=0.44}, Rightarrow, draw=none, from=2-2, to=2]
\end{tikzcd}\]
we get a morphism 
\[\begin{tikzcd}[sep=huge]
	f &&& {U(L_{A_1}(f)s^u_f)}
	\arrow["{(h^{A_2}_{U(u)f},\alpha^U_{(s^u_f, L_{A_1}(f))} * h^{A_2}_{U(u)f} (\eta^{A_1}_f)^{-1})}", from=1-1, to=1-4]
\end{tikzcd}\]
which is sent by the equivalence of homcategories to a morphism $ L_{A_1}(f) \rightarrow  L_{A_1}(f)s^u_f$ consisting of a 2-cell
\[\begin{tikzcd}
	{A_f} && {A_{U(u)f}} \\
	& {A_1}
	\arrow[""{name=0, anchor=center, inner sep=0}, "{t^u_f}", from=1-1, to=1-3]
	\arrow["{L_{A_1}(f)s^u_f}"{pos=0.3}, from=1-3, to=2-2]
	\arrow["{L_{A_1}(f)}"', from=1-1, to=2-2]
	\arrow["{\theta^u_f \atop \simeq}"{description}, Rightarrow, draw=none, from=0, to=2-2]
\end{tikzcd}\]
Whiskering this 2-cell with $u$ and then pasting it with $ \sigma^u_f$ provides us with the desired 2-cell announced in \cref{pseudoinverse of the BC map} 
\[\begin{tikzcd}
	{A_f} && {A_{U(u)f}} \\
	& {A_1} \\
	& {A_2}
	\arrow[""{name=0, anchor=center, inner sep=0}, "{t^u_f}", from=1-1, to=1-3]
	\arrow["{L_{A_2}(U(u)f)}"{description}, from=1-3, to=2-2]
	\arrow["{L_{A_1}(f)}"', from=1-1, to=2-2]
	\arrow["u"', from=2-2, to=3-2]
	\arrow[""{name=1, anchor=center, inner sep=0}, "{L_{A_2}(U(u)f)}", curve={height=-18pt}, from=1-3, to=3-2]
	\arrow["{\theta^u_f \atop \simeq}"{description, pos=0.4}, Rightarrow, draw=none, from=0, to=2-2]
	\arrow["{\sigma^u_f \atop \simeq}"{description}, Rightarrow, draw=none, from=1, to=2-2]  
\end{tikzcd} = 
\begin{tikzcd}
	{A_f} && {A_{U(u)f}} \\
	& {A_1}
	\arrow[""{name=0, anchor=center, inner sep=0}, "{t^u_f}", from=1-1, to=1-3]
	\arrow["{L_{A_2}(U(u)f)}"{pos=0.3}, from=1-3, to=2-2]
	\arrow["{uL_{A_1}(f)}"', from=1-1, to=2-2]
	\arrow["{\tau^u_f \atop \simeq}"{description}, Rightarrow, draw=none, from=0, to=2-2]
\end{tikzcd}
\]
Moreover, because of the factorization above of the 2-cell $\alpha^U_{(s^u_f, L_{A_1}(f))} * h^{A_2}_{U(u)f} (\eta^{A_1}_f)^{-1}$ through the 2-cell $ \nu^u_f$, this map $ t^u_f$ is exhibited as a pseudosection of $ s^u_f$
\[\begin{tikzcd}
	{A_f} && {A_f} \\
	& {A_{U(u)f}}
	\arrow["{t^u_f}"', from=1-1, to=2-2]
	\arrow["{s^u_f}"', from=2-2, to=1-3]
	\arrow[""{name=0, anchor=center, inner sep=0}, Rightarrow, no head, from=1-1, to=1-3]
	\arrow["{\frak{a}^u_f \atop \simeq}"{description}, Rightarrow, draw=none, from=0, to=2-2]
\end{tikzcd}\]
and the desired 2-cell $ \tau^u_f$ can be defined as the pasting 
\[\begin{tikzcd}[row sep=large]
	{A_f} && {A_{U(u)f}} \\
	{A_f} & {A_1} & {A_2}
	\arrow[Rightarrow, no head, from=1-1, to=2-1]
	\arrow[""{name=0, anchor=center, inner sep=0}, "{t^u_f}", from=1-1, to=1-3]
	\arrow["{L_{A_2}(U(u)f)}"{description}, from=1-3, to=2-3]
	\arrow[""{name=1, anchor=center, inner sep=0}, "{s^u_f}"{description}, from=1-3, to=2-1]
	\arrow["{L_{A_1}(f)}"', from=2-1, to=2-2]
	\arrow["u"', from=2-2, to=2-3]
	\arrow["{\frak{a}^u_f \atop \simeq}"{description}, Rightarrow, draw=none, from=0, to=2-1]
	\arrow["{\sigma^u_f \atop \simeq}"{description}, Rightarrow, draw=none, from=1, to=2-3]
\end{tikzcd}  \; = \; 
\begin{tikzcd}[row sep=large]
	{A_f} & {A_{U(u)f}} \\
	{A_1} & {A_2}
	\arrow["{L_{A_2}(U(u)f)}"{pos=0.3}, from=1-2, to=2-2]
	\arrow["{t^u_f}", from=1-1, to=1-2]
	\arrow["u"', from=2-1, to=2-2]
	\arrow[""{name=0, anchor=center, inner sep=0}, "{L_{A_1}s^u_f}"{description, pos=0.6}, from=1-2, to=2-1]
	\arrow["{L_{A_1}(f)}"', from=1-1, to=2-1]
	\arrow["{\sigma^u_f \atop \simeq}"{description, pos=0.8}, Rightarrow, draw=none, from=0, to=2-2]
	\arrow["{\theta^u_f \atop \simeq}"{description, pos=0.2}, draw=none, from=1-1, to=0]
\end{tikzcd} \; = \; 
\begin{tikzcd}[row sep=large]
	{A_f} & {A_{U(u)f}} \\
	{A_f} & {A_2}
	\arrow["{L_{A_2}(U(u)f)}"{pos=0.3, description}, from=1-2, to=2-2]
	\arrow["{uL_{A_1}(f)}"', from=2-1, to=2-2]
	\arrow["{t^u_f}", from=1-1, to=1-2]
	\arrow[Rightarrow, no head, from=1-1, to=2-1]
	\arrow["{\tau^u_f \atop \simeq}"{description, pos=0.3}, draw=none, from=1-1, to=2-2]
\end{tikzcd}\]
Now we want to show that this pseudosection is also a pseudoretract. Observe that $ t^u_f$ comes equipped with an invertible 2-cell provided by the invertible component of the pointwisely invertible unit $\frak{i}^{A_2}_{f,L_{A_1}(f)s^u_f}$ of the adjunction of homcategories, we shall denote as $\mu^u_f$ for concision:
\[\begin{tikzcd}
	B && {U(A_{U(u)f})} \\
	& {U(A_f)}
	\arrow["{h^{A_1}_f}"', from=1-1, to=2-2]
	\arrow[""{name=0, anchor=center, inner sep=0}, "{h^{A_2}_{U(u)f}}", from=1-1, to=1-3]
	\arrow["{U(t^u_f)}"', from=2-2, to=1-3]
	\arrow["{\mu^u_f \atop \simeq}"{description}, Rightarrow, draw=none, from=0, to=2-2]
\end{tikzcd}\]
Now if we can paste the inverse of this invertible unit with the unit 2-cell of $ U(u)f$ and the image along $U$ of the 2-cell $\tau^u_f$ obtained above
\[\begin{tikzcd}[sep=large]
	B && {U(A_f)} \\
	& {U(A_{U(u)f})} \\
	& {U(A_2)}
	\arrow["{h^{A_2}_{U(u)f}}"{description}, from=1-1, to=2-2]
	\arrow["{UL_{A_2}(U(u)f)}"{description}, from=2-2, to=3-2]
	\arrow[""{name=0, anchor=center, inner sep=0}, "{U(u)f}"', curve={height=25pt}, shift right=1, from=1-1, to=3-2]
	\arrow[""{name=1, anchor=center, inner sep=0}, "{h^{A_1}_f}", from=1-1, to=1-3]
	\arrow["{U(t^u_f)}"{description}, from=1-3, to=2-2]
	\arrow[""{name=2, anchor=center, inner sep=0}, "{U(uL_{A_1}(f))}", curve={height=-30pt}, from=1-3, to=3-2]
	\arrow["{\eta^{A_2}_{U(u)f} \atop \simeq}"{description, pos=0.4}, Rightarrow, draw=none, from=2-2, to=0]
	\arrow["{(\mu^u_f)^{-1} \atop \simeq}"{description}, Rightarrow, draw=none, from=1, to=2-2]
	\arrow["{U(\tau^u_f) \atop \simeq}"{description, pos=0.4}, Rightarrow, draw=none, from=2-2, to=2]
\end{tikzcd}\]
then we get back the 2-cell from which we induced $ (s^u_f, \sigma^u_f)$ in \cref{expression of the BC mate}; but this 2-cell factorizes through the unit of $ U(u)f$, so $ s^u_f$ is exhibited as a section of $t^u_f$ 
\[\begin{tikzcd}[sep=large]
	{A_{U(u)f}} & {A_f} \\
	{A_{U(u)f}} & {A_2}
	\arrow["{s^u_f}", from=1-1, to=1-2]
	\arrow[Rightarrow, no head, from=1-1, to=2-1]
	\arrow[""{name=0, anchor=center, inner sep=0}, "{t^u_f}"{description}, from=1-2, to=2-1]
	\arrow["{L_{A_2}(U(u)f)}"', from=2-1, to=2-2]
	\arrow["{uL_{A_1}(f)}", from=1-2, to=2-2]
	\arrow["{\frak{b}^u_f \atop \simeq}"{description, pos=0.3}, Rightarrow, draw=none, from=1-1, to=0]
	\arrow["{\tau^u_f \atop \simeq}"{description}, Rightarrow, draw=none, from=0, to=2-2]
\end{tikzcd}\]
Moreover, one can check that those point-wise inverses $ (t^u_f, \tau^u_f)$ can be gathered into a pseudonatural transformation $ \tau_u$ which is a pseudoinverse of $\sigma^u$, exhibiting the latter as a pointwise equivalence, so that we have the desired 2-dimensional Beck-Chevalley condition. 
\end{proof}

\subsection{Bistable pseudofunctor}

As well as the 1-categorical notion of local right adjoint was equivalent to the notion of stable functor, the notion of local right biadjoint can be reformulated through the following 2-dimensional analog of stable functors:


\begin{definition}
Let $ U : \mathcal{A} \rightarrow \mathcal{B}$ a pseudofunctor. A \emph{$U$-generic morphism} is a 1-cell $n: B \rightarrow U(A)$ in $ \mathcal{B}$ such that for any pseudosquare 
\[ 
\begin{tikzcd}
B \arrow[r, "f"] \arrow[d, "n"'] \arrow[rd, "\alpha \atop \simeq", phantom] & U(A_1) \arrow[d, "U(u)"] \\
U(A) \arrow[r, "U(v)"']                                                      & U(A_2)                  
\end{tikzcd} \] there exists a 1-cell $ w_\alpha : A \rightarrow A_1$ in $\mathcal{A}$, unique up to a unique invertible 2-cell, and unique pair of invertible 2-cell $ \omega_\alpha $ and $ \nu_\alpha$ related as below
\[ 
\begin{tikzcd}[column sep=large, row sep=large]
B \arrow[r, "f"] \arrow[d, "n"'] \arrow[rd, "\nu_\alpha \atop \simeq", phantom, very near start] & U(A_1) \arrow[d, "U(u)"]                                     \\
U(A) \arrow[r, "U(v)"'] \arrow[ru, "U(w_\alpha)" description]                               & U(A_2) \arrow[lu, "U(\omega_\alpha) \atop \simeq", phantom, very near start]  
\end{tikzcd} \textrm{ with } \begin{tikzcd}
{}                                       & A_1 \arrow[d, "u"]      \\                              
 A \arrow[ru, "w_\alpha"] \arrow[r, "v"'] & A_2 \arrow[lu, "\omega_\alpha \atop \simeq", phantom, very near start]
\end{tikzcd} \]
and such that moreover for any morphism of pseudosquares $ \phi : (f_1, \alpha_1) \Rightarrow (f_2, \alpha_2)$ of the following form 
\[\begin{tikzcd}
	B && {U(A_1)} \\
	{U(A)} && {U(A_2)}
	\arrow["{U(v)}"', from=2-1, to=2-3]
	\arrow["n"', from=1-1, to=2-1]
	\arrow["{U(u)}", from=1-3, to=2-3]
	\arrow[""{name=0, anchor=center, inner sep=0}, "{f_1}", bend left=25, start anchor=50, from=1-1, to=1-3]
	\arrow[""{name=1, anchor=center, inner sep=0}, "{f_2}"{description}, from=1-1, to=1-3]
	\arrow["{\alpha_2 \atop \simeq}"{description}, draw=none, from=2-1, to=1-3]
	\arrow["\phi", shorten <=2pt, shorten >=2pt, Rightarrow, from=0, to=1]
\end{tikzcd} = \begin{tikzcd}
	B && {U(A_1)} \\
	{U(A)} && {U(A_2)}
	\arrow["{U(v)}"', from=2-1, to=2-3]
	\arrow["n"', from=1-1, to=2-1]
	\arrow["{U(u)}", from=1-3, to=2-3]
	\arrow["{f_1}", from=1-1, to=1-3]
	\arrow["{\alpha_1 \atop \simeq}"{description}, draw=none, from=2-1, to=1-3]
\end{tikzcd}\]
there is a unique 2-cell $\sigma_\phi : w_{\alpha_1} \Rightarrow w_{\alpha_2}$ in $ \mathcal{A}$ between the corresponding fillers such that we have the following 2-dimensional equalities in $\mathcal{A}$ and $\mathcal{B}$ respectively
\[  \begin{tikzcd}[row sep=large, column sep=large]
 {}                                       & A_1 \arrow[d, "u"]      \\                              
 A \arrow[ru, "w_{\alpha_1}"] \arrow[r, "v"'] & A_2 \arrow[lu, "\omega_{\alpha_1} \atop \simeq", phantom, very near start]
\end{tikzcd} 
= 
\begin{tikzcd}[row sep=large, column sep=large]
 {}                                       & A_1 \arrow[d, "u"]      \\                              
 A \arrow[ru, ""{name=U, inner sep=2pt}, "w_{\alpha_2}" description], \arrow[ru, bend left=40, ""{name= D, inner sep=0.1pt, below}, "w_{\alpha_1}"] \arrow[r, "v"'] & A_2 \arrow[lu, "\omega_{\alpha_2} \atop \simeq", phantom, very near start] \arrow[from=D, to=U, Rightarrow, "\sigma_\phi"]
\end{tikzcd} \]

\[ 
\begin{tikzcd}[sep=large]
	B && {U(A_1)} \\
	{U(A)}
	\arrow[""{name=0, anchor=center, inner sep=0}, "{f_2}"{description}, from=1-1, to=1-3]
	\arrow[""{name=1, anchor=center, inner sep=0}, "{f_1}"{start anchor=60}, bend left=25, start anchor=50, from=1-1, to=1-3]
	\arrow["n"', from=1-1, to=2-1]
	\arrow[""{name=2, anchor=center, inner sep=0}, "{U(w_{\alpha_2})}"', from=2-1, to=1-3]
	\arrow["\phi", shorten <=3pt, shorten >=3pt, Rightarrow, from=1, to=0]
	\arrow["{\nu_{\alpha_2} \atop \simeq}"{description}, Rightarrow, draw=none, from=1-1, to=2]
\end{tikzcd}
= 
\begin{tikzcd}[sep=large]
	B && {U(A_1)} \\
	{U(A)}
	\arrow["{f_1}", from=1-1, to=1-3]
	\arrow["n"', from=1-1, to=2-1]
	\arrow[""{name=0, anchor=center, inner sep=0}, "{U(w_{\alpha_1})}"{description, pos=0.4}, from=2-1, to=1-3]
	\arrow[""{name=1, anchor=center, inner sep=0}, "{U(w_{\alpha_2})}"{description}, curve={height=28pt}, from=2-1, to=1-3]
	\arrow["{\nu_{\alpha_1} \atop \simeq}"{description}, Rightarrow, draw=none, from=1-1, to=0]
	\arrow["{U(\sigma_\phi)}"{pos=0.4}, shorten <=7pt, shorten >=6pt, Rightarrow, from=0, to=1]
\end{tikzcd}\]
\end{definition}

\begin{remark}
This is a pseudoversion of the notion of generic morphism in \cite{weber2007familial}[5.1]. Beware that, in the 2-dimensional part of the diagonalization condition, we consider morphisms of pseudosquare with only a 2-cell on the top side in $\mathcal{B}$ but no 2-cell on the bottom side in $\mathcal{A}$. Functoriality relatively to more general shapes of morphisms of pseudosquares (and even of morphisms of lax-squares) shall be involved later in the notion of \emph{lax-generic morphism} in the next section. \\

An explanation to this restriction will appear when comparing bistable pseudofunctors and local right biadjoints: while the 2-cell on the top side in a morphism of pseudosquare $ \phi : (f_1, \alpha_1) \Rightarrow (f_2,\alpha_2)$ as above can be represented by a 2-cell in the pseudoslice over $B/U(A_2)$, further 2-dimensional data on the bottom side $ v_1 \Rightarrow v_2$ could not be represented in such a way unless moving to lax-slices. This is also related to the truncation of the Beck-Chevalley condition for local right biajoints. All those restrictions will be lifted in the lax version.  
\end{remark}

\begin{definition}
A pseudofunctor $ U : \mathcal{A} \rightarrow \mathcal{B}$ is \emph{bistable} if any arrow of the form $ f: B \rightarrow U(A)$ admits a factorization 
\[ 
\begin{tikzcd}
B \arrow[rr, "f"] \arrow[rd, "n_f"'] & {} \arrow[d, "\alpha_f \atop \simeq", phantom] & U(A) \\
                                     & U(A_f) \arrow[ru, "U(u_f)"']                   &     
\end{tikzcd} \]
with $\alpha: U(u_f) n_f \simeq f$ invertible, which is moreover unique, in the sense that for any other invertible 2-cell of the form $ \alpha : U(u)n \simeq f$ through the image of an object $ A'$ in $\mathcal{A}$, there exists an equivalence $ A' \simeq A$ in $\mathcal{A}$, unique up to unique invertible 2-cell, and a unique pair of invertible 2-cells $ \nu_\alpha $ in $\mathcal{B}$ and $ \omega_\alpha$ in $\mathcal{A}$ such that 
\[ 
\begin{tikzcd}[column sep=large, row sep=large]
   B  \arrow[rd, "n"'] \arrow[r, "n_f" description] \arrow[rr, "f", bend left, "\alpha_f \atop \simeq"{name=U, below, inner sep=5pt}] & U(A_f) \arrow[r, "U(u_f)" description] \arrow[ld, "\nu_{\alpha} \atop \simeq", phantom, very near start] \arrow[rd, "U(\omega_\alpha) \atop \simeq", phantom, very near start] & U(A)  \\
{}                                                                         & U(A') \arrow[ru, "U(u)"'] \arrow[u, "U(e_\alpha) \atop \simeq" description]                                                                 & {}    
\end{tikzcd} = \begin{tikzcd}
 B \arrow[rd, "n"'] \arrow[rr, "f"] & {} \arrow[d, "\alpha \atop \simeq", phantom, near start] & U(A)\\
            & U(A') \arrow[ru, "U(u)"']                    &     
\end{tikzcd} \] 
\end{definition}

\begin{theorem}
A functor $U : \mathcal{A} \rightarrow \mathcal{B}$ is bistable if and only if it is a local right biadjoint.
\end{theorem}

\begin{proof}
Suppose $ U$ is bistable, and $A$ is a an object of $ \mathcal{A}$. For any arrow $ u : A' \rightarrow A$ in $\mathcal{A}$ and any $ f : B \rightarrow U(A)$ in $\mathcal{B}$, define the pseudofunctor $L_A : \mathcal{B}/U(A) \rightarrow \mathcal{A}/A $ as follows:\begin{itemize}
    \item on a 0-cell $ f : B \rightarrow U(A)$, $L_A $ returns the canonical arrow in $ \mathcal{A}$ provided by right part of the generic factorization of $f$, that is, $ L_A(f) =u_f$ with $ \alpha_f : U(u_f)n_f \simeq f$;
    \item on a 1-cell 
    \[ 
\begin{tikzcd}
B_1 \arrow[rd, "f_1"'] \arrow[rr, "g"] & {} \arrow[d, "\alpha \atop \simeq", phantom] & B_2 \arrow[ld, "f_2"] \\
                                       & U(A)                                         &                      
\end{tikzcd} \]
it returns the corresponding right 2-cell in $\mathcal{A}$
\[ 
\begin{tikzcd}
A_{f_1} \arrow[rd, "L_A(f_1)"'] \arrow[rr, "w_{\alpha \alpha_f}"] & {} \arrow[d, "\omega_{\alpha \alpha_f} \atop \simeq", phantom] & A_{f_2} \arrow[ld, "L_A(f_2)"] \\
                                                                  & A                                                              &                               
\end{tikzcd} \]
provided by the generic diagonalization of the 2-cell 
\[ 
\begin{tikzcd}
B_1 \arrow[d, "n_{f_1}"'] \arrow[r, "g"] \arrow[rrd, "\alpha_{f_2}^{-1}\alpha\alpha_{f_1} \atop \simeq", phantom] & B_2 \arrow[r, "n_{f_2}"] & U(A_{f_2}) \arrow[d, "U_AL_A(f_2)"] \\
U(A_{f_1}) \arrow[rr, "U_AL_A(f_2)"']                                                                           &                          & U(A)                                 
\end{tikzcd} \] 
where $ \alpha_{f_1}$, $\alpha_{f_2}$ are the canonical invertible 2-cell involved in the generic factorization of $f_1$, $f_2$ respectively;
\item for a 2-cell 
\[ 
\begin{tikzcd}
B_1 \arrow[rd, "f_1"'] \arrow[rr, "g", bend left, ""{name=U, inner sep=0.1pt, below}] \arrow[rr, "g'" description, ""{name=D, inner sep=1pt}] & {} \arrow[d, "\alpha' \atop \simeq", phantom] & B_2 \arrow[ld, "f_2"] \\
                                                                               & U(A)    \arrow[from=U, to=D, Rightarrow, "\phi"]                                    &                      
\end{tikzcd} = \begin{tikzcd}
B_1 \arrow[rd, "f_1"'] \arrow[rr, "g"] & {} \arrow[d, "\alpha \atop \simeq", phantom] & B_2 \arrow[ld, "f_2"] \\
                                       & U(A)                                         &                      
\end{tikzcd} \]
it returns the composite 2-cell 
 \[ 
\begin{tikzcd}[row sep=large, column sep=large]
A_{f_1} \arrow[rd, "L_A(f_1)"'] \arrow[rr, "w_{\alpha_{f_2}^{-1}\alpha'\alpha_{f_1}}" description, ""{name=D, inner sep=2pt}] \arrow[rr, "w_{\alpha_{f_2}^{-1}\alpha\alpha_{f_1}}", bend left, ""{name=U, inner sep=0.1pt, below}] & {} \arrow[d, "\omega_{\alpha_{f_2}^{-1}\alpha'\alpha_{f_1}} \atop \simeq", phantom] & A_{f_2} \arrow[ld, "L_A(f_2)"] \\
                                                                                 \arrow[from=U, to=D,"\sigma_\phi", Rightarrow]                                                                                   & A                                                                                  &                               
\end{tikzcd} = 
\begin{tikzcd}[row sep=large, column sep=large]
A_{f_1} \arrow[rd, "L_A(f_1)"'] \arrow[rr, "w_{\alpha_{f_2}^{-1}\alpha\alpha_{f_1}}"] & {} \arrow[d, "\omega_{\alpha_{f_2}^{-1}\alpha\alpha_{f_1}} \atop \simeq", phantom] & A_{f_2} \arrow[ld, "L_A(f_2)"] \\
                                                                                      & A                                                                                  &                               
\end{tikzcd} \]
provided by the functoriality of the generic diagonalization for the morphisms of pseudosquares
\[ 
\begin{tikzcd}[row sep=large, column sep=large]
B_1 \arrow[d, "n_{f_1}"'] \arrow[rrd, "\alpha_{f_2}^{-1}\alpha'\alpha_{f_1} \atop \simeq", phantom] \arrow[rr, "n_{f_2}g", bend left=30, end anchor=155, ""{name=U, inner sep=0.1pt, below}] \arrow[rr, "n_{f_2}g'" description, ""{name=D, inner sep=2pt}] &  & U(A_{f_2}) \arrow[d, "U_AL_A(f_2)"] \\
U(A_{f_1}) \arrow[rr, "U_AL_A(f_2)"']                                                                                                                 &  & U(A) \arrow[from=U, to=D, "n_{f_2}^*\phi", Rightarrow]                  
\end{tikzcd} = 
\begin{tikzcd}[row sep=large, column sep=large]
B_1 \arrow[d, "n_{f_1}"'] \arrow[rrd, "\alpha_{f_2}^{-1}\alpha\alpha_{f_1} \atop \simeq", phantom] \arrow[rr, "n_{f_2}g" description] &  & U(A_{f_2}) \arrow[d, "U_AL_A(f_2)"] \\
U(A_{f_1}) \arrow[rr, "U_AL_A(f_2)"']                                                                                               &  & U(A)
\end{tikzcd} \]
\end{itemize}

Now we want to exhibit an equivalence of categories
\[ \mathcal{A}/A [L_A(f),u] \simeq \mathcal{B}/U(A)[f,U_A(u)] \]
\begin{itemize}
    \item For the left to right direction, take any 2-cell 
\[ 
\begin{tikzcd}
A_f \arrow[rr, "w"] \arrow[rd, "L_A(f)"'] & {} \arrow[d, "\omega \atop \simeq", phantom] & A' \arrow[ld, "u"] \\
                                          & A                                            &                   
\end{tikzcd} \]
to the the pasting of its image along $U$ with the canonical 2-cell in the generic factorization of $f$
\[ 
\begin{tikzcd}[row sep=large, column sep=huge]
B \arrow[r, "n_f"] \arrow[rd, "f"'] & U(A_f) \arrow[d, "U_AL_A(f)" description] \arrow[r, "U(w)"] \arrow[rd, "U(\omega) \atop \simeq", phantom, very near start] \arrow[ld, "\alpha_f \atop \simeq", phantom, very near start] & U(A') \arrow[ld, "U(u)"] \\
{}                                  & U(A)                                                                                                                                                     & {}                      
\end{tikzcd} \]
which provides the object $ (U(w)n_f, U(\omega)^*n_f\alpha_f)$ in $\mathcal{B}/U(A)[f,U_A(u)]$, and similarily for a morphisms of triangle;
\item for the right to left direction, take a triangle 
\[ 
\begin{tikzcd}
B \arrow[rr, "g"] \arrow[rd, "f"'] & {} \arrow[d, "\alpha \atop \simeq", phantom] & U(A') \arrow[ld, "U(u)"] \\
                                   & U(A)                                         &                         
\end{tikzcd} \]
to the 2-cell 
\[ 
\begin{tikzcd}
A_{f} \arrow[rd, "L_A(f)"'] \arrow[rr, "w_{\alpha \alpha_f}"] & {} \arrow[d, "\omega_{\alpha \alpha_f} \atop \simeq", phantom] & A' \arrow[ld, "u"] \\
                                                               & A                                                              &                   
\end{tikzcd} \]
provided by the generic diagonalization of the composite 2-cell
\[ 
\begin{tikzcd}
B \arrow[r, "g"] \arrow[d, "n_f"'] \arrow[rd, "\alpha \alpha_f \atop \simeq", phantom] & U(A') \arrow[d, "U(u)"] \\
U(A_f) \arrow[r, "U_AL_A(f)"']                                                    & U(A)                   
\end{tikzcd} \]
\end{itemize}
It is easy to check that those processes define two functors that are adjoints, with the units and counit being provided as follows:\begin{itemize}
    \item for $(g,\alpha)$ in $ \mathcal{B}/U(A)[f,U_A(u)]$, the unique 2-cell $ \nu_{\alpha\alpha_f}  $ provided by the generic diagonalization can be seen as a 2-cell
    \[ 
\begin{tikzcd}[row sep=large, column sep=huge]
B \arrow[r, "n_f" description] \arrow[rd, "f"'] \arrow[rr, "g", bend left, "\nu_{\alpha\alpha_f} \atop \simeq"'{inner sep=5pt}] & U(A_f) \arrow[d, "U_AL_A(f)" description] \arrow[r, "U(w_{\alpha\alpha_f})" description] \arrow[rd, "U(\omega_{\alpha\alpha_f})  \atop \simeq", phantom, very near start] \arrow[ld, "\alpha_f \atop \simeq ", phantom, very near start] & U(A') \arrow[ld, "U(v)"] \\
{}                                                             & U(A)                                                                                                                                                                                       & {}                      
\end{tikzcd} = \begin{tikzcd}
B \arrow[rr, "g"] \arrow[rd, "f"'] & {} \arrow[d, "\alpha \atop \simeq", phantom] & U(A') \arrow[ld, "U(u)"] \\
                                   & U(A)                                         &                         
\end{tikzcd}  \] so we can define the counit as the inverse $ \nu_{\alpha\alpha_f}^{-1} : (U(w_{\alpha\alpha_f})n_f, U(\omega_{\alpha\alpha_f})^*n_f\alpha_f) \Rightarrow (g,\alpha)$;
\item for a triangle $ (w,\omega)$ in $\mathcal{A}/A[L_A(f),u]$, the generic factorization of the pasting 
\[ 
\begin{tikzcd}[row sep=large, column sep=large]
B \arrow[d, "n_f"'] \arrow[r, "U(w)n_f"]                      & U(A') \arrow[d, "U(u)"] \\
U(A_f) \arrow[r, "U_AL_A(f)"'] \arrow[ru, "U(w)" description, "="{inner sep=6pt}, "U(\omega) \atop \simeq"'{inner sep=4pt}] & U(A)                   
\end{tikzcd} \]
produces from the uniqueness of the generic factorization a unique invertible 2-cell $ \sigma : w \Rightarrow w_{U(\omega)1_{U(w)n_f}} $ such that   
\[ 
\begin{tikzcd}[column sep=huge, row sep=large]
A_f \arrow[rr, "w_{U(\omega)1_{U(w)n_f}}" description] \arrow[rd, "L_A(f)"'] \arrow[rr, "w", bend left, " \sigma \atop \simeq "'{inner sep=5pt}] & {} \arrow[d, "\omega_{U(\omega)1_{U(w)n_f}} \atop \simeq", phantom] & A' \arrow[ld, "u"] \\
                                                                                            & A                                                                   &                   
\end{tikzcd} = \begin{tikzcd}
A_f \arrow[rr, "w"] \arrow[rd, "L_A(f)"'] & {} \arrow[d, "\omega \atop \simeq", phantom] & A' \arrow[ld, "u"] \\
                                          & A                                            &                   
\end{tikzcd} \]
\end{itemize}
In this context, the unit of the adjunction $ L_A \dashv U_A$ coincides with the 2-cell produced by the generic factorization of a given $ f : B \rightarrow U(A)$, that is, $ \eta^A_f$ is the 2-cell $(n_f, \alpha_f)$. \\

Observe that in this context the 2-dimensional Beck-Chevalley condition can be retrieved more directly as a consequence of the uniqueness of the candidate factorization. Indeed, for $ u:  A_1 \rightarrow A_2$ in $\mathcal{A}$ and any arrow $ f : B \rightarrow U(A_1)$, the composite with the universal factorization
\[ 
\begin{tikzcd}
B \arrow[rr, "f"] \arrow[rd, "n_f"'] & {} \arrow[d, "\nu_f \atop \simeq", phantom] & U(A_1) \arrow[r, "U(u)"] & U(A_2) \\
                                           & U(A_f) \arrow[ru, "U_{A_1}L_{A_1}(f)"']            &                          &       
\end{tikzcd} \]
then we have a factorization through a candidate
\[ 
\begin{tikzcd}
B \arrow[rd, "h^{A_1}_f"'] \arrow[r, "f"] & U(A_1) \arrow[r, "U(u)"] \arrow[d, "\nu_f \atop \simeq", phantom] & U(A_2) \\
                                          & U(A_f) \arrow[ru, "U(u L_{A_1}(f))"']                       &       
\end{tikzcd}\]
(where we omit the associator of $ U$ at $ uL_{A_1}(f)$), so by uniqueness of the candidate factorization there exists an equivalence $ e : A_f \simeq A_{U(u)f}$, unique up to unique invertible 2-cell, and a unique pair of invertible 2-cells 
\[ 
\begin{tikzcd}[row sep=large, column sep=huge]
B \arrow[rd, "n_f"', "\lambda \atop \simeq"{inner sep=5pt}] \arrow[r, "n_{U(u)f}" description] \arrow[rrr, "U(u)f", bend left=20, " \nu_f \atop \simeq"'{inner sep=5pt}] & U(A_{U(u)f}) \arrow[rr, "UL_{A_2}(U(u)f)" description]                                   && U(A_2) \\
                                                                                       & U(A_f) \arrow[rru, "U(uL_{A_1}(f))"', "U(\rho) \atop \simeq"{inner sep=5pt}] \arrow[u, "U(e) \atop \simeq" description] &&       
\end{tikzcd} \]
We left as an exercise to check that the pseudo-inverse $ s$ of each $e$ can be gathered into a natural point-wise equivalence $L_{A_2} \mathcal{B}/U(u) \stackrel{\simeq}{\rightarrow} \mathcal{A}/u L_{A_1}  $.
\\

Conversely, suppose we are given a local right biadjoint $U$. In each $A$, we denote abusively the local unit in some $ f : B \rightarrow U(B)$ as a 2-cell $(h^A_f, \eta^A_f) : f \rightarrow U_AL_A(f)$. Then $U$-generic morphisms are all morphisms $ n : B \rightarrow U(A)$ that are sent by $L_A$ to equivalence $L_A(n) : A_n \simeq A $ in $\mathcal{A}$, with a choice of inverse $e_n : A \simeq A_n$ and the corresponding invertible 2-cell as on the right: 
\[ 
\begin{tikzcd}
B \arrow[rr, "n", " \eta_n^A \atop \simeq"'{inner sep=8pt}] \arrow[rd, "h^A_n"'] &                                 & U(A) \\
                                       & U(A_n) \arrow[ru, "U_AL_A(n) \atop \simeq"'] &     
\end{tikzcd} \textrm{ with } 
\begin{tikzcd}
A_n \arrow[d, "L_A(n)"'] \arrow[r, equal] & A_n \arrow[d, "L_A(n)"]  \\
A \arrow[r, equal]            \arrow[ru, "e_n" description, "\mathfrak{u}_n \atop \simeq"'{inner sep=4pt}, " \mathfrak{v}_n \atop \simeq"{inner sep=4pt}]            & A                                                    
\end{tikzcd} \]
Indeed, if one considers a pseudo-square 
\[ 
\begin{tikzcd}
B \arrow[r, "g"] \arrow[d, "n"'] \arrow[rd, "\alpha \atop \simeq", phantom] & U(A_1) \arrow[d, "U(v)"] \\
U(A) \arrow[r, "U(u)"']                                                     & U(A_2)                  
\end{tikzcd} \]
then by the Beck-Chevalley property of the adjunction, we are given with a point-wise natural equivalence 
\[\sigma^u : L_{A_2} \mathcal{B}/U(u) \stackrel{\simeq}{\longrightarrow} \mathcal{A}/u L_{A} \]
which returns at $n $ an equivalence in $\mathcal{A}/A$
\[ (s^u_n,\sigma^u_n) : L_{A_2}(U(u) n) \stackrel{\simeq}{\longrightarrow} u L_A(n) \]
consisting of a 2-cell
\[ 
\begin{tikzcd}
A_{U(u)n} \arrow[rr, "s_n^u \atop \simeq"] \arrow[rd, "L_A(U(u)n)"'] & {} \arrow[d, "\sigma^u_n \atop \simeq", phantom] & A_n \arrow[ld, "u L_A(n)"] \\
                                                            & A_2                                                &                                   
\end{tikzcd} \]
in $ \mathcal{A}$ with $ \tau^u$ its pseudoinverse and $ \mathfrak{a}^u$ and $ \mathfrak{b}^u$ the corresponding invertible modifications, so we get an equivalence $L_A(n) s^u_n : A_{U(u)n} \simeq A_n \simeq A $ in $\mathcal{A}$. Now $(g,\alpha ) : U(u)n \rightarrow U_{A_2}(v)$ produces through the biadjunction $ L_{A_2}\dashv U_{A_2}$ a morphism $ (w,\omega) : L_{A_2}(U(u)n) \rightarrow v$ in $ \mathcal{A}/A_2$ and a composite 2-cell in $\mathcal{A}$
\[ 
\begin{tikzcd}[row sep=large, column sep=large]
A_n \arrow[r, "t^u_n \atop \simeq"] \arrow[d, no head, equal] & A_{U(u)n} \arrow[r, "w"] \arrow[d, " L_{A_2}(U(u)n)" description] \arrow[ld, "s^u_n \atop \simeq" description, "\mathfrak{a}^u_n \atop \simeq"'{inner sep=5pt}, "\sigma^u_n \atop \simeq"{inner sep=5pt}] & A_1 \arrow[ld, "L_A(U(u)n)", "\omega \atop \simeq"'{inner sep=5pt}] \\
A_n \arrow[r, "u L_A(n)"']                                                 & A_2                                                                                                   &                             
\end{tikzcd} \]
On the other hand the unit of $ (g, \alpha)$ returns the following factorization
\[ 
\begin{tikzcd}[row sep=large, column sep=large]
    B \arrow[rr, "g"] \arrow[d, "h_n^A" description] \arrow[dd, "n"', bend right=70,"\eta^A_n \atop \simeq"{inner sep=3pt}]                            & {} \arrow[d, "{(\mathfrak{i}^A_{n,u})_{(g,\alpha)} \atop \simeq}", phantom]                                                                                                                & U(A_1) \arrow[dd, "U(v)"] \\
 U(A_n) \arrow[d, "U_AL_A(n)" description] \arrow[r, "U(s^u_n)"] & U(A_{U(u)n}) \arrow[rd, "U_{A_2}L_{A_2}(U(u)n)" description] \arrow[ru, "U(w)" description] \arrow[r, "U(\omega) \atop \simeq", phantom] \arrow[ld, "U(\sigma^u_n) \atop \simeq", phantom] & {}                        \\
    U(A) \arrow[rr, "U(u)"']                                                                                    &                                                                                                                                                                                            & U(A_2)                   
\end{tikzcd} \]
And combining those two data produces a path in $\mathcal{A}$
\[ \begin{tikzcd}
 A \arrow[r, "e_n"] & A_n \arrow[r, "t^u_n"] & A_{U(u)n} \arrow[r, "w"] & A_1
\end{tikzcd}  \]
and a decomposition of the pseudo-square 
\[ 
\begin{tikzcd}[row sep=large, column sep=huge]
B \arrow[rd, "h_n^A" description, " \eta_f^A \atop \simeq"'{inner sep=4pt, near start}] \arrow[rrr, "g"] \arrow[d, "n"'] & {} \arrow[rd, "{(\mathfrak{i}^A_{n,u})_{(g,\alpha)} \atop \simeq}", phantom]                      & {} \arrow[d, phantom]                                                                                                                          & U(A_1) \arrow[dddd, "U(v)"] \\
U(A) \arrow[rd, "U(e_n)" description, "U(\mathfrak{v}_n) \atop \simeq"{inner sep=3pt}, "U(\mathfrak{u}_n) \atop \simeq"'{inner sep=3pt}] \arrow[d, equal]           & U(A_n) \arrow[d, equal] \arrow[r, "U(t^u_n)"] \arrow[l, "U_AL_A(n)" description]             & U(A_{U(u)n}) \arrow[ru, "U(w)" description] \arrow[d, equal] \arrow[ld, "U(s_n^u)" description, "U(\mathfrak{a}^u_n) \atop \simeq"'{inner sep=5pt}, "U(\mathfrak{b}^u_n) \atop \simeq"{inner sep=5pt}]                                           & {}                          \\
U(A) \arrow[rd, "U(e_n)" description, "U(\mathfrak{v}_n) \atop \simeq"{inner sep=3pt}, "U(\mathfrak{u}_n) \atop \simeq"'{inner sep=6pt}] \arrow[dd, equal]          & U(A_n) \arrow[l, "U_AL_A(n)" description] \arrow[d, equal] \arrow[r, "U(t^u_n)" description] & U(A_{U(u)n}) \arrow[rdd, "U_{A_2}L_{A_2}(U(u)n)" description] \arrow[ld, "U(s^u_n)" description, "U(\mathfrak{a}^u_n) \atop \simeq"'{inner sep=5pt}] \arrow[ruu, "U(w)" description] \arrow[r, "U(\omega) \atop \simeq", phantom] &  {}                           \\
                                                                   & U(A_n) \arrow[ld, "U_AL_A(n)" description] \arrow[rrd, "U(\sigma^u_n) \atop \simeq", phantom]     &                                                                                                                                                &                             \\
U(A) \arrow[rrr, "U(u)"']                                          &                                                                                                   &                                                                                                                                                & U(A_2)                     
\end{tikzcd}  \]
The uniqueness of this solution up to unique invertible 2-cell follows from the uniqueness of the solutions in each step, as left as an exercise for the reader.\\

Now we prove $ n$ also satisfies the 2-dimensional condition of generic morphisms. Let be a morphism of 2-cells as below
\[\begin{tikzcd}[column sep=large]
	B & {U(A_1)} \\
	{U(A)} & {U(A_2)}
	\arrow[""{name=0, anchor=center, inner sep=0}, "{g_2}"{description}, from=1-1, to=1-2]
	\arrow["{U(v)}", from=1-2, to=2-2]
	\arrow["n"', from=1-1, to=2-1]
	\arrow["{U(u)}"', from=2-1, to=2-2]
	\arrow[""{name=1, anchor=center, inner sep=0}, "{g_1}", bend left=25, start anchor=45, end anchor=160, from=1-1, to=1-2]
	\arrow["{\alpha_2 \atop \simeq}"{description}, draw=none, from=1-1, to=2-2]
	\arrow["\phi"{pos=0.4}, shorten <=2pt, shorten >=2pt, Rightarrow, from=1, to=0]
\end{tikzcd} \; = \; \begin{tikzcd}[column sep=large]
	B & {U(A_1)} \\
	{U(A)} & {U(A_2)}
	\arrow["{U(v)}", from=1-2, to=2-2]
	\arrow["n"', from=1-1, to=2-1]
	\arrow["{U(u)}"', from=2-1, to=2-2]
	\arrow["{g_1}", from=1-1, to=1-2]
	\arrow["{\alpha_1 \atop \simeq}"{description}, draw=none, from=1-1, to=2-2]
\end{tikzcd} \]
Then from the biadjuncion $ L_{A_2} \dashv U_{A_2}$, the equivalence of categories
\[ \mathcal{B}/U(A_2) [ U(u)n, U(v)] \simeq \mathcal{A}/A[L_{A_2}(U(u)n), v] \]
sends the 2-cell $ \phi : (g_1, \alpha_1) \Rightarrow (g_2, \alpha_2)$ on a 2-cell $ \sigma_\phi :(w_{\alpha_1}, \omega_{\alpha_1}) \Rightarrow (w_{\alpha_2}, \omega_{\alpha_2})$. Then, the whiskering 
\[\begin{tikzcd}
	A & {A_n} & {A_{U(u)n}} && {A_1}
	\arrow["{e_n}", from=1-1, to=1-2]
	\arrow["{t^u_n}", from=1-2, to=1-3]
	\arrow[""{name=0, anchor=center, inner sep=0}, "{w_{\alpha_1}}", start anchor=20, bend left=20, from=1-3, to=1-5]
	\arrow[""{name=1, anchor=center, inner sep=0}, "{w_{\alpha_1}}"', start anchor=-20, bend right=20, from=1-3, to=1-5]
	\arrow["{\sigma_\phi}", shift right=1, shorten <=2pt, shorten >=2pt, Rightarrow, from=0, to=1]
\end{tikzcd}\]
with the equivalences defined in the previous step provides a morphism of lifts associated to the morphism of pseudosquares. This achieves to prove that $ n$ is $U$-generic.\\

Conversely, one can check that any $U$-generic morphism has to be of this form, by uniqueness of the factorization applied to the generic diagonal. 
Then the generic factorization of an arrow $ f : B \rightarrow U(A)$ can be defined as its unit 2-cell for the local adjunction over $A$:
\[ 
\begin{tikzcd}
B \arrow[rr, "f", "\eta_f^A \atop \simeq "'{inner sep=8pt}] \arrow[rd, "h^A_f"'] &                                 & U(A) \\
                                       & U(A_f) \arrow[ru, "U_AL_A(f)"'] &     
\end{tikzcd} \]
that is, define $ n_f = h^A_f$, $ u_f= L_A(f)$ and $\alpha_f = \eta^A_f$. Now the universal property of the unit provides us with the desired uniqueness condition associated to the generic factorization. 
\end{proof}

\begin{remark}
For a bistable pseudofunctor $U$, any two morphisms $ f_1, f_2 : B \rightarrow U(A)$ with an invertible 2-cell
\[\begin{tikzcd}
	B && {U(A)}
	\arrow[""{name=0, anchor=center, inner sep=0}, "{f_1}", bend left=25, start anchor=45, from=1-1, to=1-3]
	\arrow[""{name=1, anchor=center, inner sep=0}, "{f_2}"', bend right=25, start anchor=-45, from=1-1, to=1-3]
	\arrow["{\alpha \atop \simeq}"{description}, Rightarrow, draw=none, from=0, to=1]
\end{tikzcd}\] have equivalent generic factorization with $ A_{f_1} \simeq A_{f_2}$ as below: 
\[\begin{tikzcd}
	& {U(A_{f_1})} \\
	B && {U(A)} \\
	& {U(A_{f_2})}
	\arrow["{n_{f_1}}", from=2-1, to=1-2]
	\arrow["{U(u_{f_1})}", from=1-2, to=2-3]
	\arrow["{n_{f_2}}"', from=2-1, to=3-2]
	\arrow["{U(u_{f_2})}"', from=3-2, to=2-3]
	\arrow[""{name=0, anchor=center, inner sep=0}, "{U(u_\alpha) \atop \simeq}"{description}, from=1-2, to=3-2]
	\arrow["{\nu_\alpha \atop \simeq}"{description}, Rightarrow, draw=none, from=2-1, to=0]
	\arrow["{U(\omega_\alpha) \atop \simeq}"{description}, Rightarrow, draw=none, from=2-3, to=0]
\end{tikzcd}\]
In particular, before using the canonical structure of local right biadjoint and the associated Beck-Chevalley condition, one can see directly that postcomposing with an arrow in the range of $U$ does not modify the left part of the generic factorization.
\end{remark}

\section{Laxness conditions}

The notion of bistable pseudofunctor produces a convenient treatment of invertible 2-cells; however, one may ask how non-invertible 2-cells are to be factorized - those as below:
\[\begin{tikzcd}
	B && {U(A)}
	\arrow[""{name=0, anchor=center, inner sep=0}, "{f_1}", bend left=25, start anchor=45, from=1-1, to=1-3]
	\arrow[""{name=1, anchor=center, inner sep=0}, "{f_2}"', bend right=25, start anchor=-45, from=1-1, to=1-3]
	\arrow["{ \Downarrow \sigma }"{description}, draw=none, from=0, to=1]
\end{tikzcd}\] 
or also if one can diagonalize lax square and not just pseudosquares. The answer is that it might depend on further, different possible laxness conditions one might stipulate.\\

A possible laxness condition is described in \cite{walker2020lax} under the notion of \emph{lax-familial} pseudofunctors. They are defined also relatively to a class of generic morphisms, but this time with a more general lax-orthogonality condition relatively to lax-squares. For this notion is important in itself and is related, see \cite{walker2020lax}, to an elegant decomposition of the conerve into a lax bicolimit of representables, we choose here to give a rather detailed account of it and a few properties that are not yet found elsewhere. However this notion will not apply to our examples, where factorization of lax cell will exist, yet processing in a totally different way.

\subsection{Lax generic cells}



\begin{definition}
Let $ U : \mathcal{A} \rightarrow \mathcal{B}$ a pseudofunctor. A \emph{lax $U$-generic morphism}\index{lax-generic morphism} is a 1-cell $n: B \rightarrow U(A)$ in $ \mathcal{B}$ such that for any pseudosquare 
\[\begin{tikzcd}
	B & {U(A_1)} \\
	{U(A)} & {U(A_2)}
	\arrow["n"', from=1-1, to=2-1]
	\arrow[""{name=0, anchor=center, inner sep=0}, "f", from=1-1, to=1-2]
	\arrow["{U(u)}", from=1-2, to=2-2]
	\arrow[""{name=1, anchor=center, inner sep=0}, "{U(v)}"', from=2-1, to=2-2]
	\arrow["{\sigma \Uparrow}"{description}, Rightarrow, draw=none, from=0, to=1]
\end{tikzcd}\]
there exists a 1-cell $ w_\sigma : A \rightarrow A_1$ in $\mathcal{A}$, unique up to a unique invertible 2-cell, and unique pair of invertible 2-cells $ \nu_\sigma$ in $\mathcal{B}$ and $ \omega_\sigma $ in $\mathcal{A}$ such that $ \sigma $ decomposes as the pasting
\[\begin{tikzcd}[sep=large]
	B & {U(A_1)} \\
	{U(A)} & {U(A_2)}
	\arrow[""{name=0, anchor=center, inner sep=0}, "n"', from=1-1, to=2-1]
	\arrow[""{name=1, anchor=center, inner sep=0}, "f", from=1-1, to=1-2]
	\arrow[""{name=2, anchor=center, inner sep=0}, "{U(u)}", from=1-2, to=2-2]
	\arrow[""{name=3, anchor=center, inner sep=0}, "{U(v)}"', from=2-1, to=2-2]
	\arrow["{U(w_\sigma)}"{description}, from=2-1, to=1-2]
	\arrow["{\nu_\sigma \Uparrow}"{description}, Rightarrow, draw=none, from=1, to=0]
	\arrow["{ \Uparrow U(\omega_\sigma)}"{description}, Rightarrow, draw=none, from=2, to=3]
\end{tikzcd}\]
and moreover those data are universal in the sense that for any other factorization of this square 
\[\begin{tikzcd}[sep=large]
	B & {U(A_1)} \\
	{U(A)} & {U(A_2)}
	\arrow[""{name=0, anchor=center, inner sep=0}, "n"', from=1-1, to=2-1]
	\arrow[""{name=1, anchor=center, inner sep=0}, "f", from=1-1, to=1-2]
	\arrow[""{name=2, anchor=center, inner sep=0}, "{U(u)}", from=1-2, to=2-2]
	\arrow[""{name=3, anchor=center, inner sep=0}, "{U(v)}"', from=2-1, to=2-2]
	\arrow["{U(w)}"{description}, from=2-1, to=1-2]
	\arrow["{\nu \Uparrow}"{description}, Rightarrow, draw=none, from=1, to=0]
	\arrow["{ \Uparrow U(\omega)}"{description}, Rightarrow, draw=none, from=2, to=3]
\end{tikzcd}\]
there exists a unique 2-cell $ \xi : w \Rightarrow w_\sigma  $ such that we have the factorizations
\[\begin{tikzcd}[sep=large]
	B & {U(A_1)} \\
	{U(A)}
	\arrow["n"', from=1-1, to=2-1]
	\arrow["f", from=1-1, to=1-2]
	\arrow[""{name=0, anchor=center, inner sep=0}, "{U(w_\sigma)}"{description}, from=2-1, to=1-2]
	\arrow[""{name=1, anchor=center, inner sep=0}, "U(w)"', curve={height=24pt}, from=2-1, to=1-2]
	\arrow["U(\xi)", shorten <=3pt, shorten >=3pt, Rightarrow, from=1, to=0]
	\arrow["{\Uparrow \nu_\sigma}"{description}, Rightarrow, draw=none, from=1-1, to=0]
\end{tikzcd} = \begin{tikzcd}[sep=large]
	B & {U(A_1)} \\
	{U(A)}
	\arrow["n"', from=1-1, to=2-1]
	\arrow["f", from=1-1, to=1-2]
	\arrow[""{name=0, anchor=center, inner sep=0}, "U(w)"', from=2-1, to=1-2]
	\arrow["{\Uparrow\nu}"{description}, Rightarrow, draw=none, from=1-1, to=0]
\end{tikzcd}\]
\[\begin{tikzcd}[sep=large]
	& {A_1} \\
	A & {A_2}
	\arrow[""{name=0, anchor=center, inner sep=0}, "w_\sigma", from=2-1, to=1-2]
	\arrow["u", from=1-2, to=2-2]
	\arrow["v"', from=2-1, to=2-2]
	\arrow["{\omega_\sigma \Uparrow}"{description}, Rightarrow, draw=none, from=0, to=2-2]
\end{tikzcd} = 
\begin{tikzcd}[sep=large]
	& {A_1} \\
	A & {A_2}
	\arrow[""{name=0, anchor=center, inner sep=0}, "w_\sigma", curve={height=-18pt}, from=2-1, to=1-2]
	\arrow["u", from=1-2, to=2-2]
	\arrow["v"', from=2-1, to=2-2]
	\arrow["w"{name=1, anchor=center, inner sep=0, description}, from=2-1, to=1-2]
	\arrow["\xi", shorten <=3pt, shorten >=3pt, Rightarrow, from=1, to=0]
	\arrow["{\omega \Uparrow}"{description}, Rightarrow, draw=none, from=1, to=2-2]
\end{tikzcd}\]
Moreover the 2-cells $ \nu_\sigma$ and $ \omega_\sigma$ must be invertible as soon as $\sigma$ is.
\end{definition}

\begin{definition}
A lax generic morphism is \emph{functorially generic} if for any morphism of pseudosquares in $\mathcal{B}//U$ as below
\[\begin{tikzcd}
	B && {U(A_1)} \\
	{U(A)} && {U(A_2)}
	\arrow["n"', from=1-1, to=2-1]
	\arrow[""{name=0, anchor=center, inner sep=0}, "{f_1}"{description}, from=1-1, to=1-3]
	\arrow["{U(u)}", from=1-3, to=2-3]
	\arrow[""{name=1, anchor=center, inner sep=0}, "{U(v_1)}"', from=2-1, to=2-3]
	\arrow[""{name=2, anchor=center, inner sep=0}, "{f_2}", start anchor=40, bend left=30, from=1-1, to=1-3]
	\arrow["{\phi \Uparrow}"{description}, shorten <=2pt, shorten >=2pt, draw=none, Rightarrow, from=2, to=0]
	\arrow["{\sigma_1 \Uparrow}"{description}, Rightarrow, draw=none, from=0, to=1]
\end{tikzcd}=
\begin{tikzcd}
	B && {U(A_1)} \\
	{U(A)} && {U(A_2)}
	\arrow["n"', from=1-1, to=2-1]
	\arrow["{U(u)}", from=1-3, to=2-3]
	\arrow[""{name=0, anchor=center, inner sep=0}, "{U(v_1)}"', curve={height=18pt}, from=2-1, to=2-3]
	\arrow[""{name=1, anchor=center, inner sep=0}, "{f_2}", from=1-1, to=1-3]
	\arrow[""{name=2, anchor=center, inner sep=0}, "{U(v_2)}"{description}, from=2-1, to=2-3]
	\arrow["{\sigma_2 \Uparrow}"{description}, Rightarrow, draw=none, from=1, to=2]
	\arrow["{U(\gamma) \Uparrow}"{description}, Rightarrow, draw=none, from=2, to=0]
\end{tikzcd}\]
there is a unique 2-cell $\sigma_\phi : w_{\alpha_1} \Rightarrow w_{\alpha_2}$ in $ \mathcal{A}$ between the corresponding fillers such that
\[\begin{tikzcd}[sep=large]
	B & {U(A_1)} \\
	{U(A)}
	\arrow["n"', from=1-1, to=2-1]
	\arrow["{f_2}", from=1-1, to=1-2]
	\arrow[""{name=0, anchor=center, inner sep=0}, "{U(w_{\sigma_2})}"{description}, from=2-1, to=1-2]
	\arrow[""{name=1, anchor=center, inner sep=0}, "{U(w_{\sigma_1})}"', curve={height=24pt}, from=2-1, to=1-2]
	\arrow["{\nu_{\sigma_2} \Uparrow}"{description}, Rightarrow, draw=none, from=1-1, to=0]
	\arrow["U(\xi)", shorten <=3pt, shorten >=3pt, Rightarrow, from=1, to=0]
\end{tikzcd} =
\begin{tikzcd}[sep=large]
	B & {U(A_1)} \\
	{U(A)}
	\arrow["n"', from=1-1, to=2-1]
	\arrow[""{name=0, anchor=center, inner sep=0}, "{f_2}", curve={height=-18pt}, from=1-1, to=1-2]
	\arrow[""{name=1, anchor=center, inner sep=0}, "{U(w_{\sigma_1})}"', from=2-1, to=1-2]
	\arrow[""{name=2, anchor=center, inner sep=0}, "{f_1}"{description}, from=1-1, to=1-2]
	\arrow["{\phi \Uparrow}"{description}, Rightarrow, draw=none, from=0, to=2]
	\arrow["{\nu_{\sigma_1} \Uparrow}"{description}, Rightarrow, draw=none, from=1-1, to=1]
\end{tikzcd}\]

\[\begin{tikzcd}[sep=large]
	& {A_1} \\
	A & {A_2}
	\arrow[""{name=0, anchor=center, inner sep=0}, "{w_{\sigma_1}}"{description}, from=2-1, to=1-2]
	\arrow[""{name=1, anchor=center, inner sep=0}, "{w_{\sigma_2}}", curve={height=-18pt}, from=2-1, to=1-2]
	\arrow["u", from=1-2, to=2-2]
	\arrow["{v_1}"', from=2-1, to=2-2]
	\arrow["\xi"', shorten <=3pt, shorten >=3pt, Rightarrow, from=0, to=1]
	\arrow["{\omega_{\sigma_1} \Uparrow}"{description, pos=0.7}, Rightarrow, draw=none, from=0, to=2-2] 
\end{tikzcd} =
\begin{tikzcd}[sep=large]
	& {A_1} \\
	A & {A_2}
	\arrow[""{name=0, anchor=center, inner sep=0}, "{w_{\sigma_1}}", from=2-1, to=1-2]
	\arrow["u", from=1-2, to=2-2]
	\arrow[""{name=1, anchor=center, inner sep=0}, "{v_1}"', curve={height=18pt}, from=2-1, to=2-2]
	\arrow[""{name=2, anchor=center, inner sep=0}, "{v_2}"{description}, from=2-1, to=2-2]
	\arrow["{\gamma \Uparrow}"{description}, Rightarrow, draw=none, from=2, to=1]
	\arrow["{\omega_{\sigma_1} \Uparrow}"{description, pos=0.6}, shift left=1, Rightarrow, draw=none, from=0, to=2]
\end{tikzcd}\]
\end{definition}

\begin{remark}
In the following we shall require - and for now, suppose as so - any lax generic morphism to be functorially lax generic. We conjecture that actually any lax generic is automatically functorially lax generic for free, though we choose not to go into such discussion and incorporate this condition in the definition. 
\end{remark}

\begin{definition}
A \emph{$U$-lax generic 2-cell} is a 2-cell as below with $ \nu$ an lax generic 1-cell
\[\begin{tikzcd}
	B & {U(A_1)} \\
	{U(A)}
	\arrow[""{name=0, anchor=center, inner sep=0}, "n"', from=1-1, to=2-1]
	\arrow[""{name=1, anchor=center, inner sep=0}, "f", from=1-1, to=1-2]
	\arrow["{U(u)}"', from=2-1, to=1-2]
	\arrow["{\nu \Uparrow}"{description}, Rightarrow, draw=none, from=1, to=0]
\end{tikzcd}\]
such that we have the following two conditions: \begin{itemize}
    \item For any factorizations of $ \nu$ as a pasting of the following form
\[\begin{tikzcd}
	B & {U(A_1)} \\
	{U(A)}
	\arrow[""{name=0, anchor=center, inner sep=0}, "n"', from=1-1, to=2-1]
	\arrow[""{name=1, anchor=center, inner sep=0}, "f", from=1-1, to=1-2]
	\arrow[""{name=2, anchor=center, inner sep=0}, "{U(v)}"{description}, from=2-1, to=1-2]
	\arrow[""{name=3, anchor=center, inner sep=0}, "{U(u)}"', curve={height=18pt}, from=2-1, to=1-2]
	\arrow["{\lambda \Uparrow}"{description}, Rightarrow, draw=none, from=1, to=0]
	\arrow["{U(\zeta)}", shorten <=3pt, shorten >=3pt, Rightarrow, from=3, to=2]
\end{tikzcd}\]
there exists a unique 2-cell $ \xi : v \Rightarrow u$ which is a section of $ \zeta$ in $\mathcal{A}[A,A_1]$, that is $ \zeta\xi=1_{v}$, and such that we have a factorization of $ \lambda$ as    
\[\begin{tikzcd}
	B & {U(A_1)} \\
	{U(A)}
	\arrow[""{name=0, anchor=center, inner sep=0}, "n"', from=1-1, to=2-1]
	\arrow[""{name=1, anchor=center, inner sep=0}, "f", from=1-1, to=1-2]
	\arrow[""{name=2, anchor=center, inner sep=0}, "{U(u)}"{description}, from=2-1, to=1-2]
	\arrow[""{name=3, anchor=center, inner sep=0}, "{U(v)}"', curve={height=18pt}, from=2-1, to=1-2]
	\arrow["{\nu \Uparrow}"{description}, Rightarrow, draw=none, from=1, to=0]
	\arrow["{U(\zeta)}", shorten <=3pt, shorten >=3pt, Rightarrow, from=3, to=2]
\end{tikzcd}\]
\item Any parallel pair of 2-cells in $\mathcal{A}$ whose image along $U$ are equalized by $\nu$ as depicted below 
\[\begin{tikzcd}
	B & {U(A_1)} \\
	{U(A)}
	\arrow[""{name=0, anchor=center, inner sep=0}, "n"', from=1-1, to=2-1]
	\arrow[""{name=1, anchor=center, inner sep=0}, "f", from=1-1, to=1-2]
	\arrow[""{name=2, anchor=center, inner sep=0}, "{U(u)}"{description}, from=2-1, to=1-2]
	\arrow[""{name=3, anchor=center, inner sep=0}, "{U(v)}"', curve={height=18pt}, from=2-1, to=1-2]
	\arrow["{\nu \Uparrow}"{description}, Rightarrow, draw=none, from=1, to=0]
	\arrow["{U(\xi)}", shorten <=3pt, shorten >=3pt, Rightarrow, from=3, to=2]
\end{tikzcd}
=
\begin{tikzcd}
	B & {U(A_1)} \\
	{U(A)}
	\arrow[""{name=0, anchor=center, inner sep=0}, "n"', from=1-1, to=2-1]
	\arrow[""{name=1, anchor=center, inner sep=0}, "f", from=1-1, to=1-2]
	\arrow[""{name=2, anchor=center, inner sep=0}, "{U(u)}"{description}, from=2-1, to=1-2]
	\arrow[""{name=3, anchor=center, inner sep=0}, "{U(v)}"', curve={height=18pt}, from=2-1, to=1-2]
	\arrow["{\nu \Uparrow}"{description}, Rightarrow, draw=none, from=1, to=0]
	\arrow["{U(\zeta)}", shorten <=3pt, shorten >=3pt, Rightarrow, from=3, to=2]
\end{tikzcd}
\]
must actually already be equal in $\mathcal{A}$.
\end{itemize}
\end{definition}

\subsection{Lax familial pseudofunctor}

\begin{definition}
A pseudofunctor $ U : \mathcal{A} \rightarrow \mathcal{B}$ is \emph{lax familial}\index{lax familia} (or also \emph{lax stable} to remain coherent with our terminology) if we have the following conditions:\begin{itemize}
    \item any arrow of the form $ f: B \rightarrow U(A)$ admits a bifactorization 
\[ 
\begin{tikzcd}
B \arrow[rr, "f"] \arrow[rd, "n_f"'] & {} \arrow[d, "\nu_f \atop \simeq", phantom] & U(A) \\
                                     & U(A_f) \arrow[ru, "U(u_f)"']                   &     
\end{tikzcd} \]
with $\nu_f: f \simeq U(u_f) n_f $ invertible and $ n_f$ a $U$-lax generic 1-cell. 
\item Generic 2-cells compose: that is, for a composite 2-cell as below
\[\begin{tikzcd}
	{B_1} & {B_2} & {B_3} \\
	{U(A_1)} & {U(A_2)} & {U(A_3)}
	\arrow["{n_1}"', from=1-1, to=2-1]
	\arrow[""{name=0, anchor=center, inner sep=0}, "{f_1}", from=1-1, to=1-2]
	\arrow["{n_2}"{description}, from=1-2, to=2-2]
	\arrow[""{name=1, anchor=center, inner sep=0}, "{U(u_1)}"', from=2-1, to=2-2]
	\arrow[""{name=2, anchor=center, inner sep=0}, "{U(u_2)}"', from=2-2, to=2-3]
	\arrow[""{name=3, anchor=center, inner sep=0}, "{f_2}", from=1-2, to=1-3]
	\arrow["{n_3}", from=1-3, to=2-3]
	\arrow["{\nu_1 \Uparrow}"{description}, Rightarrow, draw=none, from=0, to=1]
	\arrow["{\nu_2 \Uparrow}"{description}, Rightarrow, draw=none, from=3, to=2]
\end{tikzcd}\]
with $ \nu_1$ and $ \nu_2$ generic 2-cells, then the following composite also is generic:
\[\begin{tikzcd}
	{B_1} &&& {U(A_3)} \\
	{U(A_1)}
	\arrow["{n_1}"', from=1-1, to=2-1]
	\arrow["{n_3f_2f_1}", from=1-1, to=1-4]
	\arrow[""{name=0, anchor=center, inner sep=0}, "{U(u_2u_1)}"', curve={height=12pt}, from=2-1, to=1-4]
	\arrow["{\Uparrow \nu_2*f_1 u(u_2)*\nu_1 \,  }"{description}, shift left=1, Rightarrow, draw=none, from=1-1, to=0]
\end{tikzcd}\]
\item the equality 2-cells as below are generic:
\[\begin{tikzcd}
	B & {U(A)} \\
	{U(A)}
	\arrow["n"', from=1-1, to=2-1]
	\arrow["n", from=1-1, to=1-2]
	\arrow[Rightarrow, no head, from=2-1, to=1-2]
\end{tikzcd}\]
\end{itemize} 
\end{definition}

\begin{remark}
The first two conditions correspond to \cite{walker2020lax} definition; however, the last item appears to be required for our purpose; it is not clear whether it can be deduced from the others.
\end{remark}

First of all, it is important to precise that such a factorization, which shall be call \emph{generic factorization}, is unique up to a unique equivalence and invertible 2-cells:

\begin{lemma}\label{uniqueness of generic factorization}
Let be $ (\nu_1, n_1, u_1)$ and $ (\nu_2, n_2, u_2)$ two factorizations of a same $f : B \rightarrow U(A)$. Then we have an equivalence unique up to a unique invertible 2-cell $ e : A_1 \simeq A_2$ in $\mathcal{A}$ together with invertible 2-cells $  n_1 \simeq U(e)n_2$, $ u_2 = u_1e $.  
\end{lemma}

\begin{proof}
This is obtained by using respectively the generic property of $n_1 $ and $n_2$ at the invertible $\nu_2\nu_1^{-1} $ and $ \nu_1^{-1}\nu_2 $ as below
\[\begin{tikzcd}
	B & {U(A_2)} \\
	{U(A_1)} & {U(A)}
	\arrow[""{name=0, anchor=center, inner sep=0}, "{n_1}"', from=1-1, to=2-1]
	\arrow[""{name=1, anchor=center, inner sep=0}, "{n_2}", from=1-1, to=1-2]
	\arrow[""{name=2, anchor=center, inner sep=0}, "{U(u_1)}"', from=2-1, to=2-2]
	\arrow[""{name=3, anchor=center, inner sep=0}, "{U(u_2)}", from=1-2, to=2-2]
	\arrow[from=1-1, to=2-2]
	\arrow["{ \Uparrow {\nu_2 \atop\simeq}}"{description}, shift left=1, Rightarrow, draw=none, from=3, to=1]
	\arrow["{{\nu_1^{-1} \atop \simeq} \Uparrow}"{description}, shift right=1, Rightarrow, draw=none, from=2, to=0]
\end{tikzcd}  \hskip1cm \begin{tikzcd}
	B & {U(A_1)} \\
	{U(A_2)} & {U(A)}
	\arrow[""{name=0, anchor=center, inner sep=0}, "{n_2}"', from=1-1, to=2-1]
	\arrow[""{name=1, anchor=center, inner sep=0}, "{n_1}", from=1-1, to=1-2]
	\arrow[""{name=2, anchor=center, inner sep=0}, "{U(u_2)}"', from=2-1, to=2-2]
	\arrow[""{name=3, anchor=center, inner sep=0}, "{U(u_1)}", from=1-2, to=2-2]
	\arrow[from=1-1, to=2-2]
	\arrow["{{\nu_2^{-1} \atop \simeq} \Uparrow}"{description}, shift right=1, Rightarrow, draw=none, from=2, to=0]
	\arrow["{ \Uparrow {\nu_1 \atop\simeq}}"{description}, shift left=1, Rightarrow, draw=none, from=3, to=1]
\end{tikzcd} \]
Then one gets two factorizations of $ n_1$ through generic morphisms
\[\begin{tikzcd}
	B & {U(A_1)} \\
	{U(A_1)}
	\arrow["{n_1}"', from=1-1, to=2-1]
	\arrow["{n_1}", from=1-1, to=1-2]
	\arrow[Rightarrow, no head, from=2-1, to=1-2]
\end{tikzcd} \hskip1cm
\begin{tikzcd}
	B && {U(A_1)} \\
	& {U(A_2)} \\
	{U(A_1)}
	\arrow[""{name=0, anchor=center, inner sep=0}, "{n_1}"', from=1-1, to=3-1]
	\arrow[""{name=1, anchor=center, inner sep=0}, "{n_1}", from=1-1, to=1-3]
	\arrow["{n_2}"{description}, from=1-1, to=2-2]
	\arrow["{U(w_{\nu_1\nu_2^{-1} })}"', from=3-1, to=2-2]
	\arrow["{U(w_{\nu_2\nu_1^{-1} })}"', from=2-2, to=1-3]
	\arrow["{\nu_{\nu_1\nu_2^{-1} } \atop \simeq}"{description}, Rightarrow, draw=none, from=0, to=2-2]
	\arrow["{\nu_{\nu_2\nu_1^{-1} } \atop \simeq}"{description}, Rightarrow, draw=none, from=1, to=2-2]
\end{tikzcd}\]
which are related (by functoriality of the diagonalizations relatively to the morphisms of lax squares) by 2-cellular factorizations
\[\begin{tikzcd}
	B && {U(A_1)} \\
	\\
	{U(A_1)}
	\arrow["{n_1}"', from=1-1, to=3-1]
	\arrow["{n_1}", from=1-1, to=1-3]
	\arrow[""{name=0, anchor=center, inner sep=0}, "{U(w_{\nu_1^{-1}\nu_2 }w_{\nu_2^{-1}\nu_1 })}"', curve={height=18pt}, from=3-1, to=1-3]
	\arrow[""{name=1, anchor=center, inner sep=0}, Rightarrow, no head, from=3-1, to=1-3]
	\arrow["\xi"', shorten <=3pt, shorten >=3pt, Rightarrow, from=1, to=0]
\end{tikzcd} = \begin{tikzcd}[sep=large]
	B && {U(A_1)} \\
	\\
	{U(A_1)}
	\arrow["{n_1}"', from=1-1, to=3-1]
	\arrow["{n_1}", from=1-1, to=1-3]
	\arrow[""{name=0, anchor=center, inner sep=0}, "{U(w_{\nu_2\nu_1^{-1} }w_{\nu_1\nu_2^{-1} })}"', from=3-1, to=1-3]
	\arrow["{\nu_{\nu_2\nu_1^{-1}}\nu_{\nu_1\nu_2^{-1} }  \atop\simeq}"{description}, Rightarrow, draw=none, from=1-1, to=0]
\end{tikzcd}\]
\[\begin{tikzcd}[sep=large]
	B && {U(A_1)} \\
	\\
	{U(A_1)}
	\arrow["{n_1}"', from=1-1, to=3-1]
	\arrow["{n_1}", from=1-1, to=1-3]
	\arrow[""{name=0, anchor=center, inner sep=0}, "{U(w_{\nu_2\nu_1^{-1} }w_{\nu_1\nu_2^{-1} })}"{description}, from=3-1, to=1-3]
	\arrow[""{name=1, anchor=center, inner sep=0}, curve={height=30pt}, Rightarrow, no head, from=3-1, to=1-3]
	\arrow["\scriptsize{\nu_{\nu_2\nu_1^{-1}}\nu_{\nu_1\nu_2^{-1} }  \atop\simeq}"{description}, Rightarrow, draw=none, from=1-1, to=0]
	\arrow["\zeta"', shorten <=4pt, shorten >=4pt, Rightarrow, from=0, to=1]
\end{tikzcd} = \begin{tikzcd}
	B & {U(A_1)} \\
	{U(A_1)}
	\arrow["{n_1}"', from=1-1, to=2-1]
	\arrow["{n_1}", from=1-1, to=1-2]
	\arrow[Rightarrow, no head, from=2-1, to=1-2]
\end{tikzcd} \]
But then by genericity, those comparison morphisms both have retracts, which entails their invertibility: hence the composite $ w_{\nu_2\nu_1^{-1} }w_{\nu_1\nu_2^{-1} }$ is an equivalence. A same argument proves that the composite $ w_{\nu_1\nu_2^{-1} }w_{\nu_2\nu_1^{-1} }$ is also an equivalence. 
\end{proof}

Though the following result is immediate by the retraction properties in the definition of lax generic 2-cells, it is worth visualizing it once for all as we are going to use it later:

\begin{lemma}\label{mutually factorizing generics}
Two lax generic 2-cells that factorize themselves mutually are equivalent: if one has 
\[\begin{tikzcd}[sep=large]
	B & {U(A')} \\
	{U(A)}
	\arrow["n"', from=1-1, to=2-1]
	\arrow["f", from=1-1, to=1-2]
	\arrow[""{name=0, anchor=center, inner sep=0}, "{U(u_1)}"', curve={height=24pt}, from=2-1, to=1-2]
	\arrow[""{name=1, anchor=center, inner sep=0}, "{U(u_2)}"{description}, from=2-1, to=1-2]
	\arrow["{\Uparrow \nu_2}"{description}, Rightarrow, draw=none, from=1-1, to=1]
	\arrow["{U(\zeta)}", shorten <=2pt, shorten >=2pt, Rightarrow, from=0, to=1]
\end{tikzcd} = \begin{tikzcd}
	B & {U(A')} \\
	{U(A)}
	\arrow["n"', from=1-1, to=2-1]
	\arrow["f", from=1-1, to=1-2]
	\arrow[""{name=0, anchor=center, inner sep=0}, "{U(u_1)}"', from=2-1, to=1-2]
	\arrow["{\Uparrow \nu_1}"{description}, Rightarrow, draw=none, from=1-1, to=0]
\end{tikzcd}\]
\[\begin{tikzcd}[sep=large]
	B & {U(A')} \\
	{U(A)}
	\arrow[""{name=0, anchor=center, inner sep=0}, "{U(u_1)}"{description}, from=2-1, to=1-2]
	\arrow["n"', from=1-1, to=2-1]
	\arrow["f", from=1-1, to=1-2]
	\arrow[""{name=1, anchor=center, inner sep=0}, "{U(u_2)}"', curve={height=24pt}, from=2-1, to=1-2]
	\arrow["{\Uparrow \nu_1}"{description}, Rightarrow, draw=none, from=1-1, to=0]
	\arrow["{U(\xi)}", shorten <=2pt, shorten >=2pt, Rightarrow, from=1, to=0]
\end{tikzcd} = \begin{tikzcd}
	B & {U(A')} \\
	{U(A)}
	\arrow["n"', from=1-1, to=2-1]
	\arrow["f", from=1-1, to=1-2]
	\arrow[""{name=0, anchor=center, inner sep=0}, "{U(u_2)}"', from=2-1, to=1-2]
	\arrow["{\Uparrow \nu_2}"{description}, Rightarrow, draw=none, from=1-1, to=0]
\end{tikzcd}\]
then in fact $ \zeta$ and $\xi$ are mutual inverses so that $ u_1 \simeq u_2$. 
\end{lemma}

\begin{proposition}\label{universal property of the lax generic factorization}
If $ U$ is lax familial, then for a 1-cell $ f : B \rightarrow U(A)$ and any 2-cell of the form 
\[\begin{tikzcd}
	B && {U(A)} \\
	& {U(A')}
	\arrow[""{name=0, anchor=center, inner sep=0}, "f", from=1-1, to=1-3]
	\arrow["g"', from=1-1, to=2-2]
	\arrow["{U(u)}"', from=2-2, to=1-3]
	\arrow["{\sigma \Downarrow}"{description}, Rightarrow, draw=none, from=2-2, to=0]
\end{tikzcd}\]
there exists a triple $(m_\sigma, \lambda_\sigma, \rho_\sigma)$, unique up to unique invertible 2-cell, such that $ \sigma$ decomposes as the following pasting
\[\begin{tikzcd}[sep=large]
	B & {U(A_f)} & {U(A)} \\
	& {U(A')}
	\arrow[""{name=0, anchor=center, inner sep=0}, "g"', from=1-1, to=2-2]
	\arrow[""{name=1, anchor=center, inner sep=0}, "{U(u)}"', from=2-2, to=1-3]
	\arrow["{n_f}"{description}, from=1-1, to=1-2]
	\arrow["{U(u_f)}"{description}, from=1-2, to=1-3]
	\arrow[""{name=2, anchor=center, inner sep=0}, "f", bend left=30, from=1-1, to=1-3]
	\arrow["{m_\sigma}"{description}, from=1-2, to=2-2]
	\arrow["{\nu_f \atop \simeq}"{description}, Rightarrow, draw=none, from=2, to=1-2]
	\arrow["{\nu_\sigma \Downarrow}"{description}, Rightarrow, draw=none, from=0, to=1-2]
	\arrow["{\Downarrow U(\omega_\sigma) }"{description}, Rightarrow, draw=none, from=1, to=1-2]
\end{tikzcd}\]
\end{proposition}

\begin{proof}
Apply the property of the lax generic part of $f$ to get an lax diagonalization of the lax square
\[\begin{tikzcd}
	B & {U(A')} \\
	{U(A_f)} & {U(A)}
	\arrow[""{name=0, anchor=center, inner sep=0}, "{n_f}"', from=1-1, to=2-1]
	\arrow[""{name=1, anchor=center, inner sep=0}, "g", from=1-1, to=1-2]
	\arrow[""{name=2, anchor=center, inner sep=0}, "{U(u)}", from=1-2, to=2-2]
	\arrow[""{name=3, anchor=center, inner sep=0}, "{U(u_f)}"', from=2-1, to=2-2]
	\arrow["f"{description}, from=1-1, to=2-2]
	\arrow["{\sigma \Uparrow}"{description}, shift left=1, Rightarrow, draw=none, from=2, to=1]
	\arrow["{\nu_f \atop \simeq}"{description}, shift left=1, Rightarrow, draw=none, from=0, to=3]
\end{tikzcd}\]
\end{proof}

\begin{lemma}\label{post composing in the range of U preserves the generic part of laxfact}
Let be $ f : B \rightarrow U(A)$ and $ u : A \rightarrow A'$ in $\mathcal{A}$. Then $f$ and $ U(u)f$ have the same lax generic part up to a unique equivalence.
\end{lemma}

\begin{proof}
The lax generic factorization of the composite produces a factorization of the following invertible 2-cell by genericity of $n_{U(u)f}$
\[\begin{tikzcd}[sep=large]
	B & {U(A)} \\
	{U(A_{U(u)f})} & {U(A')}
	\arrow[""{name=0, anchor=center, inner sep=0}, "f", from=1-1, to=1-2]
	\arrow[""{name=1, anchor=center, inner sep=0}, "{U(u)}", from=1-2, to=2-2]
	\arrow[""{name=2, anchor=center, inner sep=0}, "{n_{U(u)f}}"', from=1-1, to=2-1]
	\arrow[""{name=3, anchor=center, inner sep=0}, "{U(u_{U(u)f})}"', from=2-1, to=2-2]
	\arrow["{U(w_{\nu_{U(u)f}})}"{description}, from=2-1, to=1-2]
	\arrow["{\nu_{\nu_{U(u)f}} \atop \simeq}"{description}, Rightarrow, draw=none, from=0, to=2]
	\arrow["{U(\omega_{U(u)f}) \atop \simeq}"{description}, Rightarrow, draw=none, from=1, to=3]
\end{tikzcd}\]
This provides an invertible, generic lax 2-cell factorizing $ f$ through the generic $ n_{U(u)f}$. Hence we have two generic factorizations of $f$ related by an invertible 2-cell
\[\begin{tikzcd}
	B & {U(A_{U(u)f})} \\
	{U(A_f)} & {U(A)}
	\arrow["{n_{U(u)f}}", from=1-1, to=1-2]
	\arrow["{U(w_{\nu_{U(u)f}})}", from=1-2, to=2-2]
	\arrow["{U(u_f)}"', from=2-1, to=2-2]
	\arrow["{n_f}"', from=1-1, to=2-1]
	\arrow["{\nu_{\nu_{U(u)f}} \nu_f^{-1} \atop \simeq}"{description}, draw=none, from=1-1, to=2-2]
\end{tikzcd}\]
which entails equivalence by \cref{uniqueness of generic factorization}.
\end{proof}

The following observation is immediate from the property of the lax generic part of the codomain arrow:

\begin{proposition}
For any 2-cell as below
\[\begin{tikzcd}
	B && {U(A)}
	\arrow[""{name=0, anchor=center, inner sep=0}, "{f_1}", start anchor=40, bend left=30, from=1-1, to=1-3]
	\arrow[""{name=1, anchor=center, inner sep=0}, "{f_2}"', start anchor=-40, bend right=30, from=1-1, to=1-3]
	\arrow["{\Downarrow \sigma}"{description}, Rightarrow, draw=none, from=0, to=1]
\end{tikzcd}\]
the generic factorizations of $f_1 $ and $ f_2$ are related by the following decomposition of $ \sigma$
\[\begin{tikzcd}
	& {U(A_1)} \\
	B && {U(A)} \\
	& {U(A_2)}
	\arrow[""{name=0, anchor=center, inner sep=0}, "{n_{A_1}}", from=2-1, to=1-2]
	\arrow[""{name=1, anchor=center, inner sep=0}, "{n_{A_2}}"', from=2-1, to=3-2]
	\arrow[""{name=2, anchor=center, inner sep=0}, "{U(u_{f_1})}", from=1-2, to=2-3]
	\arrow[""{name=3, anchor=center, inner sep=0}, "{U(u_{f_2})}"', from=3-2, to=2-3]
	\arrow["{U(m_\sigma)}"{description}, from=1-2, to=3-2]
	\arrow["{\Downarrow\nu_{\sigma} }"{description}, Rightarrow, draw=none, from=0, to=1]
	\arrow["{\Downarrow U(\omega_\sigma) }"{description}, Rightarrow, draw=none, from=2, to=3]
\end{tikzcd}\]

Moreover, for any composable 2-cells as below
\[\begin{tikzcd}
	B && {U(A)}
	\arrow[""{name=0, anchor=center, inner sep=0}, "{f_2}"{description}, from=1-1, to=1-3]
	\arrow[""{name=1, anchor=center, inner sep=0}, "{f_1}", start anchor=40, bend left = 35, from=1-1, to=1-3]
	\arrow[""{name=2, anchor=center, inner sep=0}, "{f_3}"', start anchor=-40, bend right = 35, from=1-1, to=1-3]
	\arrow["{\Downarrow \sigma}"{description}, Rightarrow, draw=none, from=1, to=0]
	\arrow["{\Downarrow \sigma'}"{description}, Rightarrow, draw=none, from=0, to=2]
\end{tikzcd}\]
we have the following relations between the generic data
\[  w_{\sigma'\sigma} \simeq w_{\sigma'}w_{\sigma} \hskip1cm
\nu_{\sigma'\sigma} =  \nu_{\sigma'}w_{\sigma'} * \nu_{\sigma} \hskip1cm
\omega_{\sigma'\sigma} =  \omega_{\sigma'}w_{\sigma'} * \omega_{\sigma}\]
\end{proposition}

\begin{proof}
The first item is immediate by the property of the lax generic part of the lax generic factorization. For the second item, observe that we can find two alternative factorizations of the following square 
\[\begin{tikzcd}
	B & {U(A_{f_3})} \\
	{U(A_{f_1})} & {U(A)}
	\arrow["{n_{f_1}}"', from=1-1, to=2-1]
	\arrow["{n_{f_3}}", from=1-1, to=1-2]
	\arrow["{U(u_{f_3})}", from=1-2, to=2-2]
	\arrow["{U(u_{f_1})}"', from=2-1, to=2-2]
	\arrow["{\nu_{f_3}^{-1}\sigma'\sigma\nu_{f_1} \Uparrow}"{description}, draw=none, from=1-1, to=2-2]
\end{tikzcd}\]
(which we shall denote abusively as $\sigma'\sigma$ for concision) provided by 
\[\begin{tikzcd}
	B && {U(A_{f_3})} \\
	\\
	{U(A_{f_1})} && {U(A)}
	\arrow["{n_{f_1}}"', from=1-1, to=3-1]
	\arrow["{n_{f_3}}", from=1-1, to=1-3]
	\arrow["{U(u_{f_3})}", from=1-3, to=3-3]
	\arrow["{U(u_{f_1})}"', from=3-1, to=3-3]
	\arrow[""{name=0, anchor=center, inner sep=0}, "{U(w_{\sigma'\sigma})}"{description}, from=3-1, to=1-3]
	\arrow["{\Uparrow \nu_{\sigma'\sigma}}"{description}, Rightarrow, draw=none, from=1-1, to=0]
	\arrow["{U(\omega_{\sigma'\sigma}) \Uparrow}"{description}, Rightarrow, draw=none, from=0, to=3-3]
\end{tikzcd} = \begin{tikzcd}
	B && {U(A_{f_3})} \\
	& {U(A_{f_2})} \\
	{U(A_{f_1})} && {U(A)}
	\arrow[""{name=0, anchor=center, inner sep=0}, "{n_{f_1}}"', from=1-1, to=3-1]
	\arrow[""{name=1, anchor=center, inner sep=0}, "{n_{f_3}}", from=1-1, to=1-3]
	\arrow[""{name=2, anchor=center, inner sep=0}, "{U(u_{f_3})}", from=1-3, to=3-3]
	\arrow[""{name=3, anchor=center, inner sep=0}, "{U(u_{f_1})}"', from=3-1, to=3-3]
	\arrow["{n_{f_2}}"{description}, from=1-1, to=2-2]
	\arrow["{U(w_{\sigma})}"{description}, from=3-1, to=2-2]
	\arrow["{U(w_{\sigma'})}"{description}, from=2-2, to=1-3]
	\arrow["{U(u_{f_2})}"{description}, from=2-2, to=3-3]
	\arrow["{\nu_{\sigma} \Uparrow}"{description}, Rightarrow, draw=none, from=0, to=2-2]
	\arrow["{\nu_{\sigma'} \Uparrow}"{description}, Rightarrow, draw=none, from=1, to=2-2]
	\arrow["{U(\omega_{\sigma'}) \Uparrow}"{description}, Rightarrow, draw=none, from=2, to=2-2]
	\arrow["{U(\omega_{\sigma}) \Uparrow}"{description}, Rightarrow, draw=none, from=3, to=2-2]
\end{tikzcd}\]
By the universal condition in the lax generic property of $n_{f_3}$, we know there exist two 2-cells $ \xi$ and $\zeta$ in $\mathcal{A}$ related by the mutual factorizations below
\[\begin{tikzcd}[sep=large]
	& {U(A_{f_3})} \\
	B & {U(A_{f_2})} \\
	& {U(A_{f_1})}
	\arrow[""{name=0, anchor=center, inner sep=0}, "{n_{f_3}}", from=2-1, to=1-2]
	\arrow["{n_{f_2}}"{description}, from=2-1, to=2-2]
	\arrow[""{name=1, anchor=center, inner sep=0}, "{n_{f_1}}"', from=2-1, to=3-2]
	\arrow["{U(w_{\sigma'})}"{description}, from=2-2, to=1-2]
	\arrow["{U(w_{\sigma})}"{description}, from=3-2, to=2-2]
	\arrow[""{name=2, anchor=center, inner sep=0}, "{U(w_{\sigma'\sigma})}"', curve={height=40pt}, from=3-2, to=1-2]
	\arrow["{\nu_{\sigma'} \Uparrow}"{description}, Rightarrow, draw=none, from=0, to=2-2]
	\arrow["{\nu_\sigma \Uparrow}"{description}, Rightarrow, draw=none, from=1, to=2-2]
	\arrow["{U(\xi)}"'{pos=0.7}, shorten <=3pt, Rightarrow, from=2, to=2-2]
\end{tikzcd} = \begin{tikzcd}
	& {U(A_{f_3})} \\
	B & {U(A_{f_2})} \\
	& {U(A_{f_1})}
	\arrow[""{name=0, anchor=center, inner sep=0}, "{n_{f_3}}", from=2-1, to=1-2]
	\arrow[""{name=1, anchor=center, inner sep=0}, "{n_{f_1}}"', from=2-1, to=3-2]
	\arrow["{U(w_{\sigma'})}"', from=2-2, to=1-2]
	\arrow["{U(w_{\sigma})}"', from=3-2, to=2-2]
	\arrow[from=2-1, to=2-2]
	\arrow["{\nu_{\sigma'} \Uparrow}"{description}, Rightarrow, draw=none, from=2-2, to=0]
	\arrow["{\nu_{\sigma} \Uparrow}"{description}, Rightarrow, draw=none, from=1, to=2-2]
\end{tikzcd} \]
\[\begin{tikzcd}
	& {U(A_{f_3})} \\
	B && {U(A_{f_2})} \\
	& {U(A_{f_1})}
	\arrow["{n_{f_3}}", from=2-1, to=1-2]
	\arrow["{n_{f_1}}"', from=2-1, to=3-2]
	\arrow["{U(w_{\sigma'})}"', from=2-3, to=1-2]
	\arrow["{U(w_{\sigma})}"', from=3-2, to=2-3]
	\arrow[""{name=0, anchor=center, inner sep=0}, "{U(w_{\sigma'\sigma})}"{description, pos=0.3}, from=3-2, to=1-2]
	\arrow["{U(\zeta)} \atop \Leftarrow", shorten <=5pt, draw=none, from=0, to=2-3]
	\arrow["{\nu_{\sigma'\sigma} \Uparrow}"{description}, Rightarrow, draw=none, from=2-1, to=0] 
\end{tikzcd}= \begin{tikzcd}
	& {U(A_{f_3})} \\
	B \\
	& {U(A_{f_1})}
	\arrow["{n_{f_3}}", from=2-1, to=1-2]
	\arrow["{n_{f_1}}"', from=2-1, to=3-2]
	\arrow[""{name=0, anchor=center, inner sep=0}, "{U(w_{\sigma'\sigma})}"', from=3-2, to=1-2]
	\arrow["{\nu_{\sigma'\sigma} \Uparrow}"{description}, Rightarrow, draw=none, from=2-1, to=0]
\end{tikzcd}\]
But $ \nu_{\sigma'\sigma}$ is generic, while $  \nu_{\sigma'} U(w_\sigma) *\nu_{\sigma} $ also is generic by the composition condition in the definition of lax generic pseudofunctor. Hence by \cref{mutually factorizing generics} we have the desired equivalences.  
\end{proof}

\subsection{Lax local right biadjoint}

Here we describe the lax version of the notion of local right biadjoint; however one should beware that it \emph{does not} correspond exactly to lax-familial pseudofunctors: we are going to see that factorization of 2-cells is more rigid than for general lax-generic 1-cells, so that globular 2-cells will have to factorize through the same local unit. 

\begin{definition}
A pseudofunctor $ U : \mathcal{A} \rightarrow \mathcal{B}$ is said to be \emph{lax local right biadjoint} if for each object $A$ in $\mathcal{A}$, the restriction at the lax slice at $A$ has a left biadjoint 
\[\begin{tikzcd}
	{\mathcal{A}\Downarrow A} && {\mathcal{B}\Downarrow U(A)}
	\arrow[""{name=0, anchor=center, inner sep=0}, "{U_A}"', bend right=30, from=1-1, to=1-3]
	\arrow[""{name=1, anchor=center, inner sep=0}, "{L_A}"', bend right=30, from=1-3, to=1-1]
	\arrow["\dashv"{anchor=center, rotate=-90}, draw=none, from=1, to=0]
\end{tikzcd}\]
\end{definition}

\begin{remark}
Let us unravel this property: we still have a biadjunction (beware that we do not relax it into a lax-adjunction) at each lax slice, which manifests as the data of natural units and counits as above, except they are now lax 2-cell
\[\begin{tikzcd}
	B & {} & {U(A_f)} \\
	& {U(A)}
	\arrow["{h^A_f}", from=1-1, to=1-3]
	\arrow[""{name=0, anchor=center, inner sep=0}, "f"', from=1-1, to=2-2]
	\arrow[""{name=1, anchor=center, inner sep=0}, "{UL_A(f)}", from=1-3, to=2-2]
	\arrow["{\huge{\eta^A_f \atop \Rightarrow}}"{description}, shorten <=7pt, shorten >=7pt, Rightarrow, draw=none, from=2-2, to=1-2]
\end{tikzcd}
 \hskip1cm 
 \begin{tikzcd}
	{A_{U(u)}} & {} & {A'} \\
	& A
	\arrow["{e^A_u}", from=1-1, to=1-3]
	\arrow[""{name=0, anchor=center, inner sep=0}, "{L_A(U(u))}"', from=1-1, to=2-2]
	\arrow[""{name=1, anchor=center, inner sep=0}, "u", from=1-3, to=2-2]
	\arrow["{\huge{\epsilon^A_u \atop \Rightarrow}}"{description}, Rightarrow, draw=none, from=2-2, to=1-2]
\end{tikzcd}\]

where the unit has the universal property that any 1-cell $ (g,\alpha) : f \rightarrow U(u)$ in $ \mathcal{B}/U(A)$ decomposes as the following pasting 
\[\begin{tikzcd}[sep=large]
	B & {U(A_f)} && {U(A_{U(u)})} & {U(A')} \\
	&& {U(A)}
	\arrow["{UL_A(U(u))}"{description}, from=1-4, to=2-3]
	\arrow["{UL_A(f)}"{description}, from=1-2, to=2-3]
	\arrow["{h^A_f}"{description}, from=1-1, to=1-2]
	\arrow["{U(e^A_u)}"{description}, from=1-4, to=1-5]
	\arrow[""{name=0, anchor=center, inner sep=0}, "{U(L_A(g))}"{description}, from=1-2, to=1-4]
	\arrow[""{name=1, anchor=center, inner sep=0}, "f"', curve={height=12pt}, from=1-1, to=2-3]
	\arrow[""{name=2, anchor=center, inner sep=0}, "{U(u)}"{description}, curve={height=-12pt}, from=1-5, to=2-3]
	\arrow[""{name=3, anchor=center, inner sep=0}, "g", curve={height=-30pt}, from=1-1, to=1-5]
	\arrow["{\eta_f^A \atop \Rightarrow}"{description}, Rightarrow, draw=none, from=1, to=1-2]
	\arrow["{UL_A(\alpha) \atop \Rightarrow}"{description}, Rightarrow, draw=none, from=2-3, to=0]
	\arrow["{U(\epsilon^A_{u}) \atop \Rightarrow}"{description}, Rightarrow, draw=none, from=1-4, to=2]
	\arrow["{(\mathfrak{i}^A_{f,u})_{(g,\alpha)} \atop \simeq}"{description}, Rightarrow, draw=none, from=3, to=0]
\end{tikzcd} = \begin{tikzcd}[sep=small]
	B && {U(A')} \\
	& {U(A)}
	\arrow["f"', from=1-1, to=2-2]
	\arrow["{U(u)}", from=1-3, to=2-2]
	\arrow[""{name=0, anchor=center, inner sep=0}, "g", from=1-1, to=1-3]
	\arrow["{\alpha \atop \Rightarrow}"{description}, Rightarrow, draw=none, from=0, to=2-2]
\end{tikzcd}\]
\end{remark}

\begin{remark}
In the lax context, we can now complete \cref{expression of the BC mate} with a naturality condition relatively to 2-cell. Let be a globular 2-cell $ \omega : u \Rightarrow v$ in $ \mathcal{A}$. We can now define a natural transformation $ \mathcal{A}\Downarrow u \Rightarrow \mathcal{A} \Downarrow v $ whose component at $ w : A' \rightarrow A $ is the triangular 2-cell encoding the whiskering along $\omega$
\[\begin{tikzcd}
	{A'} && {A'} \\
	& {A_2}
	\arrow["uw"', from=1-1, to=2-2]
	\arrow["vw", from=1-3, to=2-2]
	\arrow[""{name=0, anchor=center, inner sep=0}, Rightarrow, no head, from=1-1, to=1-3]
	\arrow["{\omega*w \atop \Rightarrow}"{description}, Rightarrow, draw=none, from=0, to=2-2]
\end{tikzcd}\]
and similarly for the component of $ \mathcal{B}/U(\omega)$. On the other hand pseudofunctoriality of $ U$ produces a composite 2-cell 
\[\begin{tikzcd}[sep=huge]
	{U(A')} && {U(A')} \\
	& {U(A_1)} \\
	& {U(A_2)}
	\arrow["{U(w)}"{description}, from=1-1, to=2-2]
	\arrow["{U(w)}"{description}, from=1-3, to=2-2]
	\arrow[""{name=0, anchor=center, inner sep=0}, curve={height=-18pt}, Rightarrow, no head, from=1-1, to=1-3]
	\arrow[""{name=1, anchor=center, inner sep=0}, "{U(v)}"{description}, curve={height=-12pt}, from=2-2, to=3-2]
	\arrow[""{name=2, anchor=center, inner sep=0}, "{U(u)}"{description}, curve={height=12pt}, from=2-2, to=3-2]
	\arrow[""{name=3, anchor=center, inner sep=0}, "{U(uw)}"', curve={height=18pt}, from=1-1, to=3-2]
	\arrow[""{name=4, anchor=center, inner sep=0}, "{U(1_A)}"{description}, from=1-1, to=1-3]
	\arrow[""{name=5, anchor=center, inner sep=0}, "{U(vw)}", curve={height=-18pt}, from=1-3, to=3-2]
	\arrow["{\alpha_{u,w} \atop \simeq}"{description}, Rightarrow, draw=none, from=3, to=2-2]
	\arrow["{U(1_{w}) \atop \simeq}"{description}, Rightarrow, draw=none, from=4, to=2-2]
	\arrow["{\alpha_A \atop \simeq}"{description}, shorten <=2pt, shorten >=2pt, Rightarrow, from=0, to=4]
	\arrow["{U(\omega) \atop \Rightarrow}"{description}, Rightarrow, draw=none, from=2, to=1]
	\arrow["{\alpha_{v,w} \atop \simeq}"{description}, Rightarrow, draw=none, from=2-2, to=5]
\end{tikzcd}\]
which is the effect of $U_{A_2}$ on the 1-cell $( \mathcal{A}\Downarrow \omega)_w$.\\

Now let be $ f : B \rightarrow U(A_1)$; let us denote as
\[\begin{tikzcd}
	{A_{U(u)f}} && {A_{U(v)f}} \\
	& {A_2}
	\arrow[""{name=0, anchor=center, inner sep=0}, "{L_{A_2}(U(u)f)}"', from=1-1, to=2-2]
	\arrow["{w_{L_{A_2}(U(w))}}", from=1-1, to=1-3]
	\arrow["{L_{A_2}(U(v)f)}"{pos=0.7}, from=1-3, to=2-2]
	\arrow["{\omega_{L_{A_2}(U(w))} \atop \Rightarrow}"{description}, Rightarrow, draw=none, from=0, to=1-3]
\end{tikzcd}\]
the 2-cell in $ \mathcal{A}$ the whiskering along $ U(\omega)$ is sent to by the local left adjoint $ L_{A_2}$. Observe we have the following 2-dimensional equality provided by pseudofunctoriality of $U$
\[\begin{tikzcd}[row sep=large]
	B && {U(A_f)} && {U(A_f)} \\
	& {U(A_1)} \\
	&& {U(A_2)}
	\arrow[""{name=0, anchor=center, inner sep=0}, "f"', from=1-1, to=2-2]
	\arrow[""{name=1, anchor=center, inner sep=0}, "{UL_{A_1}(f)}"{description}, from=1-3, to=2-2]
	\arrow["{h^{A_1}_f}", from=1-1, to=1-3]
	\arrow["{U(u)}"', from=2-2, to=3-3]
	\arrow[""{name=2, anchor=center, inner sep=0}, "{U(uL_{A_1}(f))}"{description, pos=0.6}, from=1-3, to=3-3]
	\arrow[Rightarrow, no head, from=1-3, to=1-5]
	\arrow[""{name=3, anchor=center, inner sep=0}, "{U(vL_{A_1}(f))}", from=1-5, to=3-3]
	\arrow["{\eta^{A_1}_f \atop \Rightarrow}", Rightarrow, draw=none, from=0, to=1]
	\arrow["{\alpha_{u, L_{A_1}(f)} \atop \simeq}", Rightarrow, draw=none, from=2-2, to=2]
	\arrow["{U(\omega *L_{A_1}(f)) \atop \Rightarrow}"{description}, shift left=4, Rightarrow, draw=none, from=2, to=3]
\end{tikzcd}=
\begin{tikzcd}[row sep=large]
	B & B && {U(A_f)} \\
	{U(A_1)} & {U(A_1)} \\
	& {U(A_2)}
	\arrow["f"', from=1-1, to=2-1]
	\arrow["{U(u)}"', from=2-1, to=3-2]
	\arrow["{h^{A_1}_f}", from=1-2, to=1-4]
	\arrow[Rightarrow, no head, from=1-1, to=1-2]
	\arrow[""{name=0, anchor=center, inner sep=0}, Rightarrow, no head, from=2-1, to=2-2]
	\arrow["U(v)"{description}, from=2-2, to=3-2]
	\arrow["f"{description}, from=1-2, to=2-2]
	\arrow[""{name=1, anchor=center, inner sep=0}, "{UL_{A_1}(f)}"{description}, from=1-4, to=2-2]
	\arrow[""{name=2, anchor=center, inner sep=0}, "{U(vL_{A_1}(f))}", curve={height=-12pt}, from=1-4, to=3-2]
	\arrow["{U(\omega) \atop \Rightarrow}"{description}, Rightarrow, draw=none, from=3-2, to=0]
	\arrow["{\alpha_{v, L_{A_1}(f)} \atop \simeq}"{description}, Rightarrow, draw=none, from=2-2, to=2]
	\arrow["{\eta^{A_1}_f \atop \Rightarrow}"{description, pos=0.6}, Rightarrow, draw=none, from=1, to=1-2]
\end{tikzcd}\]
This equality is sent by the pseudofunctor $L_{A_2}$ to an invertible 2-cell in the lax slice over $A_2$
\[\begin{tikzcd}[sep=huge]
	{A_{U(u)f}} & {A_f} & {A_f} \\
	& {A_2}
	\arrow[""{name=0, anchor=center, inner sep=0}, "{L_{A_2}(U(u)f)}"', from=1-1, to=2-2]
	\arrow["{s^u_f}", from=1-1, to=1-2]
	\arrow["{uL_{A_1}(f)}"{description, pos=0.6}, from=1-2, to=2-2]
	\arrow[Rightarrow, no head, from=1-2, to=1-3]
	\arrow[""{name=1, anchor=center, inner sep=0}, "{vL_{A_1}(f)}", from=1-3, to=2-2]
	\arrow["{\sigma^u_f \atop \Rightarrow}"{description}, Rightarrow, draw=none, from=0, to=1-2]
	\arrow["{\omega*L_{A_1}(f) \atop \Rightarrow}"{description}, Rightarrow, draw=none, from=1, to=1-2]
\end{tikzcd}
= \begin{tikzcd}[sep=huge]
	{A_{U(u)f}} & {A_{U(v)f}} & {A_f} \\
	& {A_2}
	\arrow[""{name=0, anchor=center, inner sep=0}, "{L_{A_2}(U(u)f)}"', from=1-1, to=2-2]
	\arrow["{w_{L_{A_2}( U(\omega))}}", from=1-1, to=1-2]
	\arrow["{L_{A_2}(U(v)f)}"{description, pos=0.6}, from=1-2, to=2-2]
	\arrow["{s^v_f}", from=1-2, to=1-3]
	\arrow[""{name=1, anchor=center, inner sep=0}, "{vL_{A_1}(f)}", from=1-3, to=2-2]
	\arrow["{\omega_{L_{A_2}( U(\omega))} \atop \Rightarrow}"{description}, Rightarrow, draw=none, from=0, to=1-2]
	\arrow["{\sigma^v_f \atop \Rightarrow}"{description}, Rightarrow, draw=none, from=1, to=1-2]
\end{tikzcd}\]
One can show this equality to be natural in $f$, producing the 2-dimensional data of the Beck-Chevalley mate
\[\begin{tikzcd}[sep=huge]
	{\mathcal{A}\Downarrow A_1} & {\mathcal{B}/U(A_1)} \\
	{\mathcal{A}\Downarrow A_2} & {\mathcal{B}/U(A_2)}
	\arrow[""{name=0, anchor=center, inner sep=0}, "{\mathcal{A}\Downarrow v}"', curve={height=30pt}, from=1-1, to=2-1]
	\arrow["{L_{A_1}}"', from=1-2, to=1-1]
	\arrow["{L_{A_2}}", from=2-2, to=2-1]
	\arrow[""{name=1, anchor=center, inner sep=0}, "{\mathcal{B}/U(u)}", from=1-2, to=2-2]
	\arrow[""{name=2, anchor=center, inner sep=0}, "{\mathcal{A}\Downarrow u}"{description}, curve={height=-30pt}, from=1-1, to=2-1]
	\arrow["{\sigma^u \atop \Leftarrow}"{description}, Rightarrow, draw=none, from=1, to=2]
	\arrow["{\mathcal{A}\Downarrow \omega \atop \Leftarrow}"{description}, shorten <=4pt, shorten >=4pt, Rightarrow, draw=none, from=2, to=0]
\end{tikzcd}
= \begin{tikzcd}[sep=huge]
	{\mathcal{A}\Downarrow A_1} & {\mathcal{B}/U(A_1)} \\
	{\mathcal{A}\Downarrow A_2} & {\mathcal{B}/U(A_2)}
	\arrow[""{name=0, anchor=center, inner sep=0}, "{\mathcal{A}\Downarrow v}"', from=1-1, to=2-1]
	\arrow["{L_{A_1}}"', from=1-2, to=1-1]
	\arrow["{L_{A_2}}", from=2-2, to=2-1]
	\arrow[""{name=1, anchor=center, inner sep=0}, "{\mathcal{B}/U(v)}"{description}, curve={height=30pt}, from=1-2, to=2-2]
	\arrow[""{name=2, anchor=center, inner sep=0}, "{\mathcal{B}/U(u)}", curve={height=-30pt}, from=1-2, to=2-2]
	\arrow["{\mathcal{B}/U(\omega) \atop \Leftarrow}"{description}, Rightarrow, draw=none, from=2, to=1]
	\arrow["{\sigma^v \atop \Leftarrow}"{description}, Rightarrow, draw=none, from=1, to=0]
\end{tikzcd}\]
\end{remark}

This ``repairs" the incomplete Beck-Chevalley condition of local right biadjoints. Beware however that in general lax-local right biadjoint are not local right bi-adjoints for the local unit 2-cell are not necessarily invertible: such a condition must be stipulated.

\begin{definition}
A lax-local right biadjoint shall be said \emph{coherent} if it restricts to a local right biadjoint on pseudoslices, that is, if for any $A$ the left adjoint $ L_A$ of $U\Downarrow A$ restricts to a left adjoint of $ U/A$ along the inclusion of the pseudoslice, that is, if we have a pseudonatural equivalence
\[\begin{tikzcd}
	{\mathcal{A}/A} & {\mathcal{B}/U(A)} \\
	{\mathcal{A}\Downarrow A} & {\mathcal{B}\Downarrow U(A)}
	\arrow["{\iota_A}"', hook, from=1-1, to=2-1]
	\arrow["{L_A}"', dashed, from=1-2, to=1-1]
	\arrow["{\iota_{U(A)}}", hook, from=1-2, to=2-2]
	\arrow["{L_A}", from=2-2, to=2-1]
	\arrow["\simeq"{description}, draw=none, from=2-2, to=1-1]
\end{tikzcd}\]
\end{definition}

As the pseudoslice contains all objects of the lax-slice, this ensures that all local units $ (h^A_f, \eta^A_f)$ actually lie in the pseudoslices, that is, that the 2-cell $ \eta^A_f$ are invertible for each $ f: B \rightarrow U(A)$ at each $A$.  

\begin{proposition}
If $U$ is a coherent lax-local right biadjoint, and $ n : B \rightarrow U(A)$ is such that $ n \simeq h^A_n$. Then $ n$ is lax-generic, but moreover, in a lax square as below
\[\begin{tikzcd}
	B & {U(A_1)} \\
	{U(A)} & {U(A_2)}
	\arrow["n"', from=1-1, to=2-1]
	\arrow["{U(v)}"', from=2-1, to=2-2]
	\arrow["g", from=1-1, to=1-2]
	\arrow["{U(u)}", from=1-2, to=2-2]
	\arrow["{\Uparrow \sigma}"{description}, draw=none, from=2-1, to=1-2]
\end{tikzcd}\]
the left part lax diagonalization is invertible.
\end{proposition}

\begin{proof}
A lax square as above is like a lax-cell
\[\begin{tikzcd}
	B && {U(A_1)} \\
	& {U(A_2)}
	\arrow["g", from=1-1, to=1-3]
	\arrow[""{name=0, anchor=center, inner sep=0}, "{U(v)n}"', from=1-1, to=2-2]
	\arrow[""{name=1, anchor=center, inner sep=0}, "{U(u)}", from=1-3, to=2-2]
	\arrow["{\sigma \atop \Rightarrow}", Rightarrow, draw=none, from=0, to=1]
\end{tikzcd}\]
which factorizes by bi-adjointness as a pasting 
\[\begin{tikzcd}[sep=huge]
	B & {U(A_{U(v)n})} & {U(A_1)} \\
	& {U(A_2)}
	\arrow[""{name=0, anchor=center, inner sep=0}, "g", bend left=30, start anchor=45, from=1-1, to=1-3]
	\arrow[""{name=1, anchor=center, inner sep=0}, "{U(v)n}"', from=1-1, to=2-2]
	\arrow[""{name=2, anchor=center, inner sep=0}, "{U(u)}", from=1-3, to=2-2]
	\arrow["{h^{A_2}_{U(v)n}}"{description}, from=1-1, to=1-2]
	\arrow["{U(w_{(g,\sigma)})}"{description}, from=1-2, to=1-3]
	\arrow[from=1-2, to=2-2]
	\arrow["{\mathfrak{i}_{(g,\sigma)} \atop \simeq}"{description}, Rightarrow, draw=none, from=0, to=1-2]
	\arrow["{\eta^{A_2}_{U(v)n} \atop \simeq}"{description}, Rightarrow, draw=none, from=1, to=1-2]
	\arrow["{U(L_{A_2}(\omega_{(g,\sigma)})) \atop \Rightarrow}"{description}, Rightarrow, draw=none, from=1-2, to=2]
\end{tikzcd}\]
But now, for the factorization through local unit is essentially unique, this provides an invertible 2-cell 
\[\begin{tikzcd}[sep=huge]
	B && {U(A_1)} \\
	{U(A)} && {U(A_2)}
	\arrow["n"', from=1-1, to=2-1]
	\arrow["g", from=1-1, to=1-3]
	\arrow["U(w_{(g,\sigma)})"{description}, ""{name=0, anchor=center, inner sep=0}, from=2-1, to=1-3]
	\arrow["{U(u)}", from=1-3, to=2-3]
	\arrow["{U(v)}"', from=2-1, to=2-3]
	\arrow["{\eta^{A_2}_{U(v)n})\mathfrak{i}_{(g,\sigma)} \atop \simeq}"{description, pos=0.6}, Rightarrow, draw=none, from=0, to=1-1]
	\arrow["{ \scriptsize{\Uparrow U(L_{A_2}(\omega_{(g,\sigma)}))}}"{description}, Rightarrow, draw=none, from=0, to=2-3]
\end{tikzcd}\]
\end{proof}

\begin{remark}
Hence any 2-cell $\sigma$ factorizes as below
\[\begin{tikzcd}[sep=large]
	B & {U(A_\sigma)} && {U(A)}
	\arrow["{\eta^A_\sigma}"{description}, from=1-1, to=1-2]
	\arrow["U(L_A(f_1))"{description, name=0, anchor=center, inner sep=0}, curve={height=-18pt}, from=1-2, to=1-4]
	\arrow["U(L_A(f_2))"{description, name=1, anchor=center, inner sep=0}, curve={height=18pt}, from=1-2, to=1-4]
	\arrow[""{name=2, anchor=center, inner sep=0}, "{f_1}", shift left=2, curve={height=-28pt}, from=1-1, to=1-4]
	\arrow[""{name=3, anchor=center, inner sep=0}, "{f_2}"', shift right=2, curve={height=28pt}, from=1-1, to=1-4]
	\arrow["{\Downarrow U(\omega_\sigma)}"{description}, Rightarrow, draw=none, from=0, to=1]
	\arrow["{\eta^A_{f_1} \atop \simeq}"'{pos=0.9}, Rightarrow, draw=none, from=2, to=1-2]
	\arrow["{\eta^A_{f_2} \atop \simeq}"{pos=0.9}, Rightarrow, draw=none, from=3, to=1-2]
\end{tikzcd}\]
This is a more rigid factorization than in lax-familial pseudofunctors. It is not known which lax-version of local adjointness they correspond to. This is neither the kind of factorization we get as examples.
\end{remark}

\section{Bifactorization systems}

Bistable pseudofunctors involve both orthogonality and factorizations conditions. In fact, in the 1-dimensional setting, one can extract factorization data from stable functors; however in this work, for our purpose, we are more interested in a converse process: constructing bistable inclusion from the 2-dimensional analogs of bifactorization systems.

\subsection{Bi-orthogonality}

Let first recall some 2-dimensional notions. For $ f : A \rightarrow B$ and $g : C \rightarrow D$ 1-cells in $\mathcal{C}$, a pseudocommutative square in $f,g$ is the data of a triple $ (u,v,\alpha) : f \Rightarrow g$ whith $ \alpha : gu \stackrel{\simeq }{\Rightarrow} vf$ an invertible 2-cell as below, 
\[ 
\begin{tikzcd}
A \arrow[d, "f"'] \arrow[r, "u"] \arrow[rd, "\alpha \atop \simeq", phantom] & C \arrow[d, "g"] \\
B \arrow[r, "v"']                                                           & D         
\end{tikzcd} \]
A morphism of pseudosquares $ (\phi, \psi) : (u, v, \alpha) \Rightarrow (u', v', \alpha')  $ is the data of $ \phi : u \Rightarrow u' $ and $\psi : v \Rightarrow v'$ such that $ \alpha' g^*\phi = \psi^*f \alpha  $ as below
\[ 
\begin{tikzcd}[column sep= large, row sep= large]
A \arrow[r, "u'"', bend right, ""{name=U1, inner sep=0.1pt}] \arrow[d, "f"'] \arrow[r, "u", bend left, ""{name=D1, inner sep=0.1pt, below}] & C \arrow[d, "g"] \\
B      \arrow[r, "v", bend left, ""{name=U2, inner sep=0.1pt, below}] \arrow[r, "v'"', bend right, ""{name=D2, inner sep=0.1pt}]              & D  \arrow[Rightarrow, from=D1, to=U1]{}{\phi}   \arrow[Rightarrow, from=U2, to=D2]{}{\psi}              
\end{tikzcd} \quad \textrm{ with } 
\begin{tikzcd}[column sep= large, row sep= large]
A \arrow[r, "u'"', bend right, ""{name=U1, inner sep=0.1pt}] \arrow[d, "f"'] \arrow[r, "u", bend left, ""{name=D1, inner sep=0.1pt, below}] & C \arrow[d, "g"] \\
B  \arrow[r, "v'"', ""{name=D2, inner sep=0.1pt}]                 & D  \arrow[Rightarrow, from=D1, to=U1]{}{\phi}   \arrow[phantom, from=U1, to=U2]{}[near end]{\alpha' \atop \simeq} 
\end{tikzcd}  = \begin{tikzcd}[column sep= large, row sep= large]
A \arrow[d, "f"'] \arrow[r, "u", ""{name=D1, inner sep=0.1pt}] & C \arrow[d, "g"] \\
B \arrow[r, "v", bend left, ""{name=U2, inner sep=0.1pt, below}] \arrow[r, "v'"', bend right, ""{name=D2, inner sep=0.1pt}]   & D  \arrow[phantom, from=D1, to=U2]{}[]{\alpha \atop \simeq}   \arrow[Rightarrow, from=U2, to=D2]{}{\psi}
\end{tikzcd} 
\]

\begin{definition}
A \emph{filler}\index{filler} for a pseudosquare $ (u,v,\alpha)$ between $f,g$ is the data of a diagonal map $ s: B \rightarrow C $ and a pair of invertible natural transformations $ \lambda: sf \stackrel{\simeq}{\rightarrow} u$ and $ \rho: gs \stackrel{\simeq}{\rightarrow} v$
\[\begin{tikzcd}[sep=large]
	A & C \\
	B & D
	\arrow["u", from=1-1, to=1-2]
	\arrow["f"', from=1-1, to=2-1]
	\arrow["v"', from=2-1, to=2-2]
	\arrow["g"', from=1-2, to=2-2]
	\arrow[""{name=0, anchor=center, inner sep=0}, "s"{description}, dashed, from=2-1, to=1-2]
	\arrow["{\lambda\atop \simeq}"{description}, Rightarrow, draw=none, from=1-1, to=0]
	\arrow["{\rho \atop \simeq}"{description}, Rightarrow, draw=none, from=2-2, to=0]
\end{tikzcd}\]
while a morphism of filler $ (s, \lambda,\rho) \rightarrow (s', \lambda', \rho')$ is the data of a 2-cell $ \sigma: s \Rightarrow s'$ such that $ \lambda' \sigma^*f = \lambda  $ and $ \rho' g^*\sigma = \rho  $ as visualized below:

\[\begin{tikzcd}[sep=large]
	A & C \\
	B
	\arrow["u", from=1-1, to=1-2]
	\arrow["f"', from=1-1, to=2-1]
	\arrow[""{name=0, anchor=center, inner sep=0}, "s"{description}, from=2-1, to=1-2]
	\arrow[""{name=1, anchor=center, inner sep=0}, "{s'}"', curve={height=20pt}, from=2-1, to=1-2]
	\arrow["{\lambda\atop \simeq}"{description}, Rightarrow, draw=none, from=1-1, to=0]
	\arrow["\sigma", shorten <=2pt, shorten >=2pt, Rightarrow, from=0, to=1]
\end{tikzcd} \,= \, \begin{tikzcd}[sep=large]
	A & C \\
	B
	\arrow["u", from=1-1, to=1-2]
	\arrow["f"', from=1-1, to=2-1]
	\arrow[""{name=0, anchor=center, inner sep=0}, "{s'}"', from=2-1, to=1-2]
	\arrow["{\lambda'\atop \simeq}"{description}, Rightarrow, draw=none, from=1-1, to=0]
\end{tikzcd} \quad \quad \begin{tikzcd}[sep=large]
	& C \\
	B & D
	\arrow["v"', from=2-1, to=2-2]
	\arrow["g"', from=1-2, to=2-2]
	\arrow[""{name=0, anchor=center, inner sep=0}, "{s'}"{description}, from=2-1, to=1-2]
	\arrow[""{name=1, anchor=center, inner sep=0}, "s", curve={height=-20pt}, from=2-1, to=1-2]
	\arrow["{\rho' \atop \simeq}"{description}, Rightarrow, draw=none, from=2-2, to=0]
	\arrow["\sigma", shorten <=2pt, shorten >=2pt, Rightarrow, from=1, to=0]
\end{tikzcd} \, = \, \begin{tikzcd}[sep=large]
	& C \\
	B & D
	\arrow["v"', from=2-1, to=2-2]
	\arrow["g"', from=1-2, to=2-2]
	\arrow[""{name=0, anchor=center, inner sep=0}, "s", from=2-1, to=1-2]
	\arrow["{\rho \atop \simeq}"{description}, Rightarrow, draw=none, from=2-2, to=0]
\end{tikzcd} \]

In particular, for any such $ \sigma$ the whiskerings $ \sigma^*f$ and $ g^*\sigma$ necessarily are invertible by cancellation of invertible cells. 
\end{definition}

 \begin{definition}
 We say that $ f\perp g$, or that $f$ and $g$ are respectively \emph{left and right bi-orthogonal}\index{bi-orthogonal} if the following square is a bipullback in $Cat$
\[ 
\begin{tikzcd}
{\mathcal{C}[B,C]} \arrow[r, "{\mathcal{C}[f,C]}"] \arrow[d, "{\mathcal{C}[B,g]}"'] \arrow[rd, "\simeq", phantom] & {\mathcal{C}[A,C]} \arrow[d, "{\mathcal{C}[A,g]}"] \\
{\mathcal{C}[B,D]} \arrow[r, "{\mathcal{C}[f,D]}"']                                                               & {\mathcal{C}[A,D]}                                
\end{tikzcd} \]
 \end{definition}
\begin{remark}
This just means that \begin{itemize}
    \item for any pseudocommutative square $ (u, v, \alpha)$ there is an universal filler $(s_\alpha, \lambda_\alpha, \rho_{\alpha}) $ 
\[\begin{tikzcd}[sep=large]
	A & C \\
	B & D
	\arrow["u", from=1-1, to=1-2]
	\arrow["f"', from=1-1, to=2-1]
	\arrow["v"', from=2-1, to=2-2]
	\arrow["g"', from=2-2, to=1-2]
	\arrow[""{name=0, anchor=center, inner sep=0}, "s_\alpha"{description}, dashed, from=2-1, to=1-2]
	\arrow["{\lambda_\alpha\atop \simeq}"{description, pos=0.3}, Rightarrow, draw=none, from=1-1, to=0]
	\arrow["{\rho_\alpha \atop \simeq}"{description, pos=0.3}, Rightarrow, draw=none, from=2-2, to=0]
\end{tikzcd}\]
whith the property that any other filler $(s, \lambda, \rho)$ comes equipped with a unique, invertible morphism of fillers $ \omega : s \stackrel{\simeq}{\Rightarrow} s_\alpha$
\item and for a morphism of pseudosquares $ (\phi, \psi) : (u, v, \alpha) \Rightarrow (u', v', \alpha')  $ there exists a unique 2-cell $ \sigma : s_{\alpha} \Rightarrow s_{\alpha'} $ such that 
\[ \alpha' \phi = \sigma^*f \alpha \quad \textrm{ and } \quad \alpha' g^*\sigma = \psi \alpha \]
\[ \begin{tikzcd}[column sep= large, row sep= large]
A \arrow[r, "u'"', bend right, ""{name=U1, inner sep=0.1pt}] \arrow[d, "f"'] \arrow[r, "u", bend left, ""{name=D1, inner sep=0.1pt, below}] & C \arrow[d, equal] \\
B  \arrow[r, "s'"', ""{name=D2, inner sep=0.1pt}]                 & C  \arrow[Rightarrow, from=D1, to=U1]{}{\phi}   \arrow[phantom, from=U1, to=U2]{}[near end]{\alpha' \atop \simeq} 
\end{tikzcd}  
=
\begin{tikzcd}[column sep= large, row sep= large]
A \arrow[d, "f"'] \arrow[r, "u", ""{name=D1, inner sep=0.1pt}] & C \arrow[d, equal] \\
B \arrow[r, "s", bend left, ""{name=U2, inner sep=0.1pt, below}] \arrow[r, "s'"', bend right, ""{name=D2, inner sep=0.1pt}]   & C  \arrow[phantom, from=D1, to=U2]{}[]{\alpha \atop \simeq}   \arrow[Rightarrow, from=U2, to=D2]{}{\sigma}
\end{tikzcd} 
\textrm{ and }
 \begin{tikzcd}[column sep= large, row sep= large]
B \arrow[r, "s'"', bend right, ""{name=U1, inner sep=0.1pt}] \arrow[d, equal] \arrow[r, "s", bend left, ""{name=D1, inner sep=0.1pt, below}] & C \arrow[d, "g"] \\
B  \arrow[r, "s'"', ""{name=D2, inner sep=0.1pt}]                 & D  \arrow[Rightarrow, from=D1, to=U1]{}{\sigma}   \arrow[phantom, from=U1, to=U2]{}[near end]{\alpha' \atop \simeq} 
\end{tikzcd}  
=
\begin{tikzcd}[column sep= large, row sep= large]
B \arrow[d, equal] \arrow[r, "s", ""{name=D1, inner sep=0.1pt}] & C \arrow[d, "g"] \\
B \arrow[r, "v", bend left, ""{name=U2, inner sep=0.1pt, below}] \arrow[r, "v'"', bend right, ""{name=D2, inner sep=0.1pt}]   & C  \arrow[phantom, from=D1, to=U2]{}[]{\alpha \atop \simeq}   \arrow[Rightarrow, from=U2, to=D2]{}{\psi}
\end{tikzcd} 
\]
\end{itemize}
Actually those conditions are synthetised by the existence of an equivalence of categories \[ \mathcal{C}[B,C] \simeq \textbf{ps}[2,\mathcal{C}](f,g) \]
sending any arrow $ s : B \rightarrow C$ to the canonical isomorphisms induced by composition with $ f$ and $g$ $ (sf, gs, \alpha_{sf}\alpha_{gs})$, whith the choice of universal filler as pseudoinverse. This says that a universal filler exists up to unique invertible 2-cell, and that any such filler is universal. 
\end{remark}

\subsection{Bifactorization systems}

\begin{definition}
A \emph{bifactorization system}\index{bifactorization system} on $ \mathcal{C}$ is the data of a bi-orthogonality structure $ (\mathcal{L}, \mathcal{R})$ such that for any $ f : A \rightarrow B$ in $\mathcal{C}$, there exists a pair $(l_f, r_f)$ with $ l_f \in \mathcal{L}$ and $ r_f \in \mathcal{R}$ equipped with an invertible 2-cell 
        \[ 
\begin{tikzcd}
A \arrow[rr, "f"] \arrow[rd, "l_f"'] & {} \arrow[d, "\alpha_f \atop \simeq", phantom, near start] & B \\
                                     & C_f \arrow[ru, "r_f"']                         &  
\end{tikzcd} \] 
such that, for any other factorization $ \alpha : rl \simeq f$ with $ l_f \in \mathcal{L}$ and $ r_f \in \mathcal{R}$, there exist a unique equivalence $i$ and a pair of invertible 2-cells $ \lambda : il \simeq l_f $ and $ \rho : r \simeq r_f i$ such that $ \alpha = \alpha_f \rho^*l r^*\lambda$ as below 
\[ 
\begin{tikzcd}[row sep= small, column sep=large]
A \arrow[rr, "f"] \arrow[rd, "l_f" description] \arrow[rdd, "l"', bend right=20, "\lambda \atop \simeq"] & {} \arrow[d, "\alpha_f \atop \simeq", phantom]         & B \\
                                                                  & C_f \arrow[ru, "r_f" description]                      &   \\
                                                                  & C \arrow[ruu, "r"', bend right=20, "\rho \atop \simeq"] \arrow[u, "i"] &  
\end{tikzcd} \]
    \end{definition}
    
\pagebreak
We recall the following properties of 2-factorization systems:

\begin{proposition}
 If $\mathcal{C}$ is a 2-category endowed with a bifactorization system $ (\mathcal{L}, \mathcal{R})$. Then the left and right classes enjoy the following properties :
 \begin{multicols}{2}
 \begin{itemize}
    \item $\mathcal{L} $ is closed under composition
    \item $\mathcal{L}$ is closed under invertible 2-cell
    \item $\mathcal{L}$ contains all equivalence
    \item $\mathcal{L}$ is right-pseudocancellative\index{right-pseudocancellative}: for any invertible 2-cell
\[\begin{tikzcd}
	{C_1} & {C_2} \\
	& {C_3}
	\arrow["{l_2}"', from=1-1, to=2-2]
	\arrow["{l_1}", from=1-1, to=1-2]
	\arrow[""{name=0, anchor=center, inner sep=0}, "f", from=1-2, to=2-2]
	\arrow["\simeq"{description}, Rightarrow, draw=none, from=0, to=1-1]
\end{tikzcd}\]
    with $l_1, \; l_2 $ in $\mathcal{L}$, then $f$ also is in $\mathcal{L}$
    \item  $\mathcal{L}$ is closed under bicolimits in $ \ps[{2, \mathcal{C}}]$
    \item $\mathcal{L}$ is closed under bipushout along arbitrary maps
 \end{itemize} 
      
 \columnbreak
  \begin{itemize}
      \item $\mathcal{R} $ is closed under composition
      \item $\mathcal{R}$ is closed under invertible 2-cell
      \item $\mathcal{R}$ contains all equivalence
      \item $\mathcal{R} $ is left-pseudocancellative\index{left-pseudocancellative}: for any invertible 2-cell
\[\begin{tikzcd}
	{C_1} & {C_3} \\
	{C_2}
	\arrow["{r_1}", from=1-1, to=1-2]
	\arrow["f"', from=1-1, to=2-1]
	\arrow[""{name=0, anchor=center, inner sep=0}, "{r_2}"', from=2-1, to=1-2]
	\arrow["\simeq"{description}, Rightarrow, draw=none, from=0, to=1-1]
\end{tikzcd}\]
      with $r_1, \; r_2 $ in $\mathcal{R}$, then $f$ also is in $\mathcal{R}$
      \item $\mathcal{R} $ is closed under bilimits in $ \ps[{2, \mathcal{C}}]$
      \item $ \mathcal{R}$ is closed under bipullback along arbitrary maps
  \end{itemize}       
 \end{multicols}
\end{proposition}

From now on, we fix a factorization system on $ \mathcal{C}$.

\begin{proposition}
Let be $ f,g$ two arrows in $ \mathcal{C}$ and $ (u,v,\alpha)$ a pseudocommutative square between them; then there exist an arrow $ w_\alpha : C_f \rightarrow C_g$ and two  pseudocommutative squares $(l_f, l_g, \lambda_\alpha) $, $(r_f, r_g, \rho_\alpha)$ which are unique up to unique invertible 2-cell, such that $ \alpha_g \alpha = \rho_\alpha^*l_\alpha r_g^*\lambda_\alpha \alpha_f $ as below
\[ 
\begin{tikzcd}
                                                                                                                    & {} \arrow[d, "\alpha_f \atop \simeq",  near end, phantom]                                                                 &                  \\
A \arrow[r, "l_f"] \arrow[d, "u"'] \arrow[rd, "\lambda_\alpha \atop \simeq", phantom] \arrow[rr, "f", bend left=49] & C_f \arrow[r, "r_f"] \arrow[d, "w_\alpha" description, dashed] \arrow[rd, "\rho_\alpha \atop \simeq", phantom] & B \arrow[d, "v"] \\
C \arrow[r, "l_g"'] \arrow[r] \arrow[rr, "g"', bend right=49]                                                       & C_g \arrow[r, "r_g"']                                                                                          & D                \\
                                                                                                                    & {} \arrow[u, "\alpha_g \atop \simeq",  near end, phantom]                                                                 &                 
\end{tikzcd} 
\textrm{ with } 
\begin{tikzcd}
                                                    & C_f \arrow[d, "\alpha_f \atop \simeq", near start, phantom] \arrow[rd, "r_f", bend left=20] &                  \\
A \arrow[d, "u"'] \arrow[rr, "f"] \arrow[ru, "l_f", bend left=20] & {} \arrow[d, "\alpha \atop \simeq", phantom]                                   & B \arrow[d, "v"] \\
C \arrow[rr, "g"']                                  & {}                                                                & D               
\end{tikzcd}
= 
\begin{tikzcd}
A \arrow[r, "l_f"] \arrow[d, "u"'] \arrow[rd, "\lambda_\alpha \atop \simeq", phantom] & C_f \arrow[r, "r_f"] \arrow[d, "w_\alpha" description, dashed] \arrow[rd, "\rho_\alpha \atop \simeq", phantom] & B \arrow[d, "v"] \\
C \arrow[r, "l_g"'] \arrow[r] \arrow[rr, "g"', bend right=55]                         & C_g \arrow[r, "r_g"']                                                                                          & D                \\
                                                                                      & {} \arrow[u, "\alpha_g \atop \simeq", near end, phantom]                                                                 &                 
\end{tikzcd}
\]

Moreover, this is functorial in 2-cells in the following sense: for any morphism of commutative square $ (\phi, \psi) : (u, v \alpha) \Rightarrow (u', v', \alpha')$, there exists a unique 2-cell  $ \sigma : w_\alpha \Rightarrow w_{\alpha'}$ in $ \mathcal{C}$ such that 

\[ 
\begin{tikzcd}
A \arrow[r, "l_f"] \arrow[rd, "\lambda_{\alpha'} \atop \simeq", near end, phantom] \arrow[d, "u" description, bend right=50, ""{name=U, inner sep=2pt}] \arrow[d, "u'" description, bend left=50, ""{name=D, inner sep=6pt, below, near start}] & C_f \arrow[r, "r_f"] \arrow[rd, "\rho_{\alpha'} \atop \simeq", phantom] \arrow[d, "w_{\alpha'}" description] & B \arrow[d, "v'"] \\
C \arrow[r, "l_g"'] \arrow[r]                                                                                                   & C_g \arrow[r, "r_g"']                                                                                        & D    \arrow[Rightarrow, from=U, to=D]{}{\phi}           
\end{tikzcd}
=
\begin{tikzcd}
A \arrow[r, "l_f"] \arrow[rd, "\lambda_\alpha \atop \simeq", near start, phantom] \arrow[d, "u"'] & C_f \arrow[r, "r_f"] \arrow[d, "w_\alpha" description, bend right=50, ""{name=U, inner sep=5pt, near start, below}] \arrow[rd, "\rho_{\alpha'} \atop \simeq", near end, phantom] \arrow[d, "w_{\alpha'}" description, bend left=50, ""{name=D, inner sep=5pt, below, near start}] & B \arrow[d, "v'"] \\
C \arrow[r, "l_g"'] \arrow[r]                                                         & C_g \arrow[r, "r_g"']                                                                                                            \arrow[Rightarrow, from=U, to=D]{}{\sigma}                                     & D                
\end{tikzcd}
= 
\begin{tikzcd}
A \arrow[r, "l_f"] \arrow[rd, "\lambda_\alpha \atop \simeq", phantom] \arrow[d, "u"'] & C_f \arrow[r, "r_f"] \arrow[d, "w_\alpha" description] \arrow[rd, "\rho_\alpha \atop \simeq", near start, phantom] & B \arrow[d, "v'" description, bend left=50, ""{name=D, inner sep=8pt, pos=0.5, below, near start}] \arrow[d, "v" description, bend right=50, ""{name=U, inner sep=2pt}] \\
C \arrow[r, "l_g"'] \arrow[r]                                                         & C_g \arrow[r, "r_g"']                                                                                          & D             \arrow[Rightarrow, from=U, to=D]{}[near end]{\psi}                                          
\end{tikzcd}
\]

\end{proposition}

\begin{proof}
Using the universal property of the filler applied to the pseudocommutative square
\[ 
\begin{tikzcd}
A \arrow[r, "l_g u"] \arrow[d, "l_f"'] \arrow[rd, "\alpha \atop \simeq", phantom] & C_g \arrow[d, "r_g"] \\
C_f \arrow[r, "vr_f"']                                               & D                   
\end{tikzcd} \]
(where we omit the canonical coherent iso of the 2-category structure). The 2-cell $\sigma$ is then obtained as the morphism between the universal fillers applied to the morphism of squares \[ (l_g^*\phi, \psi^*r_f ) : (l_g u, v r_f, \rho_{\alpha}^*l_f r_g^*\lambda_{\alpha}) \Rightarrow (l_g u', v' r_f, \rho_{\alpha'}^*l_f r_g^*\lambda_{\alpha'}) : l_f \rightarrow r_g \]
\end{proof}

\begin{proposition}
The bifactorization $ ((l_f, r_f), \alpha_f)$ is biterminal amongst bifactorizations with a left map on the left, and bi-initial amongst bifactorizations with a right map on the right. 
\end{proposition}

\begin{proof}
For any 
\[\begin{tikzcd}
	A && B \\
	& C
	\arrow[""{name=0, anchor=center, inner sep=0}, "f", from=1-1, to=1-3]
	\arrow["l"', from=1-1, to=2-2]
	\arrow["g"', from=2-2, to=1-3]
	\arrow["{\alpha \atop \simeq}"{description}, Rightarrow, draw=none, from=0, to=2-2]
\end{tikzcd}\]
use the filler of the morphism of pseudosquare $ l \rightarrow r_f$ provided by the pasting $ \alpha^{-1}\alpha_f$ to get the comparison 2-cell
\[\begin{tikzcd}
	& {C_f} \\
	A && B \\
	& C
	\arrow["l"', from=2-1, to=3-2]
	\arrow["{l_f}", from=2-1, to=1-2]
	\arrow[""{name=0, anchor=center, inner sep=0}, "{s_\alpha}"{description}, from=3-2, to=1-2]
	\arrow["{r_f}", from=1-2, to=2-3]
	\arrow["g"', from=3-2, to=2-3]
	\arrow["{\lambda_\alpha \atop \simeq}"{description}, Rightarrow, draw=none, from=0, to=2-1]
	\arrow["{\rho_\alpha \atop \simeq}"{description}, Rightarrow, draw=none, from=2-3, to=0]
\end{tikzcd}\]
Similar argument for the right factorization.
\end{proof}

\begin{division}
In the 1-categorical setting, it is easy to see that two right arrows that are equalized by a left arrow must be equals. In the 2-categorical setting, things are not so well behaved: if one has a 2-cell between $\mathcal{R}$-maps that is inverted by a $\mathcal{L}$-map as below
\[\begin{tikzcd}
	{A_1} & {A_2} && {A_3}
	\arrow[""{name=0, anchor=center, inner sep=0}, "{r_1}", curve={height=-12pt}, from=1-2, to=1-4]
	\arrow[""{name=1, anchor=center, inner sep=0}, "{r_2}"', curve={height=12pt}, from=1-2, to=1-4]
	\arrow["l", from=1-1, to=1-2]
	\arrow["\sigma", shorten <=3pt, shorten >=3pt, Rightarrow, from=0, to=1]
\end{tikzcd} = \begin{tikzcd}
	{A_1} && {A_3}
	\arrow[""{name=0, anchor=center, inner sep=0}, "{r_1l}", curve={height=-12pt}, from=1-1, to=1-3]
	\arrow[""{name=1, anchor=center, inner sep=0}, "{r_2l}"', curve={height=12pt}, from=1-1, to=1-3]
	\arrow["{\sigma*l \atop \simeq}"{description}, Rightarrow, draw=none, from=0, to=1]
\end{tikzcd}\]
then the diagonalization of the pseudosquare 
\[\begin{tikzcd}
	{A_1} & {A_2} \\
	{A_2} & {A_3}
	\arrow[""{name=0, anchor=center, inner sep=0}, "l", from=1-1, to=1-2]
	\arrow["l"', from=1-1, to=2-1]
	\arrow[""{name=1, anchor=center, inner sep=0}, "{r_1}"', from=2-1, to=2-2]
	\arrow["{r_2}", from=1-2, to=2-2]
	\arrow["{\sigma*l \atop \simeq}"{description}, Rightarrow, draw=none, from=0, to=1]
\end{tikzcd}\]
must then be in both $\mathcal{L}$ and $ \mathcal{R}$, hence be an equivalence, and comes equipped with an invertible 2-cell
\[\begin{tikzcd}
	& {A_2} \\
	{A_2} & {A_3}
	\arrow["{r_1}"', from=2-1, to=2-2]
	\arrow["{r_2}", from=1-2, to=2-2]
	\arrow[""{name=0, anchor=center, inner sep=0}, "e\simeq", from=2-1, to=1-2]
	\arrow["{\rho \atop \simeq}"{description}, Rightarrow, draw=none, from=2-2, to=0]
\end{tikzcd}\]
which relates hence $ r_1$, $r_2$ by an invertible 2-cell. However this does not enforce $\sigma$ itself to be invertible. Similarly, it seems there is no way to enforce that two parallels 2-cells between $\mathcal{R}$-maps that are equified by a $\mathcal{L}$-maps should be equal; and no reason as well that a $\mathcal{L}$-map reflects 2-cells it inserts between two $\mathcal{R}$-maps. This motivates the following definition:
\end{division}

\begin{definition}
A bifactorization systems $ (\mathcal{L},\mathcal{R}) $ shall be said to be \emph{left fully faithful} if for any $l : A_1 \rightarrow A_2 $ in $\mathcal{L}$ and any $A_3$, the restricted composition
\[\begin{tikzcd}
	{\mathcal{R}[A_2, A_3]} & {\mathcal{C}[A_1, A_3]}
	\arrow["{\mathcal{C}[l,A_3]}", from=1-1, to=1-2]
\end{tikzcd}\]
is full and faithful: this means that for any $ r_1, r_2 : A_2 \rightarrow A_3$ in $\mathcal{R}$: \begin{itemize}
    \item if we have $ \sigma : r_1 \Rightarrow r_2$ such that $\sigma*l$ is invertible then $\sigma$ is invertible
    \item if we have $\sigma, \sigma' : r_1 \Rightarrow r_2$ such that $ \sigma*l=\sigma'*l$ then $\sigma=\sigma'$
    \item if we have $ \sigma : r_1l \Rightarrow r_2l$ then there exists a (unique) $ \tau : r_1 \Rightarrow r_2$ such that $ \sigma = \tau*l$
\end{itemize}
\end{definition}

Though immediate, the following observation will be important in our examples in the next chapter:

\begin{proposition}\label{discrete implies left ff}
If $ (\mathcal{L},\mathcal{R})$ is a bifactorization system such that morphisms in $ \mathcal{L}$ are discrete. Then $ (\mathcal{L}, \mathcal{R})$ is left-full and faithful.
\end{proposition}

\subsection{Lax generic bifactorization systems}

In this subsection, we consider an additional laxness property one can consider for a bifactorization system. As observed above, a bifactorization system may not have enough structure to factorize a globular 2-cell in $\mathcal{C}$. To this end we introduce here the following notion - whose terminology is inspired from \cite{walker2020lax}, as the notion we introduce there will connect with its central notion we will make mention of later.  \\

\begin{definition}
A bifactorization system $ (\mathcal{L}, \mathcal{R})$ is said to have \emph{lax-factorization of 2-cells} if for any 2-cell as below 
\[\begin{tikzcd}
	A && B
	\arrow[""{name=0, anchor=center, inner sep=0}, "{f_1}", curve={height=-12pt}, from=1-1, to=1-3]
	\arrow[""{name=1, anchor=center, inner sep=0}, "{f_2}"', curve={height=12pt}, from=1-1, to=1-3]
	\arrow["\sigma", shorten <=3pt, shorten >=3pt, Rightarrow, from=0, to=1]
\end{tikzcd}\]
the respective factorizations are related through a composite 2-cell as below 
\[\begin{tikzcd}
	& {A_{f_1}} \\
	A && B \\
	& {A_{f_2}}
	\arrow[""{name=0, anchor=center, inner sep=0}, "{l_{f_1}}", from=2-1, to=1-2]
	\arrow[""{name=1, anchor=center, inner sep=0}, "{r_{f_1}}", from=1-2, to=2-3]
	\arrow[""{name=2, anchor=center, inner sep=0}, "{l_{f_2}}"', from=2-1, to=3-2]
	\arrow[""{name=3, anchor=center, inner sep=0}, "{r_{f_2}}"', from=3-2, to=2-3]
	\arrow["{m_f}"{description}, from=1-2, to=3-2]
	\arrow["{\lambda_\sigma}", shorten <=4pt, shorten >=4pt, Rightarrow, from=0, to=2]
	\arrow["{\rho_\sigma}"', shorten <=4pt, shorten >=4pt, Rightarrow, from=1, to=3]
\end{tikzcd}\] such that we have 
\[  \alpha_{f_2} r_{f_2}*\lambda_{\sigma} l_{f_1}*\rho_\sigma \alpha^{-1} = \sigma  \]
and moreover for any other factorization 
\[\begin{tikzcd}
	& {A_{f_1}} \\
	A && B \\
	& {A_{f_2}}
	\arrow[""{name=0, anchor=center, inner sep=0}, "{l_{f_1}}", from=2-1, to=1-2]
	\arrow[""{name=1, anchor=center, inner sep=0}, "{r_{f_1}}", from=1-2, to=2-3]
	\arrow[""{name=2, anchor=center, inner sep=0}, "{l_{f_2}}"', from=2-1, to=3-2]
	\arrow[""{name=3, anchor=center, inner sep=0}, "{r_{f_2}}"', from=3-2, to=2-3]
	\arrow["m"{description}, from=1-2, to=3-2]
	\arrow["\lambda", shorten <=4pt, shorten >=4pt, Rightarrow, from=0, to=2]
	\arrow["\rho"', shorten <=4pt, shorten >=4pt, Rightarrow, from=1, to=3]
\end{tikzcd}\]
with the same property, there exists a unique $ \mu : m \Rightarrow m_\sigma$ such that one has factorization of 2-cells:
\[\begin{tikzcd}
	& {A_{f_1}} \\
	A \\
	& {A_{f_2}}
	\arrow[""{name=0, anchor=center, inner sep=0}, "{l_{f_1}}", from=2-1, to=1-2]
	\arrow[""{name=1, anchor=center, inner sep=0}, "{l_{f_2}}"', from=2-1, to=3-2]
	\arrow["m", from=1-2, to=3-2]
	\arrow["\lambda", shorten <=4pt, shorten >=4pt, Rightarrow, from=0, to=1]
\end{tikzcd}  =
\begin{tikzcd}
	& {A_{f_1}} \\
	A \\
	& {A_{f_2}}
	\arrow[""{name=0, anchor=center, inner sep=0}, "{l_{f_1}}", from=2-1, to=1-2]
	\arrow[""{name=1, anchor=center, inner sep=0}, "{l_{f_2}}"', from=2-1, to=3-2]
	\arrow[""{name=2, anchor=center, inner sep=0}, "{m_\sigma}"{description}, from=1-2, to=3-2]
	\arrow[""{name=3, anchor=center, inner sep=0}, "m", curve={height=-30pt}, from=1-2, to=3-2]
	\arrow["{\lambda_\sigma}", shorten <=4pt, shorten >=4pt, Rightarrow, from=0, to=1]
	\arrow["{\mu}"', shorten <=5pt, shorten >=5pt, Rightarrow, from=3, to=2]
\end{tikzcd}  \hskip1cm
  \begin{tikzcd}
	{A_{f_1}} \\
	& B \\
	{A_{f_2}}
	\arrow["m_\sigma"', from=1-1, to=3-1]
	\arrow[""{name=0, anchor=center, inner sep=0}, "{r_{f_1}}", from=1-1, to=2-2]
	\arrow[""{name=1, anchor=center, inner sep=0}, "{r_{f_2}}"', from=3-1, to=2-2]
	\arrow["\rho_\sigma"', shorten <=4pt, shorten >=4pt, Rightarrow, from=0, to=1]
\end{tikzcd} = 
\begin{tikzcd}
	{A_{f_1}} \\
	& B \\
	{A_{f_2}}
	\arrow[""{name=0, anchor=center, inner sep=0}, "{m}"{description}, from=1-1, to=3-1]
	\arrow[""{name=1, anchor=center, inner sep=0}, "m_\sigma"', curve={height=30pt}, from=1-1, to=3-1]
	\arrow[""{name=2, anchor=center, inner sep=0}, "{r_{f_1}}", from=1-1, to=2-2]
	\arrow[""{name=3, anchor=center, inner sep=0}, "{r_{f_2}}"', from=3-1, to=2-2]
	\arrow["{\mu}", shorten <=5pt, shorten >=5pt, Rightarrow, from=1, to=0]
	\arrow["{\rho}"', shorten <=4pt, shorten >=4pt, Rightarrow, from=2, to=3]
\end{tikzcd}  \]
\end{definition}

The factorizations of those lax 2-cell will be functorial in the following sense:

\begin{definition}
A lax generic bifactorization system will be said to be \emph{pseudofunctorial} if for any two 2-cells as below
\[\begin{tikzcd}
	A && B
	\arrow[""{name=0, anchor=center, inner sep=0}, "{f_1}", curve={height=-24pt}, from=1-1, to=1-3]
	\arrow[""{name=1, anchor=center, inner sep=0}, "{f_3}"', curve={height=24pt}, from=1-1, to=1-3]
	\arrow[""{name=2, anchor=center, inner sep=0}, "{f_2}"{description}, from=1-1, to=1-3]
	\arrow["\sigma", shorten <=3pt, shorten >=3pt, Rightarrow, from=0, to=2]
	\arrow["\sigma'", shorten <=3pt, shorten >=3pt, Rightarrow, from=2, to=1]
\end{tikzcd}\]
one has a canonical invertible 2-cell 
\[  m_{\sigma'\sigma} \simeq m_{\sigma'}m_\sigma \]
and equality of 2-cells 
\[\begin{tikzcd}[sep=large]
	& {A_{f_1}} \\
	A & {A_{f_2}} \\
	& {A_{f_3}}
	\arrow[""{name=0, anchor=center, inner sep=0}, "{l_{f_1}}", from=2-1, to=1-2]
	\arrow["{m_{\sigma}}", from=1-2, to=2-2]
	\arrow[""{name=1, anchor=center, inner sep=0}, "{l_{f_3}}"', from=2-1, to=3-2]
	\arrow["{m_{\sigma'}}", from=2-2, to=3-2]
	\arrow[""{name=2, pos=0.51,  inner sep=0}, "{l_{f_2}}"{description}, from=2-1, to=2-2]
	\arrow["{\lambda_{\sigma}}", shift left=3, shorten <=1pt, shorten >=3pt, Rightarrow, from=0, to=2]
	\arrow["{\lambda_{\sigma'}}", shift left=3, shorten <=3pt, shorten >=1pt, Rightarrow, from=2, to=1]
\end{tikzcd} = \begin{tikzcd}
	& {A_{f_1}} \\
	A \\
	& {A_{f_2}}
	\arrow[""{name=0, anchor=center, inner sep=0}, "{l_{f_1}}", from=2-1, to=1-2]
	\arrow[""{name=1, anchor=center, inner sep=0}, "{l_{f_2}}"', from=2-1, to=3-2]
	\arrow["m_{\sigma'\sigma}", from=1-2, to=3-2]
	\arrow["\lambda_{\sigma'\sigma}", shorten <=4pt, shorten >=4pt, Rightarrow, from=0, to=1]
\end{tikzcd} \hskip1cm \begin{tikzcd}[sep=large]
	{A_{f_1}} \\
	{A_{f_2}} & B \\
	{A_{f_3}}
	\arrow[""{name=0, anchor=center, inner sep=0}, "{r_{f_1}}", from=1-1, to=2-2]
	\arrow["{m_{\sigma}}"', from=1-1, to=2-1]
	\arrow[""{name=1, anchor=center, inner sep=0}, "{r_{f_3}}"', from=3-1, to=2-2]
	\arrow["{m_{\sigma'}}"', from=2-1, to=3-1]
	\arrow[""{name=2, pos=0.49, inner sep=0}, "{r_{f_2}}"{description}, from=2-1, to=2-2]
	\arrow["{\rho_{\sigma}}"', shift right=3, shorten <=1pt, shorten >=3pt, Rightarrow, from=0, to=2]
	\arrow["{\rho_{\sigma'}}"', shift right=3, shorten <=3pt, shorten >=1pt, Rightarrow, from=2, to=1]
\end{tikzcd} =  \begin{tikzcd}
	{A_{f_1}} \\
	& B \\
	{A_{f_2}}
	\arrow["m_{\sigma'\sigma}"', from=1-1, to=3-1]
	\arrow[""{name=0, anchor=center, inner sep=0}, "{r_{f_1}}", from=1-1, to=2-2]
	\arrow[""{name=1, anchor=center, inner sep=0}, "{r_{f_2}}"', from=3-1, to=2-2]
	\arrow["\rho_{\sigma'\sigma}"', shorten <=4pt, shorten >=4pt, Rightarrow, from=0, to=1]
\end{tikzcd} \]
and for any $ f: A \rightarrow B$, one has
\[ m_{1_f} = 1_{A_f} \hskip1cm \lambda_{1_f} = 1_{l_f} \hskip1cm \rho_{1_f} = 1_{r_f} \]
\end{definition}

\begin{corollary}
If $(\mathcal{L},\mathcal{R})$ is lax-generic and pseudofunctorial and $ \sigma : f_1 \Rightarrow f_2$ as above is invertible, then $ m_\sigma$ is an equivalence and $ \lambda_\sigma$, $\rho_\sigma$ are invertible.  
\end{corollary}

\begin{remark}
Then those data will coincide up to invertible 2-cells with the invertible 2-cells provided by the by uniqueness of the bifactorization
\[\begin{tikzcd}
	& {A_{f_1}} \\
	A && B \\
	& {A_{f_2}}
	\arrow[""{name=0, anchor=center, inner sep=0}, "{l_{f_1}}", from=2-1, to=1-2]
	\arrow[""{name=1, anchor=center, inner sep=0}, "{r_{f_1}}", from=1-2, to=2-3]
	\arrow[""{name=2, anchor=center, inner sep=0}, "{l_{f_2}}"', from=2-1, to=3-2]
	\arrow[""{name=3, anchor=center, inner sep=0}, "{r_{f_2}}"', from=3-2, to=2-3]
	\arrow["{m \atop \simeq}"{description}, from=1-2, to=3-2]
	\arrow["{\lambda \atop \simeq}"{description}, Rightarrow, draw=none, from=0, to=2]
	\arrow["{\rho \atop \simeq}"{description}, Rightarrow, draw=none, from=1, to=3]
\end{tikzcd}\]
\end{remark}

Pseudofunctorial lax-generic bifactorization systems have a more general diagonalization property: in fact they enjoy diagonalization of lax morphisms from a left map to a right map:

\begin{definition}\label{factorization of lax square in lax fact system}
A bifactorization system $ (\mathcal{L},\mathcal{R})$ will be said to have \emph{lax diagonalization} if \begin{itemize}
    \item for any lax square as below  with $l$ in $\mathcal{L}$ and $r$ in $\mathcal{R}$
\[\begin{tikzcd}
	A & C \\
	B & D
	\arrow["l"', from=1-1, to=2-1]
	\arrow[""{name=0, anchor=center, inner sep=0}, "f", from=1-1, to=1-2]
	\arrow["r", from=1-2, to=2-2]
	\arrow[""{name=1, anchor=center, inner sep=0}, "g"', from=2-1, to=2-2]
	\arrow["{\Uparrow \sigma}"{description}, Rightarrow, draw=none, from=1, to=0]
\end{tikzcd}\]
there exists a factorization 
\[\begin{tikzcd}[sep=large]
	A & C \\
	B & D
	\arrow["l"', from=1-1, to=2-1]
	\arrow[""{name=0, anchor=center, inner sep=0}, "f", from=1-1, to=1-2]
	\arrow["r", from=1-2, to=2-2]
	\arrow[""{name=1, anchor=center, inner sep=0}, "g"', from=2-1, to=2-2]
	\arrow[""{name=2, anchor=center, inner sep=0}, "{d_\sigma}"{description}, from=2-1, to=1-2]
	\arrow["{\Uparrow \lambda_\sigma}"{description}, shorten <=2pt, shorten >=2pt, draw=none, from=2, to=0]
	\arrow["{\Uparrow \rho_\sigma}"{description}, shorten <=2pt, shorten >=2pt,draw=none, from=1, to=2]
\end{tikzcd}\] 
which is universal amongst factorization of this 2-cell in the sense that for any other factorization of the same square 
\[\begin{tikzcd}[sep=large]
	A & C \\
	B & D
	\arrow["l"', from=1-1, to=2-1]
	\arrow[""{name=0, anchor=center, inner sep=0}, "f", from=1-1, to=1-2]
	\arrow["r", from=1-2, to=2-2]
	\arrow[""{name=1, anchor=center, inner sep=0}, "g"', from=2-1, to=2-2]
	\arrow[""{name=2, anchor=center, inner sep=0}, "{d}"{description}, from=2-1, to=1-2]
	\arrow["{\Uparrow \lambda}"{description}, shorten <=2pt, shorten >=2pt, draw=none, from=2, to=0]
	\arrow["{\Uparrow \rho}"{description}, shorten <=2pt, shorten >=2pt,draw=none, from=1, to=2]
\end{tikzcd}\] 
there exists a unique $ \xi : d \Rightarrow d_\sigma$ such that we have the factorization
\[\begin{tikzcd}[sep=large]
	A & C \\
	B
	\arrow["l"', from=1-1, to=2-1]
	\arrow["f", from=1-1, to=1-2]
	\arrow[""{name=0, anchor=center, inner sep=0}, "{w_\sigma}"{description}, from=2-1, to=1-2]
	\arrow[""{name=1, anchor=center, inner sep=0}, "w"{description}, curve={height=18pt}, from=2-1, to=1-2]
	\arrow["{\Uparrow \lambda_\sigma}"{description}, Rightarrow, draw=none, from=0, to=1-1]
	\arrow["\xi"', shorten <=3pt, shorten >=3pt, Rightarrow, from=1, to=0]
\end{tikzcd} = \begin{tikzcd}
	A & C \\
	B
	\arrow["l"', from=1-1, to=2-1]
	\arrow["f", from=1-1, to=1-2]
	\arrow[""{name=0, anchor=center, inner sep=0}, "w"{description}, from=2-1, to=1-2]
	\arrow["{\Uparrow \lambda}"{description}, Rightarrow, draw=none, from=0, to=1-1]
\end{tikzcd}\]
\[\begin{tikzcd}[sep=large]
	& C \\
	B & D
	\arrow["g"', from=2-1, to=2-2]
	\arrow["r", from=1-2, to=2-2]
	\arrow[""{name=0, anchor=center, inner sep=0}, "w"{description}, from=2-1, to=1-2]
	\arrow[""{name=1, anchor=center, inner sep=0}, "{w_\sigma}", curve={height=-18pt}, from=2-1, to=1-2]
	\arrow["{\Uparrow \rho_\sigma}"{description}, Rightarrow, draw=none, from=2-2, to=0]
	\arrow["\xi"', shorten <=3pt, shorten >=3pt, Rightarrow, from=0, to=1]
\end{tikzcd}=\begin{tikzcd}
	& C \\
	B & D
	\arrow["g"', from=2-1, to=2-2]
	\arrow["r", from=1-2, to=2-2]
	\arrow[""{name=0, anchor=center, inner sep=0}, "{w_\sigma}", from=2-1, to=1-2]
	\arrow["{\Uparrow \rho_\sigma}"{description}, Rightarrow, draw=none, from=2-2, to=0]
\end{tikzcd}\]

\item Morever, for a morphism of lax square as below
\[\begin{tikzcd}
	A && C \\
	B && D
	\arrow["l"', from=1-1, to=2-1]
	\arrow[""{name=0, anchor=center, inner sep=0}, "{g_1}"', from=2-1, to=2-3]
	\arrow[""{name=1, anchor=center, inner sep=0}, "{f_1}"{description}, from=1-1, to=1-3]
	\arrow["r", from=1-3, to=2-3]
	\arrow[""{name=2, anchor=center, inner sep=0}, "{f_2}", curve={height=-18pt}, from=1-1, to=1-3]
	\arrow["{\Uparrow \phi_1}"{description}, shorten <=2pt, shorten >=2pt, draw=none, from=1, to=2]
	\arrow["{\Uparrow \sigma_1}"{description}, shorten <=4pt, shorten >=4pt, draw=none, from=0, to=1]
\end{tikzcd} = \begin{tikzcd}
	A && C \\
	B && D
	\arrow["l"', from=1-1, to=2-1]
	\arrow[""{name=0, anchor=center, inner sep=0}, "{g_1}"', curve={height=18pt}, from=2-1, to=2-3]
	\arrow["r", from=1-3, to=2-3]
	\arrow[""{name=1, anchor=center, inner sep=0}, "{f_2}", from=1-1, to=1-3]
	\arrow[""{name=2, anchor=center, inner sep=0}, "{g_2}"{description}, from=2-1, to=2-3]
	\arrow["{\Uparrow \gamma_2}"{description}, shorten <=2pt, shorten >=2pt, draw=none, from=0, to=2]
	\arrow["{\Uparrow \sigma_2}"{description}, shorten <=4pt, shorten >=4pt, draw=none, from=2, to=1]
\end{tikzcd}\]
there is a unique 2-cell $ \xi : d_1 \Rightarrow d_2$ such that we have the factorization 
\[\begin{tikzcd}[sep=large]
	A & C \\
	B
	\arrow["l"', from=1-1, to=2-1]
	\arrow[""{name=0, anchor=center, inner sep=0}, "{f_2}", from=1-1, to=1-2]
	\arrow[""{name=1, anchor=center, inner sep=0}, "{d_2}"{description}, from=2-1, to=1-2]
	\arrow[""{name=2, anchor=center, inner sep=0}, "{d_1}"', curve={height=18pt}, from=2-1, to=1-2]
	\arrow["{\lambda_{\sigma_2}}", shorten <=2pt, shorten >=2pt, Rightarrow, from=1, to=0]
	\arrow["\xi", shorten <=2pt, shorten >=2pt, Rightarrow, from=2, to=1]
\end{tikzcd} = \begin{tikzcd}[sep=large]
	A & C \\
	B
	\arrow["l"', from=1-1, to=2-1]
	\arrow[""{name=0, anchor=center, inner sep=0}, "{f_2}", curve={height=-18pt}, from=1-1, to=1-2]
	\arrow[""{name=1, anchor=center, inner sep=0}, "{d_1}"', from=2-1, to=1-2]
	\arrow[""{name=2, anchor=center, inner sep=0}, "{f_1}"{description}, from=1-1, to=1-2]
	\arrow["{\lambda_{\sigma_1}}", shorten <=2pt, shorten >=2pt, Rightarrow, from=1, to=2]
	\arrow["{\phi_1}", shorten <=2pt, shorten >=2pt, Rightarrow, from=2, to=0]
\end{tikzcd} \hskip1cm
\begin{tikzcd}[sep=large]
	& C \\
	B & D
	\arrow[""{name=0, anchor=center, inner sep=0}, "{d_1}"{description}, from=2-1, to=1-2]
	\arrow["r", from=1-2, to=2-2]
	\arrow[""{name=1, anchor=center, inner sep=0}, "g"', from=2-1, to=2-2]
	\arrow[""{name=2, anchor=center, inner sep=0}, "{d_2}", curve={height=-18pt}, from=2-1, to=1-2]
	\arrow["\xi"', shorten <=3pt, shorten >=3pt, Rightarrow, from=0, to=2]
	\arrow["{\rho_{\sigma_1}}"', shorten <=2pt, shorten >=2pt, Rightarrow, from=1, to=0]
\end{tikzcd}=
\begin{tikzcd}[sep=large]
	& C \\
	B & D
	\arrow["r", from=1-2, to=2-2]
	\arrow[""{name=0, anchor=center, inner sep=0}, "{g_2}"{description}, from=2-1, to=2-2]
	\arrow[""{name=1, anchor=center, inner sep=0}, "{d_2}", from=2-1, to=1-2]
	\arrow[""{name=2, anchor=center, inner sep=0}, "{g_1}"', curve={height=18pt}, from=2-1, to=2-2]
	\arrow["{\gamma_2}"', shorten <=2pt, shorten >=2pt, Rightarrow, from=2, to=0]
	\arrow["{\rho_{\sigma_2}}"', shorten <=2pt, shorten >=2pt, Rightarrow, from=0, to=1]
\end{tikzcd}
\]
\end{itemize}
\end{definition}

\begin{proposition}
Let $(\mathcal{L},\mathcal{R})$ be a bifactorization system in a bicategory $ \mathcal{C}$: \begin{itemize}
    \item if $(\mathcal{L},\mathcal{R})$ is pseudofunctorial lax-generic then it has lax diagonalization.
    \item if $(\mathcal{L},\mathcal{R})$ has lax diagonalization, then it is lax-generic. 
\end{itemize} 
\end{proposition}

\begin{proof}

For the first item: if $(\mathcal{L},\mathcal{R})$ is pseudofunctorial lax-generic, for a lax square $ \sigma$ as above, one can paste $\sigma$ with the invertible 2-cell in the factorization of $f$ and $g$ to get a 2-cell $ r *\alpha_f^{-1} \sigma \alpha_g*l $ - which we call abusively $\sigma$ - and by lax-genericity we have the following universal factorization of this 2-cell as below
\[\begin{tikzcd}
	A & {C_f} & C \\
	B & {C_g} & D
	\arrow["l"', from=1-1, to=2-1]
	\arrow["{l_f}", from=1-1, to=1-2]
	\arrow["{m_\sigma}"{description}, from=2-2, to=1-2]
	\arrow["{l_g}"', from=2-1, to=2-2]
	\arrow["r", from=1-3, to=2-3]
	\arrow["{r_f}", from=1-2, to=1-3]
	\arrow["{r_g}"', from=2-2, to=2-3]
	\arrow["{\lambda_{\sigma}}", Rightarrow, from=2-1, to=1-2]
	\arrow["{\rho_\sigma}"', Rightarrow, from=2-2, to=1-3]
\end{tikzcd}\]
Then the composite $ r_f m_\sigma l_g$ provides a universal lax diagonalization 
\[\begin{tikzcd}[sep=large]
	A & C \\
	B & D
	\arrow["l"', from=1-1, to=2-1]
	\arrow[""{name=0, anchor=center, inner sep=0}, "f", from=1-1, to=1-2]
	\arrow["r", from=1-2, to=2-2]
	\arrow[""{name=1, anchor=center, inner sep=0}, "g"', from=2-1, to=2-2]
	\arrow[""{name=2, anchor=center, inner sep=0}, "{r_f m_\sigma l_g}"{description}, from=2-1, to=1-2]
	\arrow["{r_f*\lambda_\sigma}", shorten <=2pt, shorten >=2pt, Rightarrow, from=2, to=0]
	\arrow["{\rho_\sigma*l_g}"', shorten <=2pt, shorten >=2pt, Rightarrow, from=1, to=2]
\end{tikzcd}\]
The universal property of this decomposition is moreover ensured from the universal property of $(m_\sigma, \lambda_\sigma, \rho_\sigma)$. \\

For the morphisms of lax squares, consider the equality of 2-cell provided by pseudofunctoriality 
\[\begin{tikzcd}[sep=large]
	A & {A_{f_2}} \\
	B & {A_{f_1}} & C \\
	& {A_{g_1}} & D
	\arrow[""{name=0, pos=0.51, inner sep=0}, "{l_{f_2}}", from=1-1, to=1-2]
	\arrow[""{name=1, anchor=center, inner sep=0}, "{r_{f_2}}", from=1-2, to=2-3]
	\arrow[""{name=2, anchor=center, inner sep=0}, "{l_{f_1}}"{description}, from=1-1, to=2-2]
	\arrow[""{name=3, anchor=center, inner sep=0}, "{r_{f_1}}"{description}, from=2-2, to=2-3]
	\arrow["l"', from=1-1, to=2-1]
	\arrow[""{name=4, anchor=center, inner sep=0}, "{l_{g_1}}"', from=2-1, to=3-2]
	\arrow[""{name=5, anchor=center, inner sep=0}, "{r_{g_1}}"', from=3-2, to=3-3]
	\arrow["r", from=2-3, to=3-3]
	\arrow["m_{\sigma_1}"{description}, from=3-2, to=2-2]
	\arrow["m_\phi"{description}, from=2-2, to=1-2]
	\arrow["{\lambda_{\phi}}"', shorten <=2pt, shorten >=2pt, Rightarrow, from=2, to=0]
	\arrow["{\rho_\phi}", shorten <=2pt, shorten >=2pt, Rightarrow, from=3, to=1]
	\arrow["{\lambda_{\sigma_1}}"', shorten <=4pt, shorten >=4pt, Rightarrow, from=4, to=2]
	\arrow["{\rho_{\sigma_1}}"{description}, shorten <=4pt, shorten >=4pt, Rightarrow, from=5, to=3]
\end{tikzcd} = \begin{tikzcd}[sep=large]
	A & {A_{f_2}} \\
	B & {A_{g_2}} & C \\
	& {A_{g_1}} & D
	\arrow[""{name=0, anchor=center, inner sep=0}, "{l_{f_2}}", from=1-1, to=1-2]
	\arrow[""{name=1, anchor=center, inner sep=0}, "{r_{f_2}}", from=1-2, to=2-3]
	\arrow["l"', from=1-1, to=2-1]
	\arrow[""{name=2, anchor=center, inner sep=0}, "{l_{g_1}}"', from=2-1, to=3-2]
	\arrow[""{name=3, anchor=center, inner sep=0}, "{r_{g_1}}"', from=3-2, to=3-3]
	\arrow["r", from=2-3, to=3-3]
	\arrow["l_{g_2}"{description}, ""{name=4, anchor=center, inner sep=0}, from=2-1, to=2-2]
	\arrow["r_{g_2}"{description}, ""{name=5, anchor=center, inner sep=0}, from=2-2, to=3-3]
	\arrow["m_{\sigma_2}"{description}, from=2-2, to=1-2]
	\arrow["m_\gamma"{description}, from=3-2, to=2-2]
	\arrow["{\lambda_{\sigma_2}}"', shorten <=4pt, shorten >=4pt, Rightarrow, from=4, to=0]
	\arrow["{\rho_{\sigma_2}}"{pos=0.6}, shorten <=4pt, shorten >=4pt, Rightarrow, from=5, to=1]
	\arrow["{\lambda_{\gamma}}"', shorten <=2pt, shorten >=2pt, Rightarrow, from=2, to=4]
	\arrow["{\rho_\gamma}", shorten <=2pt, shorten >=2pt, Rightarrow, from=3, to=5]
\end{tikzcd}\]
But now from the equality of 2-cells 
\[  r*\phi \sigma_1 = \sigma_2 \gamma*l \]
we deduce by pseudofunctoriality the following chain of universal equivalences
\[  m_{\phi}m_{\sigma_1} \simeq m_{r*\phi\sigma_1} \simeq m_{\sigma_2\gamma*l} \simeq m_{\sigma_2}m_\gamma   \]
while we have equality of 2-cells 
\[ \lambda_{\phi} m_{\phi}*\lambda_{\sigma_1} = \lambda_{r*\phi\sigma_1} = \lambda_{\sigma_2\gamma*l}= \lambda_{\sigma_2}m_{\sigma_2}*\lambda_{\gamma}*l \]
\[ r*\rho_{\phi}*m_{\sigma_1}\rho_{\sigma_1} = \rho_{r*\phi\sigma_1} = \rho_{\sigma_2\gamma*l} = m_\gamma* \rho_{\sigma_2}\rho_{\gamma} \]
and the desired $ \xi : d_1 \Rightarrow d_2 $ is obtained as the composite
\[\begin{tikzcd}
	{d_1 = r_{f_1}m_{\sigma_1}l_{g_1}} & {r_{f_2}m_\phi m_{\sigma_1}l_{g_1} \simeq r_{f_2}m_{\sigma_2}m_{\gamma}l_{g_1} } & {r_{f_2}m_{\sigma_2}l_{g_2}=d_2}
	\arrow["{\rho_{\phi}*(m_{\sigma_1}l_{g_1})}", Rightarrow, from=1-1, to=1-2]
	\arrow["{(r_{f_2}m_{\sigma_2})*\lambda_{\gamma}}", Rightarrow, from=1-2, to=1-3]
\end{tikzcd}\]

For the second item, suppose that $(\mathcal{L},\mathcal{R})$ has lax diagonalization; let be a globular 2-cell $\sigma : f_1 \Rightarrow f_2$. Then the decomposition of $ \sigma$ is provided by lax diagonalization of the composite square 
\[\begin{tikzcd}
	A & {C_{f_2}} \\
	{C_{f_1}} & B
	\arrow["{l_{f_1}}"', from=1-1, to=2-1]
	\arrow[""{name=0, anchor=center, inner sep=0}, "{r_{f_1}}"', from=2-1, to=2-2]
	\arrow[""{name=1, anchor=center, inner sep=0}, "{l_{f_2}}", from=1-1, to=1-2]
	\arrow["{r_{f_2}}", from=1-2, to=2-2]
	\arrow["{\footnotesize{\Uparrow \alpha_{f_2}^{-1}\sigma\alpha_{f_1}}}"{description}, Rightarrow, draw=none, from=0, to=1]
\end{tikzcd}
\]
\end{proof}
where one just has to take $ (\lambda_{\alpha_{f_2}^{-1}\sigma\alpha_{f_1}}, d_{\alpha_{f_2}^{-1}\sigma\alpha_{f_1}} \rho_{\alpha_{f_2}^{-1}\sigma\alpha_{f_1}})$. 

\begin{remark}
Beware that we do not deduce pseudofunctoriality from the sole lax diagonalization. We do not know whether the result above can be strengthened. 
\end{remark}

\begin{remark}
If $(\mathcal{L}, \mathcal{R})$ is pseudofunctorial lax generic, the lax diagonalization of a pseudosquare $\sigma$ coincides with the diagonalization obtained by bi-orthogonality; in particular the corresponding universal $ \lambda_\sigma, \rho_\sigma$ are invertible.
\end{remark}

\begin{remark}
In practice, one has to determine by \emph{ad hoc} considerations whether a bifactorization system has (op)lax factorization of 2-cells. Moreover, lax-factorization of 2-cells can enjoy further truncations properties, for instance having either the left of right 2-cell invertible in the lax factorization - not both at the same time of course. In those cases we can deduce by cancellation the middle arrow to be either left or right, which may be interpreted in a geometric way as we will see in our examples, or also may be related to lax-stability condition - though this will not be the case in those very examples. 
\end{remark}

\begin{definition}
A bifactorization system $(\mathcal{L}, \mathcal{R})$ with pseudofunctorial lax-factorization of 2-cells will be said to be\begin{itemize}
    \item \emph{left truncated} if for any $ \sigma $ the left 2-cell $ \lambda_\sigma$ of the lax-factorization is invertible;
    \item \emph{right truncated} if for any $\sigma$ the right 2-cell $ \rho_\sigma$ of the lax-factorization is invertible. 
\end{itemize}
\end{definition}

\begin{remark}
By right (resp. left) cancellation of left (resp. right) maps, the middle arrow $ m_\sigma$ in the lax factorization is in $\mathcal{L}$ (resp. in $\mathcal{R}$) whenever the bifactorization system is left (resp. right) truncated. 
\end{remark}

We also have the immediate consequence:

\begin{proposition}
If $ (\mathcal{L}, \mathcal{R})$ has left (resp. right) truncated lax-factorization of 2-cells, then for any lax square $ (f,g, \sigma) : l \rightarrow r$, the left part $ \lambda_\sigma$ (resp. the right part $\rho_\sigma$) is invertible.
\end{proposition}

The right truncated case is actually the one that interests us the most, as it will encompass all our examples below - and even further examples investigated in a future work.

\section{2-dimensional stable inclusions from bifactorization systems}

Now we turn to the main construction of this work: from a bifactorization system, we define a notion of \emph{left objects} - objects whose terminal map is left - and construct a bistable inclusion from the opposite inclusion of left objects and left maps between them. This mirrors the expected statement that \emph{right objects} - those whose terminal map is right - form a full reflective sub-bicategory.

\subsection{Inheritance of bifactorization in pseudoslices}

In the next definitions, $ C $ is an arbitrary object of $\mathcal{C}$. We recall that the pseudoslice over $C$ is the bicategory $ \mathcal{C}/C$ whose \begin{itemize}
    \item 0-cells are 1-cells $f : A \rightarrow C$ in $\mathcal{C}$
    \item 1-cells are invertible 2-cells $(u,\alpha) : f \rightarrow g$ as follows:
    \[ 
\begin{tikzcd}
A \arrow[rr, "u"] \arrow[rd, "f"'] & {} \arrow[d, "\alpha \atop \simeq", near start, phantom] & B \arrow[ld, "g"] \\
                                   & C                                            &                  
\end{tikzcd} \]
    \item 2-cells $ (u, \alpha) \Rightarrow (u', \alpha')$ are the data of some $ \phi : u \Rightarrow u' $ such that $ \alpha = \alpha' g^*\phi$
    \[ 
\begin{tikzcd}
A \arrow[rr, "u"] \arrow[rd, "f"'] & {} \arrow[d, "\alpha \atop \simeq", near start, phantom] & B \arrow[ld, "g"] \\
                                   & C                                            &                  
\end{tikzcd} 
=
\begin{tikzcd}
A \arrow[rr, "u", bend left=30, ""{name=U, inner sep=0.1pt, below}] \arrow[rd, "f"'] \arrow[rr, "u'"',  ""{name=U', inner sep=0.1pt}] & {} \arrow[d, "\alpha' \atop \simeq", phantom] & B \arrow[ld, "g"] \\ \arrow[from = U, to= U', Rightarrow]{}{\phi}
                                                                           & C                                            &                  
\end{tikzcd}\]
\end{itemize}
 In particular, $\mathcal{C}$ is biequivalent to the pseudoslice $ \mathcal{C}/1$ where $ 1 $ is the biterminal object. However in our example, in the bicategory of Grothendieck toposes, $ \mathcal{S}et$ is a strict terminal object as the homcategories $ \GTop[\mathcal{E}, \Set$ are isomorphic to the terminal category. \\
 
 \begin{remark}\label{Pseudoslice is not pseudofunctorial}
 Observe that any 1-cell $u : C_1 \rightarrow C_2$ in $\mathcal{C}$ defines a pseudofunctor $ \mathcal{C}/u: \mathcal{C}/C_1 \rightarrow \mathcal{C}/C_2$ sending any $ f : A \rightarrow C_1$ on the composite $ uf : A \rightarrow C_2$. However this assignment does not extend to 2-cells: indeed, for a 2-cell $ \phi : u \Rightarrow u'$ in $\mathcal{C}$, one cannot define a pseudonatural transformation between $ \mathcal{C}/u$ and $ \mathcal{C}/u'$, for such a transformation would return morphisms in the \emph{lax} slice $ \mathcal{C}\Downarrow C_2$, as noted when examining the Beck-Chevalley condition of local right biadjoints.
 \end{remark}

\begin{remark}\label{biorthogonality in pseudoslices}

Factorization systems  are inherited be the pseudoslices as follows: for any $ C$ in $\mathcal{C}$, define $\mathcal{L}_C$, resp. $\mathcal{R}_C$, consisting of triangles $(l, \alpha)$ with $ l \in \mathcal{L}$ as on the left, resp. of triangles $(r,\alpha)$ with $ r \in \mathcal{R}$ as on the right:
\[ 
\begin{tikzcd}
A \arrow[rd, "a"'] \arrow[rr, "l \in \mathcal{L}"] & {} \arrow[d, "\alpha \atop \simeq", phantom] & B \arrow[ld, "b"] \\
                                   & C                                            &                  
\end{tikzcd} \quad 
\begin{tikzcd}
A \arrow[rd, "a"'] \arrow[rr, "r \in \mathcal{R}"] & {} \arrow[d, "\alpha \atop \simeq", phantom] & B \arrow[ld, "b"] \\
                                   & C                                            &                  
\end{tikzcd} \]

Now let us unfold the orthogonality condition in the context of a pseudoslice. Let $ (l : A_1 \rightarrow B_1, \lambda) : a_1 \rightarrow b_1$ and $(r : A_2 \rightarrow B_2,\rho) : a_2 \rightarrow b_2$ be two 1-cells in $ \mathcal{C}/C$; a pseudosquare from $ (l, \lambda)$ to $(r ,\rho)$ in $ \mathcal{C}/C$ consists of the data of two morphisms over $C$, $ (f : A_1 \rightarrow A_2, \beta) : a_1 \rightarrow a_2$ and $ (g : B_1 \rightarrow B_2, \gamma) : b_1 \rightarrow b_2$, and a 2-cell $ \alpha$ in $\mathcal{C}$ forming a pseudosquare $(f,g, \alpha) : l \rightarrow r$: 
\[ 
\begin{tikzcd}
A_1 \arrow[dd, "l"'] \arrow[rr, "f"] \arrow[rd, "a_1" description] & {} \arrow[d, "\beta\atop \simeq", phantom]                                          & A_2 \arrow[dd, "r"] \arrow[ld, "a_2" description] \\
{} \arrow[r, "\lambda \atop \simeq", phantom]   & C \arrow[d, "\gamma \atop \simeq", phantom] \arrow[r, "\rho \atop \simeq", phantom] & {}                             \\
B_1 \arrow[rr, "g"'] \arrow[ru, "b_1" description]                 & {}                                                                                  & B_2 \arrow[lu, "b_2" description]                
\end{tikzcd} \textrm{ with } 
\begin{tikzcd}
A_1 \arrow[dd, "l"'] \arrow[rr, "f"] \arrow[rrdd, "\alpha \atop \simeq", phantom] &  & A_2 \arrow[dd, "r"] \\
                                                                                  &  &                     \\
B_1 \arrow[rr, "g"']                                                              &  & B_2                
\end{tikzcd} \]
related as follows:
\[ 
\begin{tikzcd}
A_1 \arrow[rr, "f"] \arrow[rd, "a_1"'] & {} \arrow[d, "\beta \atop \simeq", phantom] & A_2 \arrow[ld, "a_2"] \\
                                       & C                                           &                      
\end{tikzcd} = 
\begin{tikzcd}
A_1 \arrow[d, "l"'] \arrow[rr, "f"]    & {} \arrow[d, "\alpha \atop \simeq"{near start}, phantom] & A_2 \arrow[d, "r"]    \\
B_1 \arrow[rr, "g"] \arrow[rd, "b_1"'] & {} \arrow[d, "\gamma \atop \simeq", phantom] & B_2 \arrow[ld, "b_2"] \\
                                       & C                                            &                      
\end{tikzcd} \textrm{ and } 
\begin{tikzcd}
A_1 \arrow[dd, "l"'] \arrow[rd, "a_1"]        &   \\
{} \arrow[r, "\lambda \atop \simeq", phantom] & C \\
B_1 \arrow[ru, "b_1"']                        &  
\end{tikzcd} = 
\begin{tikzcd}
A_1 \arrow[dd, "l"'] \arrow[r, "f"]          & A_2 \arrow[dd, "r" description] \arrow[rd, "a_2"] &    \\
{} \arrow[r, "\alpha \atop \simeq", phantom] & {} \arrow[r, "\rho \atop \simeq", phantom]        & {C} \\
B_1 \arrow[r, "g"]                           & B_2 \arrow[ru, "b_2"']                            &   
\end{tikzcd} \]

Now a morphism of pseudosquares $ ((f, \beta), (g, \gamma), \alpha) \Rightarrow ((f', \beta'), (g', \gamma'), \alpha')$ is the data of two 2-cells $ \phi : (f, \beta) \Rightarrow (f', \beta')$ and $ \psi : (g, \gamma) \Rightarrow (g',\gamma')$ in $\mathcal{C}/C$, such that $ (\phi, \psi) : (f,g, \alpha) \Rightarrow (f',g', \alpha')$ is a morphism of pseudosquare in $\mathcal{C}$.\\ 

Now requiring that $ l$ and $ r$ are orthogonal is the same as requiring that for any such pseudosquare as above in $\mathcal{C}/C$, there is a morphism $ (w : B_1 \rightarrow A_2, \omega) : b_1 \rightarrow a_2 $ and a pair of 2-cells $ \nu$, $ \mu$ such that $ (w, \nu, \mu)$ is a universal filler of the pseudosquare $ (f, g, \alpha)$ in $\mathcal{C}$ and the following 2-dimensional equations hold:
\[ \begin{tikzcd}
A_1 \arrow[rr, "f"] \arrow[rd, "a_1"'] & {} \arrow[d, "\beta \atop \simeq", phantom] & A_2 \arrow[ld, "a_2"] \\
                                       & C                                           &                      
\end{tikzcd} = 
\begin{tikzcd}
A_1 \arrow[rr, bend left=35, "f", "\nu \atop \simeq"'{inner sep=2pt}] \arrow[rd, "a_1"'] \arrow[r, "l" description]   &    B_1 \arrow[d, "b_1" description, "\omega \atop \simeq"{very near start, inner sep=7pt}, "\lambda \atop \simeq"'{very near start, inner sep=7pt}]      \arrow[r, "w" description]                                                                                      &  A_2 \arrow[ld, "a_2"]     \\ 
    & C  &
\end{tikzcd} \textrm{ and }  
\begin{tikzcd}
A_1 \arrow[r, "f"] \arrow[d, "l"']                & A_2 \arrow[d, "a_2"] \\
B_1 \arrow[ru, "w" description, " \nu \atop \simeq"{inner sep=3pt}, " \omega \atop \simeq"'{inner sep=3pt}] \arrow[r, "b_1"'] & C                   
\end{tikzcd} =  \begin{tikzcd}
A_1 \arrow[rd, "a_1" description, "\lambda \atop \simeq"'{inner sep=3pt}, "\beta \atop \simeq"{inner sep=3pt}] \arrow[r, "f"] \arrow[d, "l"']                & A_2 \arrow[d, "a_2"] \\
B_1  \arrow[r, "b_1"'] & C                   
\end{tikzcd} \] 
\[
\begin{tikzcd}
                                                                  & A_2 \arrow[dd, bend left=45, "a_2", "\rho \atop \simeq"'{inner sep=3pt}] \arrow[d, "r"]   \\
B_1 \arrow[r, "g" description, "\mu \atop \simeq"{inner sep=7pt, very near end}, "\gamma \atop \simeq"'{inner sep=7pt, very near end}] \arrow[rd, "b_1"'] \arrow[ru, "w"] & B_2 \arrow[d, "b_2"] \\
                                                                  & C                   
\end{tikzcd} = 
\begin{tikzcd}
&  A_2 \arrow[dd, "a_2", ] \\
B_1 \arrow[r, phantom, "\omega \atop \simeq" description] \arrow[ru, "w"] \arrow[rd, "b_1"'] & {} \\
                                       & C                              
\end{tikzcd} \textrm{ and }
\begin{tikzcd}
A_2 \arrow[r, "a_1"] \arrow[rd, "r" description, "\rho \atop \simeq"{inner sep=3pt}, "\mu \atop \simeq"'{inner sep=3pt}] & C                     \\
B_1 \arrow[u, "w"] \arrow[r, "g"']               & B_2 \arrow[u, "b_1"']
\end{tikzcd} = 
\begin{tikzcd}
A_2 \arrow[r, "a_1"]  & C                     \\
B_1 \arrow[u, "w"] \arrow[r, "g"']   \arrow[ru, "b_1" description, "\omega \atop \simeq"{inner sep=3pt}, "\gamma \atop \simeq"'{inner sep=3pt}]            & B_2 \arrow[u, "b_1"']
\end{tikzcd}
\]
with the desired uniqueness up to a unique invertible 2-cell, and moreover, for any morphism of pseudosquare $(\phi, \psi) : ((f,\beta),(g ,\gamma), \alpha) \Rightarrow ((f',\beta'),(g', \gamma'), \alpha')$, we get a unique 2-cell $ \sigma : (w, \omega) \Rightarrow (w', \omega')$ in $\mathcal{C}/C$ satisfying coherence conditions the reader might guess from the diagrams above. In particular, the domain functor $ \mathcal{C}/C \rightarrow \mathcal{C}$ sends a universal filler of a pseudosquare in $ \mathcal{C}/C$ to a universal filler of the underlying pseudosquare in $\mathcal{C}$.
\end{remark}

\begin{remark}\label{bifactorization in pseudoslice}
Now for a given 1-cell $ (f, \alpha) : a \rightarrow b$ in $ \mathcal{C}/C$, the factorization of $ f$ in $\mathcal{C}$ returns a factorization 
\[ 
\begin{tikzcd}[sep=large]
A \arrow[r, "l_f" description] \arrow[rr, "f", bend left=35, "\alpha_f \atop \simeq"'{inner sep=3pt}] \arrow[rd, "a"', " \alpha \alpha_f^{-1} \atop \simeq"{inner sep=1pt}] & C_f \arrow[r, "r_f" description] \arrow[d, "br_f" description] & B \arrow[ld, "b", "="'{inner sep=3pt}] \\
                                                                           & C                                                              &                  
\end{tikzcd} \]
whose right part has the identity $ 1_{br_f}$ as underlying 2-cell in $\mathcal{C}$, while the left part makes use of the inverse of $\alpha_f$ provided by the factorization up to an invertible 2-cell in $\mathcal{C}$, the uniqueness of which inducing the uniqueness of the factorization in $ \mathcal{C}/C$.
\end{remark} 

\begin{proposition}
Let $\mathcal{C}$ be a bicategory with biproducts endowed with a factorization system $(\mathcal{L},\mathcal{R})$: then for any $ C$ in $ \mathcal{C}$, the pseudoslice $ \mathcal{C}/C$ inherits a factorization system $ (\mathcal{L}_C, \mathcal{R}_C)$ as defined above.
\end{proposition}

\begin{proof}
We must prove first that $(\mathcal{L}_C, \mathcal{R}_C)$ is an orthogonality structure, that is $ \mathcal{L}_C = ^\perp\mathcal{R}_C$ and $ \mathcal{R}_C = \mathcal{L}_C^\perp$. It is easy to see that the domain functor $ \mathcal{C}/C \rightarrow \mathcal{C}$ sends $ \mathcal{R}_C$ into $ \mathcal{R}= \mathcal{L}^\perp$: indeed, if $ (r, \rho) : a \rightarrow b$ is in $\mathcal{L}_C^\perp$, then for an arrow $ l : A_1 \rightarrow A_2$ in $\mathcal{L}$ and a pseudosquare $ (f,g,\alpha) : l \rightarrow r$ in $\mathcal{C}$, pasting with $ \rho $ produces a pseudosquare $((f, 1_{af}), (g,1_{bg}), \alpha)$ in $ \mathcal{C}/C$
\[ 
\begin{tikzcd}
A' \arrow[dd, "l"'] \arrow[r, "f"]           & A \arrow[dd, "r" description] \arrow[rd, "a"] &   \\
{} \arrow[r, "\alpha \atop \simeq", phantom] & {} \arrow[r, "\rho \atop \simeq", phantom]    & C \\
B' \arrow[r, "g"]                            & B \arrow[ru, "b"]                             &  
\end{tikzcd} \]
Now, as $ (r, \rho)$ is in $ \mathcal{L}_C^\perp$, there exists a universal filler in $ \mathcal{C}/C$ which returns a universal filler for the pseudosquare $ (f,g,\alpha)$ in $\mathcal{C}$, ensuring that $ r$ is in $\mathcal{R}$ so that $ (r, \rho)$ was in $\mathcal{R}_C$. \\ 

The dual statement requires the existence of products in order to produce a pseudosquare over $C$ from an arbitrary pseudosquare in $\mathcal{C}$: let $ (l,\lambda) : a \rightarrow b$ be in $^\perp \mathcal{R}_C $, $r : A' \rightarrow B'$ in $\mathcal{R}$ and $ (f,g, \alpha) : l \rightarrow r$ a pseudosquare in $ \mathcal{C}/C$. Then by stability of $\mathcal{R}$-maps under bilimits, the product map $ (r,1_C) : A'\times C \rightarrow B'\times C$ is in $\mathcal{R}$, and pasting $ \alpha$ with the second projections returns a pseudosquare in $\mathcal{C}/C$ as below 
\[ 
\begin{tikzcd}
A \arrow[dd, "l"'] \arrow[r, "{(f,a)}"]      & A'\times C \arrow[dd, "{(r, 1_C)}" description] \arrow[rd, "p_2"] &   \\
{} \arrow[r, "\alpha \atop \simeq", phantom] & {} \arrow[r, "=", phantom]                                        & C \\
B \arrow[r, "{(g,b)}"']                      & B' \times C \arrow[ru, "p_2"']                                    &  
\end{tikzcd} = 
\begin{tikzcd}
A \arrow[dd, "l"'] \arrow[r, "{(f,a)}"] \arrow[rd, "a" description, "="{inner sep=3pt}] & A'\times C \arrow[d, "p_2"]   \\
{} \arrow[r, "\lambda \atop \simeq", phantom]                       & C                             \\
B \arrow[r, "{(g,b)}"'] \arrow[ru, "b" description, "="'{inner sep=3pt}]                 & B' \times C \arrow[u, "p_2"']
\end{tikzcd} \]
which admits a universal filler $ ((w,\omega), \nu, \mu)$ in $\mathcal{C}/C$; by pasting with the limit 2-cell associated to the first projections as below 
\[ 
\begin{tikzcd}[sep=large]
A \arrow[r, "{(f,a)}"] \arrow[d, "l"']              & A'\times C \arrow[r, "p_1"] \arrow[d, "{(r,1_C)}" description] \arrow[rd, "=", phantom] & A' \arrow[d, "r"] \\
B \arrow[r, "{(g,b)}"'] \arrow[ru, "w" description, "\nu \atop \simeq"{inner sep=3pt}, "\mu \atop \simeq"'{inner sep=3pt}] & B'\times C \arrow[r, "p_1"']                                                            & B'               
\end{tikzcd} \]
we get a filler $ (p_1 w, p_1^*\nu, p_1^*\mu)  $ for the pseudosquare $ (f,g,\alpha)$ in $ \mathcal{C}$. We must now check that it possesses the desired universal property. Suppose we have another filler $( w', \nu', \mu')$: then we get another filler 
\[ 
\begin{tikzcd}
A \arrow[r, "{(f,a)}"] \arrow[d, "l"']                    & A'\times C \arrow[d, "{(r,1_C)}"] \\
B \arrow[r, "{(g,b)}"'] \arrow[ru, "{(w',b)}" description] & B'\times C                                   
\end{tikzcd} \]
But then by uniqueness of the filler of the pseudocommutative square $ (((f,a), p_2(f,a) \simeq a), ((g,b), p_2(g,b) \simeq b), \alpha)$, there exists a unique invertible 2-cell $ \sigma : w \simeq (w',b) $ such that 
\[ 
\begin{tikzcd}
B \arrow[rr, "w" description] \arrow[rr, "{(w',b)}", bend left, end anchor=160, " \sigma \atop \simeq"'{inner sep=3pt}] \arrow[rd, "b"'] & {} \arrow[d, "\omega \atop \simeq", phantom] & A'\times C \arrow[ld, "p_2"] \\
                                                                                 & C                                            &                             
\end{tikzcd} = 
\begin{tikzcd}
B \arrow[rd, "b"'] \arrow[rr, "{(w',b)}"] & {} \arrow[d, "=", phantom] & A'\times C \arrow[ld, "p_2"] \\
                                          & C                          &                             
\end{tikzcd} \]
and by whiskering we get an invertible 2-cell $ p_1 ^* \sigma : p_1 (w', b) = w' \simeq p_1w$. We then must check the uniqueness of this invertible 2-cell: but this is a consequence of the fact that projections are jointly full and faithful morphisms.
\end{proof}
 
\begin{remark}\label{diagonalization of pseudosquares in pseudoslices}
The diagonalization of pseudosquares in $\mathcal{C}/C$ is obtained from the diagonalization of the underlying pseudosquare in $\mathcal{C}$: if $ ((f,\beta), (g,\gamma), \alpha) : (l, \lambda) \rightarrow (r,\rho)$ is a pseudosquare in $\mathcal{C}/C$ with $ l$ in $\mathcal{L}$ and $ r$ in $\mathcal{R}$, then the pseudosquare $(f,g,\alpha) : l \rightarrow r$ in $ \mathcal{C}$ has a unique diagonalization $(d_\alpha, \lambda_\alpha, \rho_\alpha)$. Then the invertible 2-cell $ \omega_\alpha$ under $ d_\alpha$ can be chosen to be the composite 
\[\begin{tikzcd}
	& {A_2} \\
	{B_1} & {B_2} \\
	&& C
	\arrow[""{name=0, anchor=center, inner sep=0}, "{d_\alpha}", from=2-1, to=1-2]
	\arrow["r", from=1-2, to=2-2]
	\arrow["g"{description}, from=2-1, to=2-2]
	\arrow[""{name=1, anchor=center, inner sep=0}, "{b_2}"{description}, from=2-2, to=3-3]
	\arrow[""{name=2, anchor=center, inner sep=0}, "{a_2}", curve={height=-18pt}, from=1-2, to=3-3]
	\arrow[""{name=3, anchor=center, inner sep=0}, "{b_1}"', curve={height=18pt}, from=2-1, to=3-3]
	\arrow["{\gamma\atop \simeq}"', Rightarrow, draw=none, from=1, to=3]
	\arrow["{\rho \atop \simeq}"', Rightarrow, draw=none, from=2-2, to=2]
	\arrow["{\rho_\alpha \atop \simeq}"{description}, Rightarrow, draw=none, from=0, to=2-2]
\end{tikzcd}\]
Indeed this composite will then satisfy the following equality as a consequence of the 2-dimensional equalities satisfied by the pseudosquare in $\mathcal{C}/C$ as explained in \cref{biorthogonality in pseudoslices}:
\[ \begin{tikzcd}
	{A_1} & {A_2} \\
	{B_1} & C
	\arrow[""{name=0, anchor=center, inner sep=0}, "{d_\alpha}"{description}, from=2-1, to=1-2]
	\arrow["{a_2}", from=1-2, to=2-2]
	\arrow["{b_1}"', from=2-1, to=2-2]
	\arrow["l"', from=1-1, to=2-1]
	\arrow["f", from=1-1, to=1-2]
	\arrow["{\omega_\alpha \atop \simeq}"{description}, Rightarrow, draw=none, from=0, to=2-2]
	\arrow["{\lambda_\alpha \atop \simeq}"{description}, Rightarrow, draw=none, from=1-1, to=0]
\end{tikzcd} = \begin{tikzcd}
	{A_1} & {A_2} \\
	{B_1} & {B_2} \\
	&& C
	\arrow[""{name=0, anchor=center, inner sep=0}, "{d_\alpha}"{description}, from=2-1, to=1-2]
	\arrow["r", from=1-2, to=2-2]
	\arrow["g"{description}, from=2-1, to=2-2]
	\arrow[""{name=1, anchor=center, inner sep=0}, "{b_2}"{description}, from=2-2, to=3-3]
	\arrow[""{name=2, anchor=center, inner sep=0}, "{a_2}", curve={height=-18pt}, from=1-2, to=3-3]
	\arrow[""{name=3, anchor=center, inner sep=0}, "{b_1}"', curve={height=18pt}, from=2-1, to=3-3]
	\arrow["l"', from=1-1, to=2-1]
	\arrow["f", from=1-1, to=1-2]
	\arrow["{\gamma\atop \simeq}"', Rightarrow, draw=none, from=1, to=3]
	\arrow["{\rho \atop \simeq}"', Rightarrow, draw=none, from=2-2, to=2]
	\arrow["{\rho_\alpha \atop \simeq}"{description}, Rightarrow, draw=none, from=0, to=2-2]
	\arrow["{\lambda_\alpha \atop \simeq}"{description}, Rightarrow, draw=none, from=1-1, to=0]
\end{tikzcd}=\begin{tikzcd}
	{A_1} & {A_2} \\
	{B_1} & C
	\arrow["{a_2}", from=1-2, to=2-2]
	\arrow["{b_1}"', from=2-1, to=2-2]
	\arrow["l"', from=1-1, to=2-1]
	\arrow["f", from=1-1, to=1-2]
	\arrow[""{name=0, anchor=center, inner sep=0}, from=1-1, to=2-2]
	\arrow["{\lambda \atop \simeq}"{description}, Rightarrow, draw=none, from=0, to=2-1]
	\arrow["{\beta \atop \simeq}"{description}, Rightarrow, draw=none, from=1-2, to=0]
\end{tikzcd}\]
\end{remark}

 \subsection{Left and right objects}
 
\begin{definition}
Define the bicategory of left objects over $C$, denoted $\mathcal{L}\hy\textbf{Obj}_C$, (resp. the bicategory of right objects over $C$, denoted $\mathcal{R}\hy\textbf{Obj}_C) $) as the subbicategory  of $ \mathcal{C}/C$  whose \begin{itemize}
    \item 0-cells are $\mathcal{L}$-maps $ l : A \rightarrow C$ (resp. $ \mathcal{R}$-maps $ r : A \rightarrow C$)
    \item 1-cells are invertible 2-cells $ (a, \alpha) : l \rightarrow l'  $ with $a \in \mathcal{L}$ (resp. $(b, \beta) : r \rightarrow r' $ with $ b \in \mathcal{R}$)
    \item 2-cells are the same as in $ \mathcal{C}/C$
\end{itemize}
In particular we denote by $ \mathcal{L}\hy\textbf{Obj}$ (resp. $\mathcal{R}\hy\textbf{Obj}$) the categories of left and right objects over the biterminal object $1$.
\end{definition}

\begin{remark}
In the factorization of an $ \mathcal{L}$-map $ f : A \rightarrow B$ (resp. an $\mathcal{R}$-map), the right (resp. left) part, which is an equivalence, can be chosen to be the identity 2-cell $(1_A, 1_{f})$ which has the desired universal property of the factorization. Indeed for any factorization of $ f$, the right part is an equivalence by left cancellation of left maps, so that it possesses a weak inverse $ e : B \stackrel{\simeq}{\rightarrow} C_f$ which is the unique up to unique invertible 2-cell equivalence provided by the factorization through the identity, pictured as follows
\[
\begin{tikzcd}
                                                                                & B \arrow[rd, equal] \arrow[dd, "e" description] &                                                &     \\
A \arrow[rd, "l_f"'] \arrow[ru, "f"] \arrow[r, "\lambda \atop \simeq", phantom] & {} \arrow[r, "\simeq", phantom]                               & B \arrow[rd, "e"] \arrow[d, "\simeq", phantom] &     \\
                                                                                & C_f \arrow[ru, "r_f", description] \arrow[rr, equal]        & {}                                             & C_f
\end{tikzcd} \] 
where the two unnamed invertible 2-cells are the unit and counit of the equivalence $(r_f, e)$, which are uniquely determined, while $ \lambda$ itself is uniquely determined by the choice of the factorization. Hence one can choose $ (f,1_B)$ to be the factorization of $f$. The dual property is true if we suppose $ f$ to be a right maps.  
\end{remark}

\begin{remark} Alike the 1-dimensional case, we have the following properties:
\begin{itemize}
    \item { $1_C$ is the only object to be both left and right over $C$ up to unique equivalence. }
    \item $ \mathcal{R}\hy\textbf{Obj}_C$ is a 1,2-full subcategory of $\mathcal{C}/C$ by right cancellation of $\mathcal{R}$-maps.
    \item In particular, for $ r : A \rightarrow C $, any $ 1_C \rightarrow r$ is in $\mathcal{R}$ ($C$-points are right).
\end{itemize}
\end{remark}

\begin{proposition}\label{Right objects are relflective}
The 1,2-full subbicategory $ \mathcal{R}\hy\textup{\textbf{Obj}}_C$ is pseudo-reflective in $ \mathcal{C}/C$. In particular $ \mathcal{R}\hy\textup{\textbf{Obj}}$ is pseudo-reflective in $\mathcal{C}$
\end{proposition}

\begin{proof}
The pseudo-reflection returns on 0-cells $ f : A \rightarrow C$ the right part $ r_f : C_f \rightarrow C$, and on 1 (resp. 2-cells), it returns the middle 1-cell and the associated right invertible 2-cell (resp. the morphism of fillers between the right invertible cells) provided by the  pseudo-functoriality of the factorization: that is, $ (u, \alpha) : l \rightarrow g$ in $\mathcal{C}$ is sent to the 1-cell $ (w_\alpha, \rho_\alpha)$ as below
\[ 
\begin{tikzcd}
A \arrow[rr, "u"] \arrow[rd, "f"'] & {} \arrow[d, "\alpha \atop \simeq", near start, phantom] & B \arrow[ld, "g"] \\
                                   & C                                            &                  
\end{tikzcd} \mapsto 
\begin{tikzcd}
C_f \arrow[rd, "r_f"'] \arrow[rr, "w_\alpha"] & {} \arrow[d, "\rho_\alpha \atop \simeq", near start, phantom] & C_g \arrow[ld, "r_g"] \\
                                              & C                                            &                      
\end{tikzcd}  \]
where $ w_\alpha$ is forced to be a $\mathcal{R}$-map by right cancellation of $\mathcal{R}$-maps (similarly, a 2-cell $ \phi : (u, \alpha) \Rightarrow (u', \alpha') $ is sent to $ \sigma : w_\alpha \Rightarrow w_{\alpha'} $ as provided above).\\

For a $\mathcal{R}$-object over $C$ $ r : B \rightarrow C$, whose factorization can be identified with the composite $ (1_B, r)$ up to a unique pair of 2-cells and an equivalence which unique up to unique invertible 2-cell, and for a map $ f : A \rightarrow C$, we claim there is a equivalence of categories 
\[ \mathcal{C}/C [f, r] \simeq \mathcal{R}\hy\textbf{Obj}_C[r_f, r] \]
sending: \begin{itemize}
    \item a 1-cell $ (u, \alpha) : f \rightarrow r $ to the composite 2-cell
    \[ 
\begin{tikzcd}[column sep= large]
C_f \arrow[rd, "r_f"'] \arrow[r, "w_\alpha"] & C_r \arrow[d, "r_r"] \arrow[r, "e"] \arrow[rd, "\iota \atop \simeq", phantom, very near start] \arrow[ld, "\rho_\alpha \atop \simeq", phantom, very near start] & R \arrow[ld, "r"] \\
{}                               & C                                                                                                                             & {}               
\end{tikzcd} \]
where $e $ is the equivalence, unique up to unique invertible 2-cell, relating the trivial factorization $ (1_R, r)$ with $(l_f, r_r)$ as pictured below
\[ 
\begin{tikzcd}[column sep= large]
R \arrow[d, "l_r"'] \arrow[rd, equal]                                                       & {}                \\
C_r \arrow[d, "r_r"'] \arrow[r, "e"] \arrow[ru, " \epsilon \atop \simeq", phantom, very near start] \arrow[rd, "\iota \atop \simeq", phantom, very near start] & R \arrow[ld, "r"] \\
C                                                                                                & {}               
\end{tikzcd} \]
that is, to the composite  $ (ew_\alpha, i^*w_\alpha\rho_\alpha) : r_f \rightarrow r $, and a 2-cell $ \phi : (u_1, \alpha_1) \Rightarrow (u_2, \alpha_2)$ to the induced 2-cell between fillers $ \sigma : (w_{\alpha_1}, \rho_{\alpha_1}) \Rightarrow (w_{\alpha_2}, \rho_{\alpha_2}) $ in one direction,
    \item a 1-cell $ (v,\beta) : r_f \rightarrow r $ to the 2-cell obtained as the pasting $ l_f^*\beta \alpha_f : $
    \[ 
\begin{tikzcd}[column sep=large]
A \arrow[rd, "f"'] \arrow[r, "l_f"] & C_f \arrow[d, "r_f" description] \arrow[r, "v"] \arrow[rd, "\beta \atop \simeq", very near start, phantom] \arrow[ld, "\alpha_f \atop \simeq", very near start, phantom] & R \arrow[ld, "r"] \\
{}                                  & C                                                                                                            & {}               
\end{tikzcd}  \] 
and a 2-cell $ \psi : (v_1, \beta_1) \Rightarrow (v_2, \beta_2)$ to $ l_f ^*\psi : (l_f^*v_1, \alpha_1 \alpha_f) \Rightarrow (l_f^*v_2, \alpha_2 \alpha_f)  $. 
\end{itemize}   

Those processes are easily seen to be mutual inverse up to natural equivalence: \begin{itemize}
    \item If one starts with $ (u, \alpha) : f \rightarrow r$, then going back and forth produces a diagram 
    \[ 
\begin{tikzcd}[column sep= large]
                                                                 & {}                                                                                                                                         & R \arrow[d, "l_r"'] \arrow[rd, equal]                                                                                                       & {}                \\
C \arrow[r, "l_f"] \arrow[rrd, "f"', bend right] \arrow[rru, bend left, "u"] & C_f \arrow[r, "w_\alpha"] \arrow[rd, "r_f"'] \arrow[d, "\alpha_f \atop \simeq", phantom] \arrow[u, "\lambda_\alpha \atop \simeq", phantom] & C_r \arrow[d, "r_r" description] \arrow[r, "e"] \arrow[ru, "\epsilon \atop \simeq", phantom, very near start] \arrow[rd, "\iota \atop \simeq", phantom, very near start] \arrow[ld, "\rho_\alpha \atop \simeq", phantom, very near start] & R \arrow[ld, "r"] \\
                                                                 & {}                                                                                                                                         & C                                                                                                                                                & {}               
\end{tikzcd} \]
whose pasting defines an invertible 2-cell $  \epsilon^*u \lambda_\alpha : (u, \alpha) \simeq (e w_\alpha l_f, \iota^*(w_\alpha l_f) \rho_\alpha^*l_f \alpha_f) $. 
\item If conversely one starts with $(v,\beta) : r_f \rightarrow r$, then going forth and back returns a composite 2-cell
\[ 
\begin{tikzcd}[column sep= large]
C_f \arrow[r, "w_{\beta^*l_f \alpha_f}"] \arrow[rd, "r_f"'] & C_r \arrow[d, "r_r" description] \arrow[r, "e"] \arrow[rd, "\iota \atop \simeq", phantom, very near start] \arrow[ld, "\rho_{\beta^*l_f \lambda_f} \atop \simeq", phantom, very near start] & R \arrow[ld, "r"] \\
{}                                                           & C                                                                                                                                               & {}               
\end{tikzcd} \]
produced by the functorial factorization of the pseudosquare $(vl_f, 1_C, \beta^l_f \alpha_f) : f \Rightarrow r$ in $ ps[2, \mathcal{C}]$; 
but the diagram 
\[ 
\begin{tikzcd}
C \arrow[rr, "vl_f"] \arrow[d, "l_f"'] \arrow[rrd, "\simeq", phantom] &                                                                                                                                   & R \arrow[d, "l_r"]    \\
C \arrow[rd, "r_f"'] \arrow[r, "v"]                                   & R \arrow[d, "r" description] \arrow[r, "l_r"] \arrow[rd, "\iota \atop \simeq", phantom, very near start] \arrow[ld, "\beta \atop \simeq", phantom, very near start] & C_r \arrow[ld, "r_r"] \\
{}                                                                    & C                                                                                                                                 & {}                   
\end{tikzcd} \]
(where the upper square is the associator of $ l_r v l_f$) provides us with another functorial factorization of the pseudosquare $(vl_f, 1_C, \beta^*l_f \alpha_f)$, inducing a unique invertible 2-cell 
\[ 
\begin{tikzcd}
                                                                                     & {} \arrow[dd, "\exists ! \sigma \atop \simeq", phantom] &     \\
C_f \arrow[rr, "w_{\beta^*l_f \alpha_f}", bend left] \arrow[rr, "l_rv"', bend right] &                                                         & C_r \\
                                                                                     & {}                                                      &    
\end{tikzcd} \]
Pasting the whiskering of $ e$ with $ \sigma $ and the counit $ \epsilon^*v : e l_r v \Rightarrow v$ of the equivalence $ e \dashv l_r $ in $v $ returns the desired invertible 2-cell 
\[ 
\begin{tikzcd}
                                                                                   & {} \arrow[dd, "\epsilon^*v e^*\sigma \atop \simeq", phantom] &   \\
C_f \arrow[rr, "ew_{\beta^*l_f \alpha_f}", bend left] \arrow[rr, "v"', bend right] &                                                              & R \\
                                                                                   & {}                                                           &  
\end{tikzcd} \]
\end{itemize}
\end{proof}

Now we turn to the an alternative way to construct bistable pseudofunctor from a bifactorization system:

\begin{theorem}\label{Bistable inclusion of left objects}
Let be $ (\mathcal{L},\mathcal{R})$ be a bifactorization system in a bicategory $\mathcal{C}$ with a biterminal object. Then:
\begin{itemize}
    \item we have a bistable inclusion 
\[\begin{tikzcd}
	{\mathcal{L}\hy\textup{\textbf{Obj}}^{\op}} & {\mathcal{C}^{\op}}
	\arrow["{\iota_\mathcal{L}^{\op}}", hook, from=1-1, to=1-2]
\end{tikzcd}\]
   \item for each object $C$ we have a bistable inclusion 
   \[\begin{tikzcd}
	{\mathcal{L}\hy\textup{\textbf{Obj}}_C^{\op}} & {\mathcal{C}/C^{\op}}
	\arrow["{\iota_{\mathcal{L},C}^{\op}}", hook, from=1-1, to=1-2]
\end{tikzcd}\]
\end{itemize}

\end{theorem}

\begin{proof}
The first item is the instance of the second item where $ C$ is the biterminal object. In this case, generic morphisms are the duals of right morphisms of the form 
\[\begin{tikzcd}
	A && B \\
	& C
	\arrow[""{name=0, anchor=center, inner sep=0}, "r", from=1-1, to=1-3]
	\arrow["b", from=1-3, to=2-2]
	\arrow["l"', from=1-1, to=2-2]
	\arrow["{\rho\atop \simeq}"{description}, Rightarrow, draw=none, from=0, to=2-2]
\end{tikzcd}\]
Indeed, for any morphism between left objects as below
\[\begin{tikzcd}
	{A_1} && {A_2} \\
	& C
	\arrow[""{name=0, anchor=center, inner sep=0}, "l", from=1-1, to=1-3]
	\arrow["{a_1}"', from=1-1, to=2-2]
	\arrow["{a_2}", from=1-3, to=2-2]
	\arrow["{\lambda \atop \simeq}"{description}, Rightarrow, draw=none, from=0, to=2-2]
\end{tikzcd}\]
with $ a_1, a_2$ and $l$ in $\mathcal{L}$. Now if one has a pseudosquare $ ((m, \mu),(f, \beta), \alpha) : (l,\lambda) \rightarrow (r,\rho)$ as below
\[\begin{tikzcd}[sep=small]
	{A_1} && A \\
	& C
	\arrow[""{name=0, anchor=center, inner sep=0}, "m", from=1-1, to=1-3]
	\arrow["{a_1}"', from=1-1, to=2-2]
	\arrow["a", from=1-3, to=2-2]
	\arrow["{\mu \atop \simeq}"{description}, Rightarrow, draw=none, from=0, to=2-2]
\end{tikzcd} = \begin{tikzcd}[sep=small]
	{A_1} && A \\
	{A_2} && B \\
	& C
	\arrow["m", from=1-1, to=1-3]
	\arrow["l"', from=1-1, to=2-1]
	\arrow[""{name=0, anchor=center, inner sep=0}, "f"{description}, from=2-1, to=2-3]
	\arrow["r", from=1-3, to=2-3]
	\arrow["{a_2}"', from=2-1, to=3-2]
	\arrow["b", from=2-3, to=3-2]
	\arrow["{\alpha \atop \simeq}"{description}, draw=none, from=2-1, to=1-3]
	\arrow["{\beta \atop \simeq}"{description}, Rightarrow, draw=none, from=0, to=3-2]
\end{tikzcd} \textrm{ and } \begin{tikzcd}[sep=small]
	{A_1} & A & \\
	&& C \\
	{A_2} & B &
	\arrow["m", from=1-1, to=1-2]
	\arrow[""{name=0, anchor=center, inner sep=0}, "r"{description}, from=1-2, to=3-2]
	\arrow["l"', from=1-1, to=3-1]
	\arrow["f"', from=3-1, to=3-2]
	\arrow["a", from=1-2, to=2-3]
	\arrow["b"', from=3-2, to=2-3]
	\arrow["{\alpha \atop \simeq}"{description}, draw=none, from=1-1, to=3-2]
	\arrow["{\rho \atop \simeq}"{description}, Rightarrow, draw=none, from=0, to=2-3]
\end{tikzcd} = \begin{tikzcd}[sep=small]
	A \\
	& C \\
	B
	\arrow[""{name=0, anchor=center, inner sep=0}, "r"', from=1-1, to=3-1]
	\arrow["a", from=1-1, to=2-2]
	\arrow["b"', from=3-1, to=2-2]
	\arrow["{\rho \atop \simeq}"{description}, Rightarrow, draw=none, from=0, to=2-2]
\end{tikzcd}\]
with $ m$ itself in $\mathcal{L}$. Then by \cref{diagonalization of pseudosquares in pseudoslices} we know that the diagonalization of the underlying pseudosquare $ (m,f, \alpha)$ in $\mathcal{C}$ can be used to produce the diagonalization in the pseudoslice $ \mathcal{C}$, with the 2-cell below it being the composite 
\[\begin{tikzcd}
	& A \\
	{A_2} & B \\
	&& C
	\arrow[""{name=0, anchor=center, inner sep=0}, "{l_\alpha}", from=2-1, to=1-2]
	\arrow["r"{description}, from=1-2, to=2-2]
	\arrow["f"{description}, from=2-1, to=2-2]
	\arrow["b"{description}, from=2-2, to=3-3]
	\arrow[""{name=1, anchor=center, inner sep=0}, "a", curve={height=-18pt}, from=1-2, to=3-3]
	\arrow[""{name=2, anchor=center, inner sep=0}, "{a_2}"', curve={height=18pt}, from=2-1, to=3-3]
	\arrow["{\rho_\alpha \atop \simeq}"{description}, Rightarrow, draw=none, from=0, to=2-2]
	\arrow["{\lambda \atop \simeq}"{description}, Rightarrow, draw=none, from=1, to=2-2]
	\arrow["{\beta \atop \simeq}"{description}, Rightarrow, draw=none, from=2-2, to=2]
\end{tikzcd}\]
but as $ \lambda_\alpha$ is invertible, we have that $ l_\alpha$ is in $\mathcal{L}$ by right cancellation. Hence the 2-cell above defines a left map in $\mathcal{L}_C$ in the pseudoslice $\mathcal{C}/C$. Moreover the 2-dimensional property of the generic morphisms follows from the 2-dimensional property of the diagonalization.\\

Finally, for the factorization is inherited in $\mathcal{C}/C$, we have, as seen in \cref{bifactorization in pseudoslice}, that a morphism in the pseudoslice factorizes as below
\[\begin{tikzcd}[sep=large]
	A & {A_f} & B \\
	& C
	\arrow["{l_f}", from=1-1, to=1-2]
	\arrow["{r_f}", from=1-2, to=1-3]
	\arrow[""{name=0, anchor=center, inner sep=0}, "a"', from=1-1, to=2-2]
	\arrow["{a_f}"{description}, from=1-2, to=2-2]
	\arrow[""{name=1, anchor=center, inner sep=0}, "b", from=1-3, to=2-2]
	\arrow["{\lambda_{(f,\alpha)} \atop \simeq}"{description, pos=0.4}, Rightarrow, draw=none, from=1-2, to=0]
	\arrow["{\rho_{(f,\alpha)}  \atop \simeq}"{description, pos=0.4}, Rightarrow, draw=none, from=1-2, to=1]
\end{tikzcd}\]
where $ l_f, r_f$ is the $(\mathcal{L},\mathcal{R})$ factorization of $f$ in $\mathcal{C}$. Hence if $ a$ is in $\mathcal{L}$, so is $ a_f$ by right cancellation, and $ \rho_{(f,\alpha)} : a_f \rightarrow b$ is dual to a generic morphism. This proves the inclusion above to be bistable. 
\end{proof}

\subsection{Specification of left objects}

In the examples listed at the end of this work, the bistable inclusion obtained from bifactorizations systems does not encode by themselves exactly what we would expect from a geometry: indeed, those bifactorizations systems are all part of a ``connected-truncated" paradigm where the left maps code for some notion of connectedness rather than punctualness, and are somewhat too coarse to represent focal data. However, this can be corrected by choosing, amongst left maps, a subclass of left maps with a convenient cancellation property which will ensure that one can restrict to left objects with a specified left map as local objects, and arbitrary left maps between them. In particular, while left maps are seldom - but sometime, see below the hyperconnected-localic case - stable under bipullbacks, we can choose the specified left maps to be so to encode more local-like behavior. 

\begin{definition}
Let be $ \mathcal{L}$ a class of 1-cell in a bicategory $ \mathcal{C}$ and $ \mathcal{L}' \subseteq \mathcal{L}$ a subclass. We say that $ \mathcal{L}'$ has \emph{right cancellation along $\mathcal{L}$-maps} if for any invertible 2-cell as below
\[\begin{tikzcd}
	A && B \\
	& C
	\arrow[""{name=0, anchor=center, inner sep=0}, "l", from=1-1, to=1-3]
	\arrow["a"', from=1-1, to=2-2]
	\arrow["f", from=1-3, to=2-2]
	\arrow["{\alpha \atop \simeq}"{description}, Rightarrow, draw=none, from=0, to=2-2]
\end{tikzcd}\]
where $ a$ is in $\mathcal{L}'$ and $ l$ is in $\mathcal{L}$ then $f$ is in $\mathcal{L}'$. 
\end{definition}

Again, in the presence of a biterminal object, we can define a class of $ \mathcal{L}'$-objects as those objects whose terminal map is in $\mathcal{L}'$, and also the relativized version over arbitrary objects. The key idea is that left cancellation along $\mathcal{L}$-maps encodes exactly the gliding condition required to have bistable factorization. In the following we denote as ${\mathcal{L}'\hy\textup{\textbf{Obj}}^{\mathcal{L}}}$ (resp. $\mathcal{L}'\hy\textup{\textbf{Obj}}^{\mathcal{L}}$) the (1,2)-full sub-bicategory of $ \mathcal{L}\hy\textbf{Obj}$ (resp. $ \mathcal{L}\hy\textbf{Obj}_C$) whose objects are $\mathcal{L}'$-objects (resp. $ \mathcal{L}'$-maps with codomain $C$). 

\begin{theorem}\label{bistable for specified left objects}
Let be $ (\mathcal{L},\mathcal{R})$ a bifactorization system on a bicategory $\mathcal{C}$ and $\mathcal{L}' \subseteq \mathcal{L}$ with right cancellation along $\mathcal{L}$-maps. Then:
\begin{itemize}
    \item if $\mathcal{C}$ has a biterminal object and $ \mathcal{L}'$ contains identities then we have a bistable inclusion 
\[\begin{tikzcd}
	({\mathcal{L}'\hy\textup{\textbf{Obj}}^{\mathcal{L}})^{\op}} & {\mathcal{C}^{\op}}
	\arrow["{\iota_\mathcal{L}^{\op}}", hook, from=1-1, to=1-2]
\end{tikzcd}\]
   \item For each object $C$ we have a bistable inclusion 
   \[\begin{tikzcd}
	{(\mathcal{L}'\hy\textup{\textbf{Obj}}^{\mathcal{L}_C}_C)^{\op}} & {\mathcal{C}/C^{\op}}
	\arrow["{\iota_{\mathcal{L},C}^{\op}}", hook, from=1-1, to=1-2]
\end{tikzcd}\]
\end{itemize}
\end{theorem}

\begin{proof}
Again the first item is a specific case of the second item. The generic morphisms are now dual to invertible 2-cells
\[\begin{tikzcd}
	A && B \\
	& C
	\arrow[""{name=0, anchor=center, inner sep=0}, "r", from=1-1, to=1-3]
	\arrow["a"', from=1-1, to=2-2]
	\arrow["b", from=1-3, to=2-2]
	\arrow["{\alpha \atop \simeq}"{description}, Rightarrow, draw=none, from=0, to=2-2]
\end{tikzcd}\]
with $ a$ in $\mathcal{L}'$ and $ r$ in $\mathcal{R}$: the genericness is a restriction of the genericness of right maps toward left objects. Now just use the right cancellation of $\mathcal{L}'$-objects along $ \mathcal{L}$-map to ensure that the $ (\mathcal{L}_C, \mathcal{R}_C)$ factorization the middle object is in $\mathcal{L}'$. 
\end{proof}

\section{Bifactorization geometries for Grothendieck topoi}

In this section, we apply \cref{Bistable inclusion of left objects} to produce bistable sub-bicategories of the opposite bicategory of Grothendieck topoi - that are, bistable sub-bicategories of the bicategory of \emph{logoi}. \\

We investigate examples for the following bifactorizations systems:\begin{itemize}
    \item $(\textbf{Hyp}, \textbf{Loc})$ in $\GTop$, the \emph{hyperconnected-localic} bifactorization system on the bicategory of Grothendieck topoi and geometric morphisms
    \item $ (\textbf{Co}, \textbf{Et})$ in $\textbf{Loco}\GTop$, the \emph{connected-etale} bifactorization system on the bicategory of locally connected Grothendieck topoi and \emph{locally connected} geometric morphisms; this bifactorization system is to be used together with the following specifications of left maps \begin{itemize}
        \item $\textbf{Foc}$, the class of \emph{local} geometric morphism
        \item $\textbf{Toc}$, the class of \emph{totally connected} geometric morphisms
    \end{itemize}
    \item $(\textbf{TCo}, \textbf{Et})$ on $\textbf{Loco}\GTop^{\textbf{Ess}}$, the \emph{terminally connected-etale} bifactorization system on the bicategory of locally connected Grothendieck topoi and \emph{essential} geometric morphisms; this bifactorization system generalizes the previous one and is to be used with the same specifications as above. 
    \item Though we do not know yet a characterization of the right class, we prove there exists an orthogonal bifactorization in $ \GTop$ with terminally connected geometric morphisms on the left; this is also to be used with the local and totally connected specification. 
    \item $(\textbf{Surj}, \textbf{Emb})$ in $\GTop$, the \emph{surjection-embedding} factorization system in the bicategory of Grothendieck topoi.
\end{itemize}

The key idea of this section relies on the observation that the geometric properties of Grothendieck topoi are often encoded in their terminal map above $\Set$. \\

In all the following examples, we will end with bistable factorization from (opposite bicategories of) left objects and left maps - with eventual specification of left objects - with the right maps toward left objects as generic morphisms. However it is interesting to notify that in none of those examples, the right maps are lax-generic; in fact, those factorization systems all are instances of bifactorization systems with right-truncated lax-factorizations, so that in particular diagonalization of lax squares returns diagonal maps in the right class, not in the left class. This seems to be related to the fact that all those factorization systems are part of a ``connected-truncated" scheme where the right class exhibits some level of truncation as opposed to the left class, so that 2-dimensional data are ``expelled" on the left side in the factorization of globular 2-cells. We do not know currently a bifactorization systems for geometric morphisms such that the duals of right maps have the lax-generic property. \\

We also conjecture that the terminally connected-etale bifactorization generalizes to a bifactorization system on the category of Grothendieck topoi without restrictions on geometric morphisms, where the left maps will be part of a wider class of \emph{pro-etale geometric morphisms}; this factorization would be probably related to the connected-disconnected factorization mentioned in \cite{topologie}[3.2.15], and will certainly involve a generalization of Grothendieck-Verdier localization. However we do not know yet whether the local geometric morphisms have right cancellation along terminally connected geometric morphisms in the general case - nor for the totally connected geometric morphisms.  

\subsection{The hyperconnected-localic geometry }

This is a situation where the left maps are both sufficient to encode the local objects and the map between them for they enjoy good properties. We first recall some generalities about those two classes of maps. 

\begin{definition}
Let $\mathcal{E}$ be a topos; a \emph{subquotient} of an object $E$ is an object $X$ related to $E$ by either a span or a cospan as follows
\[\begin{tikzcd}
	S & X \\
	E
	\arrow["m"', tail, from=1-1, to=2-1]
	\arrow["q", two heads, from=1-1, to=1-2]
\end{tikzcd} \hskip1cm 
\begin{tikzcd}
	& X \\
	E & Q
	\arrow["q", two heads, from=2-1, to=2-2]
	\arrow["m", tail, from=1-2, to=2-2]
\end{tikzcd}\]
with $m$ a mono and $q$ an epi.
\end{definition}

\begin{remark}
Those two conditions are equivalent: if $X$ a subobject of a quotient of $E$, the pullback exhibits $X$ as a quotient of a subobject by stability of epi and mono under pullback in a topos, for topos are regular categories; if $X$ is a quotient of a subobject, the pushout exhibits $X$ as a subobject of a quotient by stability of mono and epi under pushout for topos are adhesive categories.  
\end{remark}

\begin{definition}
A geometric morphism $ f : \mathcal{E} \rightarrow \mathcal{F}$ is said to be \begin{itemize}
    \item \emph{hyperconnected} if it is connected ($f^*$ is full and faithful) and the essential image of $ f^*$ is closed in $\mathcal{E}$ under subquotient
    \item \emph{localic} if the essential image of $f^*$ generates $ \mathcal{E}$ under subquotient.
\end{itemize}
\end{definition}

Localic morphisms are related to the notion of internal locale; we choose here not to define the notion of internal locale, as the internale version of the completenes under arbitrary suprema is tiresome to express; for this we refer to \cite{elephant}[C.1]. For any Grothendieck topos $ \mathcal{E}$ we can define the locally posetal 2-category $ \textbf{Loc}[\mathcal{E}]$ of internal locales in $\mathcal{E}$; in particular internal locales in $\Set$ are ordinary locales. For any topos, the subobject classifier is in particular an internal locale. Similarly, it is well known that the direct image part of a geometric morphism lift to the 2-categories of internal locales.  

\begin{proposition}
A geometric morphism $f : \mathcal{E} \rightarrow \mathcal{F}$ is localic if and only if there exists an internal local $ L_f $ in $\textbf{Loc}[\mathcal{F}]$ such that one has a geometric equivalence together with an invertible 2-cell as below
\[\begin{tikzcd}
	{\mathcal{E}} && {\Sh_\mathcal{F}(L_f)} \\
	& {\mathcal{F}}
	\arrow[""{name=0, anchor=center, inner sep=0}, "f"', from=1-1, to=2-2]
	\arrow["\simeq"{description}, from=1-1, to=1-3]
	\arrow[""{name=1, anchor=center, inner sep=0}, "{l_f}", from=1-3, to=2-2]
	\arrow["\simeq", Rightarrow, draw=none, from=0, to=1]
\end{tikzcd}\]
In particular we have an equivalence of categories
\[ \textbf{Loc}/\mathcal{F} \simeq \textbf{Loc}[\mathcal{F}]  \]
\end{proposition}

\begin{proposition}
Hyperconnected geometric morphisms are left bi-orthogonal to localic geometric morphisms. 
\end{proposition}

\begin{proposition}\label{Hyperconnected-localic factorization}
Hyperconnected and localic geometric morphisms from a bifactorization system on $\GTop$: any geometric morphism factorizes in an essentially unique way as an hyperconnected geometric morphism followed by a localic geometric morphism.
\end{proposition}

\begin{remark}
The factorization can be alternatively described in two ways:\begin{itemize}
    \item from the point of view of the subquotients, the factorization is obtained from the factorization of the inverse image part in $\Cat$
\[\begin{tikzcd}
	{\mathcal{E}} && {\mathcal{F}} \\
	& {\overline{\textup{Im} f}}
	\arrow[""{name=0, anchor=center, inner sep=0}, "{f^*}"', from=1-3, to=1-1]
	\arrow["{l_f^*}", from=1-3, to=2-2]
	\arrow["{h_f^*}", hook, from=2-2, to=1-1]
	\arrow["{\alpha_f \atop \simeq}"{description}, Rightarrow, draw=none, from=2-2, to=0]
\end{tikzcd}\]
where $ \overline{\textup{Im} f}$ is the closure of the essential image of $ f^*$ in $\mathcal{E}$ under subquotient, which can be proven directly to be a Grothendieck topos.
\item from the point of view of internal locales, the factorization is obtained from the direct image of the subobject classifier as 
\[\begin{tikzcd}
	{\mathcal{E}} && {\mathcal{F}} \\
	& {\Sh_\mathcal{F}(f_*\Omega_\mathcal{E})}
	\arrow[""{name=0, anchor=center, inner sep=0}, "f", from=1-1, to=1-3]
	\arrow["{l_f}"', from=2-2, to=1-3]
	\arrow["{h_f}"', from=1-1, to=2-2]
	\arrow["{\alpha_f \atop \simeq}"{description}, Rightarrow, draw=none, from=2-2, to=0]
\end{tikzcd}\]
That is, $ f_*\Omega_\mathcal{E}$ is the desired internal local corresponding to the localic geometric morphism $l_f$.
\end{itemize}
\end{remark}

\begin{definition}
A Grothendieck topoi will be said to be \emph{localic} if it is of the form $ \Sh(L)$ for $L$ a locale - that is, an internal locale in $\Set$. In the following, we denote as $ \textbf{LocGTop}$ the bicategory of localic Grothendieck topoi and localic geometric morphisms between them.  
\end{definition}

\begin{remark}
In particular, observe that a Grothendieck topos is localic if and only if its terminal map $ !_\mathcal{E} : \mathcal{E} \rightarrow \Set$ is localic. For any topos, the global section functor $ \Gamma_\mathcal{E} = !_{\mathcal{E}*}$ sends the subobject classifier $\Omega_\mathcal{E}$ to a locale in $\Set$ which we also denote in the same way. 
\end{remark}

The following standard fact can be exhibited as a corollary of \cref{Right objects are relflective} for localic topoi are the right objects for the hyperconnected-localic bifactorization system:

\begin{corollary}
We have a (1,2)-full, bireflective, sub-bicategory 
\[\begin{tikzcd}
	{\textup{\textbf{LocGTop}}} && \GTop
	\arrow[""{name=0, anchor=center, inner sep=0}, bend right=30, hook, from=1-1, to=1-3]
	\arrow[""{name=1, anchor=center, inner sep=0}, "{\textbf{L}}"', bend right=30, from=1-3, to=1-1]
	\arrow["\dashv"{anchor=center, rotate=-90}, draw=none, from=1, to=0]
\end{tikzcd}\]
where the pseudofunctor $ \textbf{L}$, known as the \emph{localic reflection}, sends a Grothendieck topos $ \mathcal{E} $ on the localic topos $ \Sh(\Omega_\mathcal{E})$. This generalizes, for any Grothendieck topos $ \mathcal{E}$, as (1,2)-full bireflective sub-bicategory 
\[\begin{tikzcd}
	{\textup{\textbf{Loc}}/\mathcal{E}} && \GTop/\mathcal{E}
	\arrow["\iota_\mathcal{E}"', ""{name=0, anchor=center, inner sep=0}, bend right=30, hook, from=1-1, to=1-3]
	\arrow[""{name=1, anchor=center, inner sep=0}, "{\textbf{L}_\mathcal{E}}"', bend right=30, from=1-3, to=1-1]
	\arrow["\dashv"{anchor=center, rotate=-90}, draw=none, from=1, to=0]
\end{tikzcd}\]
where the pseudofunctor $ \textbf{L}_\mathcal{E}$, known as the \emph{localic reflection}, sends a geometric morphism $ f : \mathcal{F} \rightarrow \mathcal{E} $ on the localic geometric morphism $ \Sh_\mathcal{E}(f_*(\Omega_\mathcal{F}))$.
\end{corollary}

Now let us look to left objects:

\begin{definition}
A Grothendieck topos is said to be \emph{hyperconnected} if its terminal map $ !_\mathcal{E} : \mathcal{E} \rightarrow \Set$ is hyperconnected. 
\end{definition}

\begin{proposition}
A Grothendieck topos is hyperconnected if and only if its localic reflection is $ \Set$. 
\end{proposition}

\begin{remark}
Beware that we do not claim however that any geometric morphism which is sent to an equivalence the localic reflection is hyperconnected. We do not know how to characterize the class of geometric morphism that are localized by the localic reflection - as characterizing the morphisms that are localized by a reflection can be a difficult problem even in the 1-dimensional case.  
\end{remark}

Now by the gliding property of left objects, we know that a Grothendieck topos which is the codomain of a hyperconnected geometric morphism with hyperconnected domain is itself hyperconnected.

\begin{proposition}
We have a bistable inclusion  
\[\begin{tikzcd}
	{({\textbf{\textup{HypGTop}}^{\textbf{\textup{Hyp}}}})^{\op}} & {\GTop^{\op}}
	\arrow[hook, from=1-1, to=1-2]
\end{tikzcd}\]
of the dual of the inclusion of the bicategory of hyperconnected Grothendieck topoi and hyperconnected geometric morphisms between them is dual to a bistable pseudofunctor
\end{proposition}

This suggests that hyperconnected topoi can also play the role of focal components of a geometry.

\begin{proposition}
The hyperconnected-localic factorization system has oplax right truncated factorization of 2-cells.
\end{proposition}

\begin{proof}
Let be a globular 2-cell $ \sigma : f_1 \Rightarrow f_2$ in $\GTop$: this corresponds to a pair of adjoint 2-cell between the inverse and direct image parts, respectively $ \sigma^\flat : f_1^* \Rightarrow f^*_2$ and $ \sigma^\sharp : f_{2*} \Rightarrow f_{1*}$; in particular the direct image part $ \sigma^\sharp $ lifts to the induced direct image part functors for internal locales, producing a morphism $ \sigma^\sharp_{\Omega_\mathcal{E}} : f_{2*}(\Omega_\mathcal{E}) \rightarrow f_{1*}(\Omega_\mathcal{E})$ in $\textbf{Loc}[\mathcal{F}]$, which is sent by the localic reflection to a morphism in the pseudoslice over $\mathcal{F}$ 
\[\begin{tikzcd}
	{\Sh(f_{2*}(\Omega_\mathcal{E}))} && {\Sh(f_{1*}(\Omega_\mathcal{E}))} \\
	& {\mathcal{F}}
	\arrow["{l_{f_2}}"', from=1-1, to=2-2]
	\arrow["{l_{f_1}}", from=1-3, to=2-2]
	\arrow[""{name=0, anchor=center, inner sep=0}, "{\Sh(\sigma^\sharp_{\Omega_\mathcal{E}})}", from=1-1, to=1-3]
	\arrow["\alpha_{\sigma^\sharp_{\Omega_{\mathcal{E}}}} \atop \simeq"{description}, Rightarrow, draw=none, from=0, to=2-2]
\end{tikzcd}\]
\end{proof}

\subsection{Connected-Etale bifactorization of locally connected geometric morphisms}

Here we investigate a restricted version of what should actually be a more general geometry related to the construction of spectra. 

\begin{definition}
A geometric morphism $ f : \mathcal{E} \rightarrow \mathcal{F}$ is said to be \emph{locally connected} if for any $ F$ in $\mathcal{F}$ the slice of the inverse image part $ (f/A)^* : \mathcal{F}/F \rightarrow \mathcal{E}/f^*F$ is cartesian closed. Equivalently, this means that $f^* $ has a further left adjoint $f_!$ defining a $\mathcal{F}$-indexed functor. A Grothendieck topos $\mathcal{E}$ is \emph{locally connected} if its terminal map is locally connected. In the following we denote as $ \textbf{LocoGTop}$ the 2-full sub-bicategory of locally connected Grothendieck topoi and locally connected geometric morphisms. 
\end{definition}

\begin{remark}
It can be shown that a Grothendieck topos is locally connected if its terminal map is essential. This condition is also equivalent to require that any object $ E$ is a coproduct of connected objects indexed by the set $f_!(E)$ of connected components. 
\end{remark}

\begin{definition}
A geometric morphism $f : \mathcal{E} \rightarrow \mathcal{F}$ is said to be \begin{itemize}
    \item \emph{connected} if $ f^*$ is full and faithful
    \item \emph{etale} if there exists an object $ F$ in $\mathcal{F} $ such that one has an invertible 2-cell
\[\begin{tikzcd}
	{\mathcal{E}} && {\mathcal{F}/F} \\
	& {\mathcal{F}}
	\arrow["f"', from=1-1, to=2-2]
	\arrow[""{name=0, anchor=center, inner sep=0}, "\simeq", from=1-1, to=1-3]
	\arrow["{\pi_F}", from=1-3, to=2-2]
	\arrow["\simeq"{description}, Rightarrow, draw=none, from=0, to=2-2]
\end{tikzcd}\]
\end{itemize}
\end{definition}

\begin{remark}
As localic morphisms were indexed by internal locales, etale geometric morphisms are indexed by objects, that is, we have an equivalence of categories 
\[ \textbf{Et}/\mathcal{E} \simeq \mathcal{E} \]
\end{remark}

\begin{remark}
In fact etale geometric morphisms are locally connected; however beware that not all connected geometric morphisms are locally connected. 
\end{remark}

\begin{proposition}
Connected geometric morphisms are left biorthogonal to etale geometric morphisms.
\end{proposition}

\begin{remark}
Beware that not any geometric morphism that is right biorthogonal to connected geometric morphism has to be etale: for instance, 2-cofiltered limits of etale geometric morphisms are so, though they are not etale, not even locally connected.  
\end{remark}

It is well known that the bicategory of locally connected topoi and geometric morphisms enjoy the following factorization system:

\begin{proposition}

Any locally connected geometric morphism factorizes uniquely as a connected, locally connected geometric morphism followed by an etale geometric morphism as follows:
\[\begin{tikzcd}
	{\mathcal{E}} && {\mathcal{F}} \\
	& {\mathcal{F}/f_!(1)}
	\arrow["{c_f}"', from=1-1, to=2-2]
	\arrow[""{name=0, anchor=center, inner sep=0}, "f", from=1-1, to=1-3]
	\arrow["{\pi_{f!(1)}}"', from=2-2, to=1-3]
	\arrow["\simeq"{description}, Rightarrow, draw=none, from=0, to=2-2]
\end{tikzcd}\]
where $ (c_f)_!$ sends $ E $ to the image of the terminal map $ f_!(!_E) : f_!(E) \rightarrow f_!(1)$.
\end{proposition}

\begin{proposition}
The connected-etale factorization on $\textup{\textbf{LocoGTop}}$ has pseudofunctorial oplax right truncated factorization of 2-cells.
\end{proposition}

\begin{proof}
Let be $ \sigma : f_1 \Rightarrow f_2$ a 2-cell between locally connected geometric morphisms. Then the inverse image part $ \sigma^\flat : f_1^* \Rightarrow f_2^*$ has a mate $ \sigma_! : f_{2!} \Rightarrow f_{1*}$, whose component $ \sigma_{!1_\mathcal{E}} : f_{2!}(1_\mathcal{E}) \rightarrow f_{1!}(1_\mathcal{E})$ at $ 1_\mathcal{E}$ is sent to an invertible 2-cell between etale geometric morphism over $\mathcal{F}$
\[\begin{tikzcd}
	{\mathcal{F}/f_{2!}(1_\mathcal{E})} && {\mathcal{F}/f_{1!}(1_\mathcal{E})} \\
	& {\mathcal{F}}
	\arrow["{\pi_{f_{2!}(1_\mathcal{E})}}"', from=1-1, to=2-2]
	\arrow[""{name=0, anchor=center, inner sep=0}, "{\mathcal{F}/\sigma_{!1_\mathcal{E}}}", from=1-1, to=1-3]
	\arrow["{\pi_{f_{1!}(1_\mathcal{E})}}", from=1-3, to=2-2]
	\arrow["\simeq"{description}, Rightarrow, draw=none, from=0, to=2-2]
\end{tikzcd}\]
where the intermediate geometric morphism has postcomposition with $ f_!(!_E)$ as left adjoint to the inverse image part. Now we have by naturality of $\sigma_!$ a square in $\mathcal{F}$ for each $E$ in $\mathcal{E}$
\[\begin{tikzcd}
	{f_{2!}(E)} & {f_{2!}(1_\mathcal{E})} \\
	{f_{1!}(E)} & {f_{1!}(1_\mathcal{E})}
	\arrow["{\sigma_{!E}}"', from=1-1, to=2-1]
	\arrow["{\sigma_{!1_{\mathcal{E}}}}", from=1-2, to=2-2]
	\arrow["{f_{2!}(!_E)}", from=1-1, to=1-2]
	\arrow["{f_{1!}(!_E)}"', from=2-1, to=2-2]
\end{tikzcd}\]
which provides us with a 2-cell 
\[\begin{tikzcd}
	& {\mathcal{E}} \\
	{\mathcal{F}/f_{2!}(1_\mathcal{E})} && {\mathcal{F}/f_{1!}(1_\mathcal{E})}
	\arrow["{c_{f_2}}"', from=1-2, to=2-1]
	\arrow[""{name=0, anchor=center, inner sep=0}, "{\mathcal{F}/\sigma_{!1_\mathcal{E}}}"', from=2-1, to=2-3]
	\arrow["{c_{f_1}}", from=1-2, to=2-3]
	\arrow["{\lambda_\sigma \atop \Leftarrow}"{description}, Rightarrow, draw=none, from=0, to=1-2]
\end{tikzcd}\]
Those 2-cells can be shown to provide with a decomposition of $\sigma$, by observing that their adjoint part provide such an adjoint decomposition of $\sigma_!$. Similarly, pseudofunctoriality processes from pseudofunctoriality of the externalization functor.
\end{proof}

\begin{proposition}
We have a bistable inclusion 
\[\begin{tikzcd}
	{(\textup{\textbf{CoLocoGTop}}^{\textup{\textbf{Co}}})^{\op}} & {\textup{\textbf{Loco}}\GTop^{\op}}
	\arrow[hook, from=1-1, to=1-2]
\end{tikzcd}\]
of the dual of the inclusion of the bicategory of connected, locally connected topoi and connected, locally connected geometric morphisms between them into the bicategory of locally connected Grothendieck topoi and locally connected geometric morphisms between them.   
\end{proposition}

\begin{remark}
Beware that we must restrict also the objects to locally connected topoi in order to preserve terminalness of $\Set$. 
\end{remark}

However, connected geometric morphisms are not stable under bipullback, and cannot be seen as a notion of focal objects as lacking a focal point. However, there are two classes of geometric morphisms that are known to have right cancellation relatively to connected geometric morphisms: local and totally connected ones.

\subsection{Terminally connected-Etale bifactorization of essential geometric morphisms}

A generalization of the previous bifactorization system was recently described in \cite{caramello2020denseness}[4.7], where one relax the condition of local connectedness to essentialness: there one can still use the further left adjoint to the inverse image part to produce a factorization, but the left maps are in a wider class generalizing connected geometric morphisms. 

\begin{definition}
A geometric morphism $ f : \mathcal{E} \rightarrow \mathcal{F}$ is \emph{essential} if $f^*$ has a further left adjoint $ f_!$. 
\end{definition}

\begin{remark}
One can show that the terminal map of a Grothendieck topos $!_\mathcal{E} : \mathcal{E} \rightarrow \Set$ is automatically locally connected whenever it is essential: hence over set, there is no distinction between locally connected topoi and essential topoi. 
\end{remark}

\begin{definition}
A geometric morphism $f : \mathcal{E} \rightarrow \mathcal{F}$ is sais to be \emph{terminally connected} if its inverse image $ f^*$ lifts uniquely global elements, that is, if one has for each $ F$ a natural equivalence 
\[  \mathcal{E}[1_\mathcal{E}, f^*(F)] \simeq \mathcal{F}[1_\mathcal{F}, F]  \]
\end{definition}

The following observation is immediate by adjunction:

\begin{proposition}
An essential geometric morphism is terminally connected if and only if one has $ f_!(1_\mathcal{E}) \simeq 1_\mathcal{F}$. 
\end{proposition}

A connected geometric morphism is terminally connected ; beware that despite the terminology, terminal connectedness is a weaker condition that connectedness, as it requires only full faithfulness relatively to global elements and not arbitrary generalized elements. \\

\begin{remark}
Now we should give a few remarks about etale geometric morphisms. Being locally connected, they are always essential, with the further left adjoint of the projection $ \pi_{E} : \mathcal{E}/E \rightarrow \mathcal{E}$ given by postcomposition with the terminal map $ !_E : E \rightarrow 1_\mathcal{E}$. Moreover, we should give here an interpretation of what encodes a factorization through an etale geometric morphism on the right. Suppose we have an invertible 2-cell
\[\begin{tikzcd}
	{\mathcal{F}} && {\mathcal{E}} \\
	& {\mathcal{E}/E}
	\arrow["a"', from=1-1, to=2-2]
	\arrow["{\pi_E}"', from=2-2, to=1-3]
	\arrow[""{name=0, anchor=center, inner sep=0}, "f", from=1-1, to=1-3]
	\arrow["\simeq"{description}, Rightarrow, draw=none, from=0, to=2-2]
\end{tikzcd}\]
Then we have a factorization of $a$ through the base change 
\[\begin{tikzcd}
	{\mathcal{F}} \\
	& {\mathcal{F}/f^*E} & {\mathcal{E}/E} \\
	& {\mathcal{F}} & {\mathcal{E}}
	\arrow["{\pi_E}", from=2-3, to=3-3]
	\arrow["f"', from=3-2, to=3-3]
	\arrow[from=2-2, to=2-3]
	\arrow[curve={height=12pt}, Rightarrow, no head, from=1-1, to=3-2]
	\arrow["{\pi_{f^*E}}"{description}, from=2-2, to=3-2]
	\arrow["\lrcorner"{anchor=center, pos=0.125}, draw=none, from=2-2, to=3-3]
	\arrow["a"{description}, curve={height=-12pt}, from=1-1, to=2-3]
	\arrow[dashed, from=1-1, to=2-2]
\end{tikzcd}\]
which also defines a global section of $ \pi_{f^*E}$, which is the name of a global element $ a : 1_\mathcal{F} \rightarrow f^*E$: that is 
\[ \GTop/\mathcal{E}[f,\pi_E] \simeq \mathcal{F}[1_\mathcal{F},f^*E] \]
\end{remark}

\begin{proposition}
Terminally connected geometric morphisms are left orthogonal to etale geometric morphisms.
\end{proposition}

\begin{proof}
Let be an invertible square as below
\[\begin{tikzcd}
	{\mathcal{G}} & {\mathcal{E}/E} \\
	{\mathcal{F}} & {\mathcal{E}}
	\arrow["t"', from=1-1, to=2-1]
	\arrow["f"', from=2-1, to=2-2]
	\arrow["a", from=1-1, to=1-2]
	\arrow["{\pi_{E}}", from=1-2, to=2-2]
	\arrow["\simeq"{description}, draw=none, from=1-1, to=2-2]
\end{tikzcd}\]
corresponding to a global element $ a : 1_\mathcal{G} \rightarrow t^*f^*E  $ in $\mathcal{G}$. If $ t$ is terminally connected, then this element comes uniquely from a global element $ \overline{a} : 1_\mathcal{F} \rightarrow f^*E$ in $\mathcal{F}$, whose name is a diagonalization of the pseudosquare above $ \mathcal{F} \rightarrow \mathcal{E}/E$. As any other possible diagonalization would name another antecedent of $a$, this diagonalization has to be unique. The converse is exactly the converse reasoning. 
\end{proof}

Then the connected-etale factorization of locally connected geometric morphisms generalizes to essential geometric morphisms as below - this is \cite{caramello2020denseness}[Proposition 4.62], where the bifactorization is obtained exactly in the same way:

\begin{lemma}
Any essential geometric morphism factorizes uniquely as a terminally connected geometric morphism followed by an etale geometric morphism: that is, we have a bifactorization system $ (\textbf{TCoEss}, \textbf{Et})$ on $\GTop^{\textbf{Ess}}$.
\end{lemma}


\begin{proposition}
We have a bistable inclusion
\[\begin{tikzcd}
	{(\textup{\textbf{CoLocoGTop}}^{\textup{\textbf{TCo}}})^{\op}} & {(\textup{\textbf{Loco}}\GTop^{\textup{\textbf{Ess}}})^{\op}}
	\arrow[hook, from=1-1, to=1-2]
\end{tikzcd}\]
of the dual of the inclusion of the bicategory of connected, locally connected Grothendieck topoi and terminally connected essential geometric morphisms between them into the bicategory of locally connected Grothendieck topoi and essential geometric morphisms between them.
\end{proposition}

\begin{remark}
Again, we must restrict to locally connected topoi to ensure terminalness of $\Set$, though the bifactorization still exists if we only restrict to essential geometric morphisms. 
\end{remark}

\subsection{Local topoi as specified left objects}

In the connected-etale and its generalization, the left map encode something a bit ``coarse"; for connected - nor terminally connected - geometric morphisms are stable under bipullback, one might not want them to be local objects; however we can specify two kinds of left objects: \emph{local} and \emph{totally connected topoi}, the first specification being of special interest. 

\begin{definition}
A geometric morphism $ f : \mathcal{F} \rightarrow \mathcal{E}$ is said to be \emph{local}\index{geometric morphism!local} if its inverse image part $f^*$ is full and faithful and moreover the direct image part $ f_*$ has a further right adjoint $ f^!$. \\

In particular, a Grothendieck topos $\mathcal{E}$ is said to be \emph{local}\index{topos!local} if the global section functor $ !_{\mathcal{E}*} : \mathcal{E} \rightarrow \mathcal{S}$ has a further right adjoint.
\end{definition}

\begin{remark}
Observe that, for a Grothendieck topos to be local, the full-and-faithfulness condition is automatic as soon as the existence of the further left adjoint is established: indeed, for $\Set$ is the terminal topos, the pair $ \gamma_{\mathcal{E}*} \dashv !_{\mathcal{E}}^!$ must form a retract of $ !_\mathcal{E}$ in the bicategory of Grothendieck topoi: hence we must automatically have natural isomorphisms $ !_{\mathcal{E}*}!_{\mathcal{E}}^* \simeq \textup{id}_{\Set}$ and $ !_{\mathcal{E}*}!_{\mathcal{E}}^! \simeq \textup{id}_{\Set} \simeq !_{\mathcal{E}*}!_{\mathcal{E}}^*$, which force both $ !_{\mathcal{E}}^*$ and $!_{\mathcal{E}}^!$ to be full and faithful.
\end{remark}

Now we have the following property over $\Set$ expressing why local topoi can serve as a specification of left objects:

\begin{lemma}
Local topoi have the gliding property along terminally connected geometric morphisms. 
\end{lemma}

\begin{proof}
Suppose we have a triangle as below in $\GTop$
\[\begin{tikzcd}
	{\mathcal{L}} && {\mathcal{E}} \\
	& \Set
	\arrow["{!_\mathcal{L}}"', from=1-1, to=2-2]
	\arrow[""{name=0, anchor=center, inner sep=0}, "t", from=1-1, to=1-3]
	\arrow["{!_{\mathcal{E}}}", from=1-3, to=2-2]
	\arrow["\simeq"{description}, Rightarrow, draw=none, from=0, to=2-2]
\end{tikzcd}\]
where $ !_\mathcal{L}$ is local and $t$ is terminally connected. By what precedes, it suffices to prove that $ !_{\mathcal{E}*}$ has a further right adjoint. Recall that the direct image part of the terminal map is the global section functor returning the set of global elements. But as $t$ is terminally connected one has for any $E$ in $\mathcal{E}$ 
\begin{align*}
    !_{\mathcal{L}*}(t^*E) &\simeq \mathcal{L}[1_\mathcal{L}, t^*E] \\
    &\simeq \mathcal{E}[1_\mathcal{E}, E] \\
    &\simeq !_{\mathcal{E}*}(E)
\end{align*}
Hence for any $E$ in $\mathcal{E}$ and any $X$ in $\Set$ we have the following isomorphisms of homsets
\begin{align*}
    \Set[!_{\mathcal{E}*}(E), X] &\simeq \Set[!_{\mathcal{L}*}(t^*E), X] \\
    &\simeq \mathcal{L}[t^*E, !_{\mathcal{L}}^!(X)] \\
    &\simeq \mathcal{E}[E, t_*!_{\mathcal{L}}^!(X)]
\end{align*}
Hence one can choose the composite $t_*!_{\mathcal{L}}^!$ as a further right adjoint to $ !_{\mathcal{E}*}$. As $!_\mathcal{E}$ is moreover connected by what precedes, it is actually local. 
\end{proof}

\begin{proposition}
We have a bistable inclusion
\[\begin{tikzcd}
	{(\textup{\textbf{LocoFocGTop}}^{\textup{\textbf{TCoEss}}})^{\op}} & {(\textup{\textbf{Loco}}\GTop^{\textup{\textbf{Ess}}})^{\op}}
	\arrow[hook, from=1-1, to=1-2]
\end{tikzcd}\]
of the dual of the inclusion of the bicategory of locally connected local Grothendieck topoi and terminally connected essential geometric morphisms between them into the bicategory of locally connected Grothendieck topoi and essential geometric morphisms between them.\\

This restricts to a bistable inclusion 
\[\begin{tikzcd}
	{(\textup{\textbf{LocoFocGTop}}^{\textup{\textbf{Co}}})^{\op}} & {(\textup{\textbf{Loco}}\GTop)^{\op}}
	\arrow[hook, from=1-1, to=1-2]
\end{tikzcd}\]
of the dual of the inclusion of the bicategory of local, locally connected Grothendieck topoi and connected, locally connected geometric morphisms between them into the bicategory of locally connected Grothendieck topoi and locally connected geometric morphisms.
\end{proposition}

\begin{remark}
Generic morphisms for this bistable inclusion are duals of the etale geometric morphisms $\pi_E : \mathcal{E}/E \rightarrow \mathcal{E}$ such that $ \mathcal{E}/E$ is local. One can characterize objects $E$ in $\mathcal{E}$ with this properties as those that are \emph{connected projective}. Recall that $E$ is \emph{projective} if for any epimorphism $ e : F \twoheadrightarrow G$ and $ a :E \rightarrow G$ one has a lift 
\[\begin{tikzcd}
	& F \\
	E & G
	\arrow["e",two heads, from=1-2, to=2-2]
	\arrow["a"', from=2-1, to=2-2]
	\arrow[dashed, from=2-1, to=1-2]
\end{tikzcd}\]
so that $\mathcal{E}[E,e]$ is epic, while $E$ is \emph{connected} if for any family $(E_i)_{i \in I}$ one has 
\[ \mathcal{E}[ E, \coprod_{i \in I}E_i] \simeq \coprod_{i \in I} \mathcal{E}[E,E_i]  \]
If $ E$ is then both connected and projective, then the hom functor $ \mathcal{E}[E,-] : \mathcal{E} \rightarrow \Set$ preserves colimits. But now for colimit in $\mathcal{E}/E$ are computed in $\mathcal{E}$, we have that $ 1_\mathcal{E}$ is connected-projective in $\mathcal{E}/E$. Then $ \mathcal{E}/E$ is local.
\end{remark}




The prototypical example of local geometric morphisms is the universal domain map $\partial_0 : \mathcal{E}^2 \rightarrow \mathcal{E}$ of a topos. Now for any point $ p : \mathcal{S} \rightarrow \mathcal{E}$, we can consider the \emph{Grothendieck Verdier localization at $p$}\index{Grothendieck-Verdier localization}, which is defined as the pseudopullback\[\begin{tikzcd}
	{\mathcal{E}_p} & {\mathcal{E}^2} \\
	{\mathcal{S}} & {\mathcal{E}}
	\arrow["{\partial_0}", from=1-2, to=2-2]
	\arrow["p"', from=2-1, to=2-2]
	\arrow["{p^*\partial_0}"', from=1-1, to=2-1]
	\arrow[from=1-1, to=1-2]
	\arrow["\lrcorner"{anchor=center, pos=0.125}, draw=none, from=1-1, to=2-2]
\end{tikzcd}\]
Its universal property is that for any Grotendieck topos $ \mathcal{F}$, we have an equivalence with the cocomma category
\[ \GTop[\mathcal{F}, \mathcal{E}_p] \simeq p \; !_\mathcal{F} \! \downarrow \GTop[\mathcal{F}, \mathcal{E}] \]
In particular, when $ \mathcal{F}$ is $\mathcal{S}$, we have an equivalence of category
\[ \pt(\mathcal{E}_p) \simeq p\downarrow \pt(\mathcal{E})  \]
From \cite{localmap}[Theorem 3.7] we know that if $ \mathcal{E}$ has $ (\mathcal{C},J)$ as a lex site of definition, then $ \mathcal{E}_p$ can be expressed as a cofiltered pseudolimit of etales geometric morphism 
\[ \mathcal{E}_p \simeq \underset{(C,a) \in \int p^*}{\bilim} \mathcal{E}/\mathfrak{a}_J\hirayo^*_C  \]
where $ \int p^*$ is the cofiltered category of elements of the $J$-flat functor $ p^*: \mathcal{C} \rightarrow \mathcal{S}$. \\

Now we discuss how the terminally connected-etale factorization generalizes to the bicategory of Grothendieck topoi and all geometric morphisms between them. While the left class will be unchanged, the right class has to be generalized; though we do not know the best characterization of it, we can still tell something about.

\begin{definition}
    A geometric morphism will be called \emph{pro-etale} if it decomposes as a cofiltered bilimit of a etale geometric morphisms over its codomain. 
\end{definition}

\begin{remark}
Pro-etale geometric morphisms are those of the form 
\[\begin{tikzcd}
	{\underset{i \in I}{\bilim} \, \mathcal{E}/E_i} && {\mathcal{E}/E_i} \\
	& {\mathcal{E}}
	\arrow[""{name=0, anchor=center, inner sep=0}, "{p_i}", from=1-1, to=1-3]
	\arrow["{\underset{i \in I}{\bilim} \, \pi_{E_i}}"', from=1-1, to=2-2]
	\arrow["{\pi_{E_i}}", from=1-3, to=2-2]
	\arrow["{\alpha \atop \simeq}"{description}, Rightarrow, draw=none, from=0, to=2-2]
\end{tikzcd}\]
with $ I$ a small cofiltered category. Observe that the induced map $ \bilim_{i \in I}f_i$ really is the bilimit in $\GTop/\mathcal{E}$ for the domain projection $\GTop/\mathcal{E} \rightarrow \GTop$ must preserve cofiltered bilimits. 
\end{remark}

\begin{remark}\label{cofiltered bilimt of etale geometric morphisms}
From \cite{dubuc2011construction}[Theorem 2.4 and 2.5] we know that a cofiltered bilimit in $\GTop$ of a diagram of geometric morphisms induced by morphisms of sites between them is computed as the corresponding filtered bicolimits of the underlying sites, that is, the filtered pseudocolimit of the underlying categories with the induced topology. If $ \mathcal{E}$ has $(\mathcal{C},J)$ as a small standard site of definition, then for each object $ C$ of $\mathcal{C}$, a site for $\mathcal{E}/\hirayo_C$ can be constructed by $ (\mathcal{C}/C, J_C)$ where $ J_C$ is the relativised topology. 
Suppose now that the objects $E_i$ are chosen as arising from representable $ \hirayo_{C_i}$; for etale geometric morphisms have left cancellation, all the transition morphisms in the diagram must be etale, but then they correspond to morphisms in $\mathcal{E}$ between the corresponding object, which, by full faithfulness of the Yoneda embedding, come uniquely from morphisms $ C_i \rightarrow C_j$ in $\mathcal{C}$. Then a the cofiltered bilimit $ \bilim_{i \in I} \mathcal{E}/E_i$ has the following presentation
\[ \underset{i \in I}{\bilim} \, \mathcal{E}/E_i \simeq \Sh(\underset{i \in I}{\bicolim} \, \mathcal{C}/C_i, \langle \iota_i (J_{C_i}) \rangle )  \]
where $\bicolim_{i \in I} \mathcal{C}/C_i $ is computed in $\Cat$ where one can chose the pseudocolimit, and $\langle \iota_i (J_{C_i}) \rangle$ is the Grothendieck topology induced by the bicolimit inclusions.  
\end{remark}

The following fact is an immediate consequence of closure properties of right classes under limits in the pseudoslices:
\begin{lemma}
Pro-etale geometric morphisms are right biorthogonal to terminally connected geometric morphisms. 
\end{lemma}

\begin{theorem}\label{terminally connected-proetale factorization}
Any geometric morphism factorizes in an essentially unique way as a terminally connected geometric morphism followed by a pro-etale geometric morphism.
\end{theorem}

\begin{proof}
Let be $ f : \mathcal{E} \rightarrow \mathcal{F}$ be a geometric morphism and $ (\mathcal{C},J)$ a small standard site for $\mathcal{F}$. Then consider the category of global elements of the inverse image part $ f^*$ restricted along the Yoneda embedding $ \mathcal{C} \hookrightarrow \mathcal{F}$ (which we denote as $f^*$), that is, the comma category $ 1_\mathcal{E}\downarrow f^*$. Then we have an equivalence of categories
\[  1_\mathcal{E} \downarrow f^* \simeq f \downarrow \textbf{Et}/\mathcal{E}  \]
(where the latter is a 1-category for etale geometric morphisms are discrete) as global elements of $f^*$ are in correspondence with factorization of $f^*$ along etale geometric morphisms. Moreover, $f^*$ defines a representably flat functor $ f^* : \mathcal{C} \rightarrow \mathcal{E}$ so in particular $1_\mathcal{E} \downarrow f^*$ is cofiltered. Hence we have a unique factorization as below
\[\begin{tikzcd}
	{\mathcal{E}} && {\mathcal{F}} \\
	& {\underset{(C,a) \in 1_{\mathcal{E}}\downarrow f^*}{\bilim} \, \mathcal{F}/\hirayo_C}
	\arrow["{l_c}"', from=1-1, to=2-2]
	\arrow["{\underset{(C,a) \in 1_{\mathcal{E}}\downarrow f^*}{\bilim} \, \pi_{\hirayo_C}}"', from=2-2, to=1-3]
	\arrow[""{name=0, anchor=center, inner sep=0}, "f", from=1-1, to=1-3]
	\arrow["{\alpha_f \atop \simeq}"{description}, Rightarrow, draw=none, from=2-2, to=0]
\end{tikzcd}\]
with $\bilim_{(C,a) \in 1_\mathcal{E} \downarrow f^*} \pi_{\hirayo_C}$ pro-etale, and moreover this factorization is initial amongst all factorization of $f$ through a pro-etale geometric morphism on the right. \\

Now we must prove $ l_f$ to be terminally connected. From \cref{cofiltered bilimt of etale geometric morphisms}, as $ 1_\mathcal{E}\downarrow f^*$ we know $\bilim_{(C,a) \in 1_\mathcal{E} \downarrow f^*} \mathcal{F}/\hirayo_C$ to have $ (\bicolim_{(C,a) \in 1_\mathcal{E} \downarrow f^*} \mathcal{C}/C_i, \langle \iota_{(C,a)} (J_C) \rangle_{(C,a) \in 1_\mathcal{E} \downarrow f^*}) $ as presentation site, and the inverse image part $ l_f^*$ is the map induced by the universal property of the bicolimit as below
\[\begin{tikzcd}
	{\mathcal{E}} && {\mathcal{C}} \\
	{\underset{(C,a) \in 1_\mathcal{E} \downarrow f^*}{\pscolim} \, \mathcal{C}/C} && {\mathcal{C}/C}
	\arrow["{l_f^*}", from=2-1, to=1-1]
	\arrow["{f^*}"', from=1-3, to=1-1]
	\arrow["{\pi_C^*}", from=1-3, to=2-3]
	\arrow["{\iota_{(C,a)}}", hook', from=2-3, to=2-1]
	\arrow[draw=none, from=2-3, to=1-1]
	\arrow["{a^*}"{description}, from=2-3, to=1-1]
\end{tikzcd}\]
where $ a^*$ is the inverse image part of the name $ \mathcal{E} \rightarrow \mathcal{F}/\hirayo_C$ of the global element $ a: 1_\mathcal{E} \rightarrow f^*(C)$. But recall that objects of $ \pscolim_{(C,a) \in 1_\mathcal{E} \downarrow f^*} \mathcal{C}/C$ is a pair $ ((C,a), h)$ with $ {(C,a) \in 1_\mathcal{E} \downarrow f^*}$ and $ h : D \rightarrow C$ an object of $ \mathcal{C}/C$, and as the map induced by the bicolimit property, one has $ l_f^*((C,a),h) = a^*(h)$.\\

Let us describe the action of $a^*$ in term of global sections, that is of the functor $ \mathcal{E}[1_\mathcal{E}, a^*]$: this latter sends an object $ h : D \rightarrow C$ of the slice $ \mathcal{C}/C$ to the set of global elements $ b : 1_\mathcal{E} \rightarrow f^*(D)$ such that we have a factorization 
\[\begin{tikzcd}
	{1_\mathcal{E}} & {f^*(D)} \\
	& {f^*(C)}
	\arrow["b", from=1-1, to=1-2]
	\arrow["a"', from=1-1, to=2-2]
	\arrow["{f^*(h)}", from=1-2, to=2-2]
\end{tikzcd}\]

But from $ l_f^*((C,a),h) = a^*(h)$, we know such a factorization is exactly what a global element $ b : 1 \rightarrow l_f^*((C,a),h) $ codes for. We want to prove that any such global element comes from a unique global element in the pseudocolimit. \\

We must first understand how are global elements in $ \pscolim_{(C,a) \in 1_\mathcal{E} \downarrow f^*} \mathcal{C}/C$, which require a description of its terminal object. The terminal object in the oplax colimit $\oplaxcolim_{(C,a) \in 1_\mathcal{E} \downarrow f^*} \mathcal{C}/C $ is the triple $((1_\mathcal{C},1_\mathcal{E}), 1_{1_\mathcal{C}})$. Now recall that the pseudocolimit in $\Cat$ is obtained as the localization of the oplax colimit at the cartesian morphisms. Hence here we know that for any $ (D,b)$ in $1_\mathcal{E}\downarrow f^*$, the triple $ ((D,b), 1_D)$ is a representant for the terminal object as one always has the commutation 
\[\begin{tikzcd}
	{1_\mathcal{E}} & {f^*(D)} \\
	& {f^*(1_\mathcal{C})}
	\arrow["b", from=1-1, to=1-2]
	\arrow["{!}", from=1-2, to=2-2]
	\arrow[Rightarrow, no head, from=1-1, to=2-2]
\end{tikzcd}\]
for $ f^*$ preserves the terminal element and $ f^*(1_\mathcal{C})$ has only one global element, while the vertical  is related by a pullback square
\[\begin{tikzcd}
	D & {1_\mathcal{C}} \\
	D & {1_\mathcal{C}}
	\arrow[Rightarrow, no head, from=1-2, to=2-2]
	\arrow[Rightarrow, no head, from=1-1, to=2-1]
	\arrow["{!}"', from=2-1, to=2-2]
	\arrow["{!}", from=1-1, to=1-2]
	\arrow["\lrcorner"{anchor=center, pos=0.125}, draw=none, from=1-1, to=2-2]
\end{tikzcd}\]
which codes for a cartesian morphism $ ((D,b), 1_D) \rightarrow ((1_\mathcal{C}, 1_\mathcal{E}), 1_{1_\mathcal{C}})$ identifying them in the pseudocolimit. \\

But then we see that a global element $1 \rightarrow l_f^*((C,a),h) $ defines uniquely a morphism 
\[\begin{tikzcd}
	{((D,b), 1_D)} & {((C,a),h)}
	\arrow["{(h, 1_D)}", from=1-1, to=1-2]
\end{tikzcd}\]
in the oplax colimit, which induces a global element 
\[\begin{tikzcd}
	{((1_\mathcal{C},1_\mathcal{E}), 1_{1_\mathcal{C}})} & {((C,a),h)}
	\arrow["{[(h, 1_D)]_\sim}", from=1-1, to=1-2]
\end{tikzcd}\]
after the localization to the pseudocolimit. This achieves to prove that $ l^*$ is terminally connected.

\end{proof}

Hence, though we cannot yet characterize exactly the right maps, we have an biorthogonal bifactorization in $\GTop$. Moreover, we can use the gliding property of local topoi along terminally connected geometric morphisms to get the following: 

\begin{proposition}
We have a bistable inclusion 
\[\begin{tikzcd}
	{(\textup{\textbf{FocGTop}}^{\textup{\textbf{TCo}}})^{\op}} & {\GTop^{\op}}
	\arrow[hook, from=1-1, to=1-2]
\end{tikzcd}\]
dual to the inclusion of the bicategory of local topoi and terminally connected geometric morphisms between them into the bicategory of Grothendieck topoi.
\end{proposition}

\subsection{Totally connected topoi as specified left objects}

\begin{definition}[Theorem C3.6.16 \cite{elephant}]
	A geometric morphism $f$ is said to be \emph{totally connected} if the following equivalent conditions are fulfilled:\begin{itemize}
		\item $f^*$ has cartesian left adjoint $ f_!$;
		\item $ f$ has a right adjoint in the 2-category of Grothendieck toposes;
		\item $ f$ has a terminal section.
\end{itemize}\end{definition}
Morally, a topos is totally connected if it has a terminal point, dually to local toposes who have an initial point. In particular, the universal codomain morphism $ \partial_1$ of a topos always is totally connected (as the universal domain morphism $ \partial_0$ always is local). Totally connected geometric morphisms are stable under pullback (cf. Lemma C3.6.18 \cite{elephant}). Observe that totally connected geometric morphisms are both connected and locally connected.\\

The following is well known, see \cite{elephant}[C3.6.18]: 

\begin{proposition}
Totally connected geometric morphisms have right cancellation along connected geometric morphisms.
\end{proposition}

\begin{proposition}
We have a bistable inclusion 
\[\begin{tikzcd}
	{(\textup{\textbf{LocoTocGTop}}^{\textup{\textbf{Co}}})^{\op}} & {(\textup{\textbf{Loco}}\GTop)^{\op}}
	\arrow[hook, from=1-1, to=1-2]
\end{tikzcd}\]
of totally connected topoi and connected, locally connected geometric morphisms between them into the bicategory of locally connected Grothendieck topoi and locally connected geometric morphisms. 
\end{proposition}

\begin{remark}
For now it is not clear whether totally connected geometric morphisms, though being themselves terminally connected, have right cancellation along terminally connected geometric morphisms.
\end{remark}

\printbibliography

\end{document}